\documentclass[11pt]{amsart}
\usepackage[T1]{fontenc}
\usepackage[utf8]{inputenc}
\usepackage[english]{babel}
\usepackage{amsmath, amsthm, amssymb, amsfonts, enumerate, mathrsfs, dsfont, comment, bm, mathtools}
\usepackage{algorithm}
\usepackage{algpseudocode}
\usepackage{enumitem}
\usepackage{subcaption}
\usepackage{graphicx}
\usepackage{color, geometry, todonotes, indentfirst, graphicx, pdfpages}
\usepackage{fancyhdr}
\usepackage{cases}
\usepackage{appendix}
\usepackage{float}

\usepackage{etoolbox}
\makeatletter
\def\section{\@startsection{section}{1}%
  \z@{2.5ex plus 1ex minus .2ex}{1.5ex plus .2ex}%
  {\centering\normalfont\Large\scshape}}
\makeatother

\usepackage[colorlinks=true,linkcolor=blue,urlcolor=blue]{hyperref}

\allowdisplaybreaks
\newtheorem{theorem}{Theorem}[section]
\newtheorem{remark}[theorem]{Remark}
\newtheorem{assumption}[theorem]{Assumption}
\newtheorem{lemma}[theorem]{Lemma}
\newtheorem{prop}[theorem]{Proposition}
\newtheorem{corollary}[theorem]{Corollary}
\newtheorem{definition}[theorem]{Definition}

\theoremstyle{plain}

\newcommand{\D}{\mathcal{D}}

\newcommand{\norm}[1]{\left\lVert{#1}\right\rVert}

\DeclareMathAlphabet{\mathbbold}{U}{bbold}{m}{n}

\numberwithin{equation}{section}

\newcommand{\pushright}[1]{\ifmeasuring@#1\else\omit\hfill$\displaystyle#1$\fi\ignorespaces}
\newcommand{\mail}[1]{\href{mailto:#1}{\texttt{#1}}}

\title[]{Optimal Consumption and Portfolio Choice with No-Borrowing Constraint in the Kim-Omberg Model: \\
The Complete Market Case}

\author[G.~Ferrari]{Giorgio~Ferrari\textsuperscript{\MakeLowercase{a},1}}
\author[T. N.~Schütz]{Tim Niclas~Schütz\textsuperscript{\MakeLowercase{a},2}}
\thanks{\noindent\textsuperscript{a} Bielefeld University, Center for Mathematical Economics (IMW), Bielefeld (Germany).}
\thanks{\noindent
\noindent \textsuperscript{1} E-mail: \mail{giorgio.ferrari@uni-bielefeld.de}.
\\
\noindent \textsuperscript{2} E-mail: \mail{tim.schuetz@uni-bielefeld.de}.}

\date{\today}

\numberwithin{equation}{section}

\begin{document}


\begin{abstract}
In this paper, we study an intertemporal utility maximization problem in which an investor chooses consumption and portfolio strategies in the presence of a stochastic factor and a no-borrowing constraint. In the spirit of the Kim-Omberg model, the stochastic factor represents the expected excess return of the risky asset. It is perfectly negatively correlated with shocks to the risky asset, and follows an Ornstein-Uhlenbeck process, thereby capturing the mean reversion of expected excess returns--a feature well supported by empirical evidence in financial markets. The investor seeks to maximize expected utility from consumption, subject to the constraint that wealth remains nonnegative at all times. To address the dynamic no-borrowing constraint, we use Lagrange duality to transform the primal problem into a singular control problem in the dual space. We then characterize the solution to the dual singular control problem via an auxiliary two-dimensional optimal stopping problem featuring stochastic volatility, and subsequently retrieve the primal value function as well as the optimal portfolio and consumption plans. Finally, a numerical study is conducted to derive economic and financial implications.
\end{abstract} 

\maketitle



\section{Introduction}
We study an infinite-horizon optimal consumption and investment problem in which an agent chooses how much to consume and how to allocate wealth in a financial market. The agent receives a constant stream of labor income and faces a no-borrowing constraint, meaning that wealth must remain nonnegative at all times. This rules out borrowing either in financial markets or against future labor income, so that all decisions must be financed by current resources.

Investment opportunities vary over time because the expected excess return of the risky asset is driven by a stochastic factor $(\beta_t)_t$. Empirical evidence (see, e.g., \cite{fama1988dividend} and \cite{poterba1988mean}) suggests that expected excess returns are predictable and mean-reverting. Motivated by this finding, we model $(\beta_t)_t$ as a mean-reverting Ornstein--Uhlenbeck process, following the so-called Kim-Omberg model introduced in \cite{kim1996dynamic}. Furthermore, as in \cite{wachter2002portfolio}, we focus on the case in which the Brownian motion driving $(\beta_t)_t$ is perfectly negatively correlated with the Brownian motion driving the risky asset's returns, a specification that is empirically motivated by the return-predictability literature (see, e.g., \cite{barberis2000investing}, \cite{fama1988dividend} and \cite{poterba1988mean}), where variables proxying time-varying expected excess returns are found to be strongly negatively correlated with contemporaneous stock returns and positively correlated with future returns.

In such a complete market setting, we derive optimal consumption and portfolio policies, as well as regularity results for the value function. As discussed in more detail below, this is achieved by means of a Lagrange duality approach, which allows us to connect the (primal) dynamic optimization problem with a no-borrowing constraint to a singular stochastic control problem. The latter is then analyzed through an auxiliary genuinely two-dimensional optimal stopping problem, in which one of the state variables exhibits stochastic volatility. In particular, we establish continuous differentiability of the optimal stopping problem's value function and characterize the optimal stopping time in terms of an excess-return-dependent free boundary. These properties are then instrumental in obtaining the complete solution to the original optimal consumption and portfolio choice problem.
\smallskip

\textbf{Methodology and Results.} The no-borrowing constraint requires that the agent's wealth $(X_t)_t$ remain nonnegative at all times, almost surely. Such a restriction affects the set of admissible consumption and portfolio choices and, when the agent's optimization problem is addressed via dynamic programming, it complements the Hamilton–Jacobi–Bellman equation with an appropriate boundary condition.

Inspired by \cite{el1998optimization}, \cite{he1993labor}, and the more recent \cite{jeon2025finite}, we instead adopt a duality-based approach to handle the no-borrowing constraint. Instead of enforcing $X_t \ge 0$ dynamically at every point in time, we reformulate it as a single static budget constraint. This transformation is achieved by introducing a non-increasing, c\`adl\`ag process $(D_t)_t$ that acts as an endogenous dynamic Lagrange multiplier designed to ensure $X_t \ge 0$ at all times $t\ge 0$ almost surely. This dual formulation allows us to express the problem in terms of an auxiliary process $(Z_t)_t$ rather than the wealth process itself. Consequently, the dual problem is recast as a two-dimensional singular control problem in the state variables $(Z_t, \beta_t)_t$, where $(Z_t)_t=(Z_t^D)_t$ is controlled through the monotone control $(D_t)_t$. Crucially, because the stochastic factor $(\beta_t)_t$ directly drives the diffusion of the dual state, the system inherently features stochastic volatility. It is worth noticing that directly approaching the dynamic programming principle (Hamilton-Jacobi-Bellman, HJB) equation related to the primal problem is particularly challenging in our setting, because of a non-geometric wealth dynamics due to the presence of labor income. As a matter of fact, in many classical settings, the wealth process is geometric, and the positivity of wealth is automatically guaranteed. In presence of power utility functions, this geometric structure also allows the value function to be scaled with wealth, reducing the HJB equation to a one-dimensional nonlinear ODE (see, e.g., \cite{merton1969lifetime} and \cite{Merton} for earlier works or the recent \cite{guasoni2025variational} and \cite{gutekunst2025optimal}). In our setting, however, the usual scaling arguments break down, the no-borrowing constraint must be enforced explicitly, and a duality approach is needed (see also Remark \ref{geometricremark} below).

A key methodological step in our analysis is relating the dual singular stochastic control problem to an equivalent two-dimensional optimal stopping problem--which inherits state variables featuring stochastic volatility--using probabilistic arguments. Through this approach, our main results provide a complete characterization of the investor's optimal behavior under stochastic investment opportunities, labor income, and a binding no-borrowing constraint. 

First, we establish that the auxiliary two-dimensional optimal stopping problem features a lower-semicontinuous free boundary $\beta \mapsto z^*(\beta)$ that strictly separates the continuation and stopping regions, thereby defining the optimal stopping time. We then prove that the stopping value function $v$ is locally Lipschitz continuous across the entire state space and infinitely differentiable within both the continuation region and the interior of the stopping region. This latter regularity result is obtained by showing that -- despite the degeneracy of the second-order differential operator $\mathcal{L}$ associated with the optimal stopping problem's state process and the stochastic volatility of one state variable -- H\"ormander's condition holds. Consequently, $\mathcal{L}$ is hypoelliptic, and $v$ is a classical solution to the associated partial differential equation within the continuation region. H\"ormander's condition also has the important implication that the optimal stopping problem's state process admits a smooth transition density. This, in turn, guarantees, via an application of a result in \cite{Jacka-Snell}, that $v$ is continuously differentiable across the entire state space.

By connecting the optimal stopping problem back to the dual singular control problem, we show that the dual value function is precisely the integral of the stopping value function with respect to the first component of the state process. This critical relationship allows us to uniquely characterize the optimal singular control $(D^*_t)_t$, which is the minimal process (à la Skorokhod) that keeps the dual state process $(Z_t, \beta_t)_t$ within the region $\{(z,\beta): z < z^*(\beta)\}$. Finally, in our complete market setting, we prove that strong duality holds, demonstrating that the primal value function can be recovered from the dual value function, and vice versa. Using these results, we are then able to retrieve the optimal consumption plan $(c^*_t)_t$, the optimal portfolio strategy $(\pi^*_t)_t$, and the optimal wealth process $(X^*_t)_t$, which, as required, satisfies the dynamic no-borrowing constraint.

Economically, these mathematical results offer a clear interpretation of the system's dynamics. The process $(D^*_t)_t$ represents a shadow price, reflecting the marginal cost of relaxing the wealth constraint. When consumption increases significantly, it exerts pressure on the borrowing constraint as wealth depletes more quickly. To maintain feasibility, $(D^*_t)_t$ adjusts downward, reflecting the reduced capacity to fund future consumption or investments. Concurrently, strong duality will imply that $(Z^{D^*}_t)_t$ represents the marginal value of an additional unit of wealth. As wealth becomes scarce and approaches zero, its marginal value naturally rises, causing $(Z^{D^*}_t)_t$ to increase. Specifically, as the wealth process approaches zero, the marginal value process $(Z^{D^*}_t)_t$ freely evolves upward until it hits the state-dependent free boundary $z^*$. Exactly at the point where $(Z^{D^*}_t)_t$ touches this boundary, the singular control $(D^*_t)_t$ activates and decreases, pushing $(Z^{D^*}_t)_t$ downward. By the established strong duality, this reflection of the dual state corresponds precisely to the wealth process being reflected upward, preventing bankruptcy and ensuring that the no-borrowing constraint is strictly satisfied.
\smallskip

\textbf{Related Literature.} The study of optimal consumption and investment problems with labor income and no-borrowing constraints dates back to 1993, when \cite{he1993labor} developed a duality approach to study an individual's optimal consumption and portfolio policy when borrowing against future labor income is limited. Similar to \cite{he1993labor}, \cite{el1998optimization} addresses an optimal consumption problem for an agent facing stochastic labor income and a strict no-borrowing constraint. They also utilize a duality approach to transform the constrained problem into a solvable, unconstrained dual problem. The resulting optimal strategy is characterized as a singular control problem, where the agent's actions near the wealth boundary ($X_t=0$) are described by a local time component. Crucially, the constrained optimal wealth is shown to be equivalent to the unconstrained wealth minus the value of an American put option, establishing a direct link between the portfolio problem and optimal stopping theory. However, while these earlier works include labor income and borrowing constraints, they do not feature a stochastic factor. Recent studies such as \cite{jeon2025finite} study optimal consumption, investment, and early retirement decisions for an agent under a finite-time horizon and a strict no-borrowing constraint against future labor income. Using the dual-martingale method, the problem is uniquely formulated as a two-person zero-sum game between a singular controller (managing the borrowing constraint) and a discretionary stopper (choosing the retirement time). The solution is governed by a min-max parabolic variational inequality that results in two time-varying free boundaries: one for optimal retirement and one for the active wealth binding constraint. While \cite{jeon2025finite} considers a related singular control structure, it again operates without a stochastic factor.

Optimal consumption and investment problems in the presence of a stochastic factor have also been widely studied. For example, \cite{munk2004optimal} consider an optimal consumption problem involving stochastic interest rates, while \cite{hata2018optimal} study a Merton consumption and portfolio problem with stochastic asset returns and volatilities (see also \cite{hata2025optimal}). Closely related to our work is the Kim-Omberg model \cite{kim1996dynamic}, which studies the dynamic non-myopic portfolio behavior of an investor trading a risk-free and a risky asset, with expected excess returns following a mean-reverting Ornstein-Uhlenbeck process. Extending this framework to include intermediate consumption, \cite{wachter2002portfolio} appears to be the first to derive exact optimal portfolio and consumption rules in this setting. However, unlike our model, \cite{wachter2002portfolio} does not incorporate labor income, works with a finite-time horizon and enforces the non-negativity constraint on wealth only at the terminal time $T$. Furthermore, \cite{guasoni2025variational} derive optimal consumption and investment policies in a complete market featuring a stochastic factor, modeled as a general scalar diffusion that drives investment opportunities. In their framework, the absence of labor income ensures that the wealth dynamics remain purely geometric. This in turn allows the value function to be scaled by wealth, thereby reducing the dimension of the HJB equation. Lastly, \cite{gutekunst2025optimal} consider an optimal consumption and investment model with general stochastic factors. They address optimal investment and consumption for a power utility investor using an incomplete stochastic factor model on an infinite horizon and provide a complete characterization for a finite state space. When the factor follows a diffusion process, they develop a new theoretical framework to prove existence and bound the HJB solution, verifying models like the Heston model rigorously for the first time. Conversely to our approach, \cite{gutekunst2025optimal} allows for general stochastic factors, including the Kim-Omberg setting, but assumes geometric wealth without labor income, which automatically enforces the non-negativity of wealth and crucially permits dimension reduction of the associated HJB equation.

In our analysis, the study of a stationary two-dimensional optimal stopping problem plays a crucial role. For optimal stopping problems involving multi-dimensional processes, the standard guess-and-verify approach is generally no longer applicable. This is because the free boundary separating the continuation and stopping regions is no longer a simple scalar threshold, but rather a complex curve or surface, making it practically infeasible to postulate a parameterized closed-form solution for both the value function and the boundary simultaneously. Consequently, a direct study of the problem's value function and the corresponding variational inequality must be performed on a case-by-case basis via probabilistic methods and/or techniques from the theory of partial differential equations. Notable contributions in this direction, with applications ranging from optimal dividend distribution to public debt reduction and quickest detection, include \cite{bandini}, \cite{callegaro2020optimal}, \cite{christensen}, \cite{de2017optimal}, \cite{ferrari2018optimal}, and \cite{quickest}, among others.

A major mathematical hurdle in our analysis arises from the presence of the stochastic factor, which explicitly introduces stochastic volatility into the optimal stopping problem. Beyond our specific model, optimal stopping under stochastic volatility is known to be highly challenging in general. In standard problems driven by uniformly elliptic diffusions, the value function typically enjoys strong smoothing properties across the entire state space. However, stochastic volatility breaks uniform ellipticity, rendering the associated infinitesimal generator degenerate. This degeneracy makes it particularly challenging to establish the global regularity of the value function and to rigorously characterize the behavior of the free boundary. For example, comparison theorems for solutions to SDEs, which are usually employed in optimal stopping theory to show monotonicity of the optimal stopping boundary, are generally not helpful in settings with stochastic volatility. As a consequence, it is particularly difficult to provide fine smoothness properties of the optimal stopping problem's value function. 

Because of these severe analytical challenges, the overall body of literature addressing optimal stopping under stochastic volatility remains relatively sparse, with a few notable contributions. \cite{jacka} proves, via purely probabilistic techniques, the monotonicity and continuity of the value function for an optimal stopping problem featuring stochastic volatility; \cite{frey} discusses the superreplication of derivatives in a stochastic volatility model under the additional assumption that the volatility follows a bounded process; \cite{Touzi-American} considers a stochastic volatility model for the asset price underlying an American option, extending regularity results for the American put option price function and proving that the optimal exercise boundary is a decreasing function of the current volatility process realization. Finally, \cite{LambertonTerenzi} provides an analytical characterization of the price function of an American option in Heston-type models using an approach based on variational inequalities, while \cite{LambertonTerenzi2} studies important properties of the American option price in the stochastic volatility Heston model, including monotonicity and smoothness of the value function, as well as its early-exercise-premium representation.
\smallskip

\textbf{Structure of the Paper.} The rest of the paper is organized as follows. In Section \ref{section2}, we introduce the primal problem and show how the no-borrowing constraint can be transformed into a static budget constraint. Section \ref{section3} derives the associated dual problem using duality principles and obtains the corresponding singular control formulation. We also establish a probabilistic link between this singular control problem and an auxiliary optimal stopping problem, which is fully analyzed in Sections \ref{section4}, \ref{preliminary} and \ref{freeboundarysection}. Section \ref{section5} uses the dual formulation to recover the optimal consumption and investment strategies in the primal problem. Finally, we provide numerical illustrations in Section \ref{numerics}.


\section{The Primal Problem}
\label{section2}

\subsection{The Financial Market and the Agent's Problem}
Let $(\Omega,\mathcal{F}, \mathbb{P})$ be a complete probability space and denote by $\mathbb{E}[\cdot]$ the expectation under $\mathbb{P}$. Consider an agent whose goal is to maximize the intertemporal expected utility functional given by

\[
\mathbb{E}\left[\int_{0}^{\infty} e^{-\delta t} u(c_t) \, dt \right],
\]
where \( \delta > 0 \) denotes the discount rate, and the utility function \( u \) is of power type; that is,

\[
u(c) = \frac{c^{1-\gamma}}{1-\gamma},
\]
where \( \gamma > 1 \) represents the agent's risk aversion. This value of \( \gamma \) is supported by empirical evidence on individual time preferences. Numerous studies, such as those by \cite{campbell1999force} and \cite{mehra1985equity}, document risk aversion rates well above 1 based on experimental and survey data. This behavior reflects strong present-biased preferences, which aligns with the assumption \( \gamma > 1 \) in our model.\\
\indent The market is described by the so-called Kim-Omberg model (see \cite{kim1996dynamic}), as it follows. The agent can invest in a risk-free asset with a risk-free rate \( r > 0 \) as well as a risky asset \( (S_t)_t \) whose dynamics are given by

\begin{equation}\label{riskyasset}
dS_t = (r + \beta_t) S_t \, dt + \sigma S_t \, dW_t,\quad t>0,\quad S_0=s>0,
\end{equation}
where \( (W_t)_t \) is a standard Brownian motion on $(\Omega,\mathcal{F},\mathbb{P})$ generating the filtration (completed by $\mathbb{P}\text{-null sets of}\; \mathcal{F}$) $\mathbb{F}^W:=(\mathcal{F}_t^W)_t$, and \( \sigma > 0 \) denotes the volatility of the risky asset. The expected return of the risky asset is \( \mu(\beta_t) := r + \beta_t \), where the process \( (\beta_t)_t \) represents the expected excess return of \( (S_t)_t \). Asset pricing studies suggest that expected excess returns are predictable and tend to revert to their long-run mean (see \cite{fama1988dividend} and \cite{poterba1988mean}, among others). To capture this mean-reversion, we define \( (\beta_t)_t \) as

\begin{equation}\label{primalbeta}
d\beta_t = \kappa (\overline{\beta} - \beta_t) \, dt - \sigma_\beta \, dW_t,\quad t>0,\quad \beta_0=\beta\in\mathbb{R},
\end{equation}
where \( \kappa > 0 \) is the speed of mean reversion, \( \overline{\beta} > 0 \) is the equilibrium level, and \( \sigma_\beta > 0 \) is the volatility of \( (\beta_t)_t \). By definition, the expected excess return \( (\beta_t)_t \) reverts to its long-run average \( \overline{\beta} \). As one can see from (\ref{riskyasset}) and (\ref{primalbeta}), $(S_t)_t$ and $(\beta_t)_t$ are perfectly negatively correlated. This complete-market specification is in the spirit of \cite{wachter2002portfolio}, where perfect negative correlation is also assumed in a Kim-Omberg setting, and it plays a fundamental role in the subsequent mathematical analysis. Empirically, this assumption is motivated by the evidence that variables proxying time-varying expected excess returns, such as the dividend-price ratio, are strongly negatively correlated with contemporaneous stock returns and positively correlated with future returns (cf.\ \cite{barberis2000investing}, \cite{fama1988dividend} and \cite{poterba1988mean}).\\
\indent The agent chooses a consumption plan \( (c_t)_t \), with \( c_t \geq 0 \), and an investment strategy for the risky asset, denoted by \( (\pi_t)_t \). Additionally, the agent receives a flow of constant labor income \( \ell > 0 \). The agent’s wealth \( (X_t)_t \) thus follows the dynamics

\begin{equation}\label{wealthdynamics}
dX_t = \left(r X_t + \beta_t \pi_t - c_t + \ell \right) \, dt + \sigma \pi_t \, dW_t,\quad t>0,
\end{equation}
with initial wealth \( X_0 = x>0 \).\\
As usual, we define the market price of risk as

\[
\theta(\beta_t) := \frac{\mu(\beta_t) - r}{\sigma} = \frac{\beta_t}{\sigma},\quad t\geq 0,
\]
and we also introduce the process \( (\mathcal{H}_t)_t \), which acts as a stochastic discount factor, as follows

\[
\mathcal{H}_t := \exp\left(- \int_{0}^{t} \left[ r + \frac{1}{2} \frac{\beta_s^2}{\sigma^2} \right] ds - \int_{0}^{t} \frac{\beta_s}{\sigma} \, dW_s \right),\quad t\geq 0;
\]
equivalently,

\begin{equation}\label{stochdiscount}
d\mathcal{H}_t = - r \mathcal{H}_t \, dt - \frac{\beta_t}{\sigma} \mathcal{H}_t \, dW_t,\quad t>0,\quad \mathcal{H}_0=1.
\end{equation}
\indent We make the following \textbf{standing assumption}.
\begin{assumption} \label{assumption_novikov}
We assume $\kappa\sigma>\sigma_\beta$. 
\end{assumption}
The requirement $\kappa\sigma>\sigma_\beta$ in Assumption \ref{assumption_novikov} implies (by the Novikov condition; see, e.g., Corollary 5.13 on p.\ 199 in \cite{karatzas2014brownian}) that \begin{equation}\label{martingale}M_t:=\exp\bigg(- \int_{0}^{t}  \frac{1}{2} \frac{\beta_s^2}{\sigma^2} ds - \int_{0}^{t} \frac{\beta_s}{\sigma} \, dW_s \bigg)=e^{rt}\mathcal{H}_t\end{equation} is an $\mathbb{F}^W$- martingale under $\mathbb{P}$, with $(\mathcal{H}_t)_t$ as in (\ref{stochdiscount}).\\
\indent The agent faces a no-borrowing constraint; that is, \( X_t \geq 0 \) $\mathbb{P}$-a.s.\ for all \( t \geq 0 \). This implies that the agent cannot borrow against future labor income and motivates the following definition of admissible controls.

\begin{definition}\label{def_admissible}
We call the pair of controls \( (\pi, c) \) admissible if:
\begin{enumerate}
    \item \( (c_t)_t\) and \( (\pi_t)_t \) are $\mathbb{F}^W$-progressively measurable, and are such that $c_t\geq 0$ $\mathbb{P}$-a.s.\ for all $t\geq 0$, \( \int_0^T c_s \, ds < \infty \) and \( \int_0^T \pi_s^2 \, ds < \infty \) $\mathbb{P}$-a.s.\ for all \( T > 0 \).
    \item \( X_t \geq 0 \) for all \( t \geq 0 \) $\mathbb{P}$-a.s.
\end{enumerate}
We denote by \( \mathcal{A}(x) \) the set of admissible controls.
\end{definition}
The agent's optimization problem then reads as

\begin{equation}\label{firstV}
\sup_{(\pi, c) \in \mathcal{A}(x)} \mathbb{E}\left[ \int_{0}^{\infty} e^{-\delta t} u(c_t) \, dt \right].
\end{equation}
\begin{remark}\label{geometricremark}
The recent \cite{gutekunst2025optimal} treats a general stochastic-factor model (including the Kim–Omberg setting) but works in a geometric-wealth framework without labor income. This framework automatically ensures the nonnegativity of wealth and allows for a dimension reduction of the associated HJB equation. By contrast, since we consider total investment and consumption and explicitly include labor income, these simplifications are no longer available. Consequently, we adopt the duality–singular control–optimal stopping approach developed in the following sections.
\end{remark}

\subsection{From a Dynamic to a Static Budget Constraint}
\indent Following \cite{jeon2025finite} (see also \cite{el1998optimization} and \cite{he1993labor} for earlier studies), we now transform the dynamic budget constraint \( X_t \geq 0 \) $\mathbb{P}$-a.s.\ $\forall$ \( t \geq 0 \) into a single, static  budget constraint. To that end, we define
\begin{equation}
\mathcal{D}
:= \Big\{ (D_t)_{t \ge 0} : 
D \text{ is } \mathbb{F}^W\text{-adapted, nonnegative, nonincreasing, càdlàg, and } D_{0^-}=1\;\mathbb{P}\text{-a.s.}
\Big\},
\end{equation}
and we then have the following result.

\begin{prop}\label{prop_staticbudget}
\begin{enumerate}
    \item Let $(c_t)_t$ be a consumption plan such that \( (\pi, c) \in \mathcal{A}(x) \). Then it also satisfies the constraint
    \begin{equation}\label{budgetconstraint}
    \sup_{D \in \mathcal{D}} \mathbb{E} \left[ \int_0^\infty \mathcal{H}_t D_t (c_t - \ell) \, dt \right] \leq x.    
    \end{equation}
    Moreover, we have $\mathbb{E} \left[ \int_0^\infty \mathcal{H}_t |c_t - \ell| \, dt \right] < \infty$.
    \item For any nonnegative $\mathbb{F}^W$-progressively measurable \( (c_t)_t \) with \( \int_0^T c_s \, ds < \infty \) $\mathbb{P}$-a.s.\ for all $T>0$, and such that
    \begin{equation}\label{sup=x}
    \sup_{D \in \mathcal{D}} \mathbb{E} \left[ \int_0^\infty \mathcal{H}_t D_t (c_t - \ell) \, dt \right] = x,
    \end{equation}
    there exists a process \( (\pi_t)_t \) such that \( (\pi, c) \in \mathcal{A}(x) \).
\end{enumerate}
\end{prop}
\begin{proof}
\indent\emph{Proof of (1).} The proof borrows arguments from \cite{jeon2025finite}.\\
\indent Let $(c_t)_t$ be a consumption plan such that $(\pi,c)\in\mathcal{A}(x)$ and let $D\in\mathcal{D}$. An application of \text{It\^o}'s formula for semimartingales (see Theorem 32 on p.\ 78 in \cite{protter2012stochastic}) yields \begin{equation}\label{HXD}d(\mathcal{H}_tX_tD_{t-})=\mathcal{H}_tD_t(\ell -c_t)dt+\mathcal{H}_tD_t(\sigma\pi_t-\frac{\beta_t}{\sigma}X_t)dW_t+\mathcal{H}_tX_tdD_{t-}.\end{equation} Next, we define the localizing sequence of stopping times \[\tau_n:=\inf\bigg\lbrace s\geq 0:\int_0^s\mathcal{H}^2_tD^2_t\bigg|\sigma\pi_t-\frac{\beta_t}{\sigma}X_t\bigg|^2dt\geq n\bigg\rbrace,\quad n\in\mathbb{N},\] and note that $\tau_n\rightarrow\infty$ a.s., as $n\rightarrow\infty$. Integrating (\ref{HXD}) and rearranging terms we find \begin{align*}\int_0^{\tau_n}\mathcal{H}_sD_s(c_s-\ell )ds-x&=\int_0^{\tau_n}\mathcal{H}_sD_s(\sigma\pi_s-\frac{\beta_s}{\sigma}X_s)dW_s\\
&\quad+\int_0^{\tau_n}\mathcal{H}_sX_sdD_{s-}-\mathcal{H}_{\tau_n}X_{\tau_n}D_{\tau_n-},
\end{align*}
from which, by taking expectations, \begin{align}\label{tau_n}
&\mathbb{E}\bigg[\int_0^{\tau_n}\mathcal{H}_sD_s(c_s-\ell )ds\bigg]-x \notag\\
&\qquad = \mathbb{E}\bigg[\int_0^{\tau_n}\mathcal{H}_sD_s(\sigma\pi_s-\frac{\beta_s}{\sigma}X_s)dW_s+\int_0^{\tau_n}\mathcal{H}_sX_sdD_{s-}-\mathcal{H}_{\tau_n}X_{\tau_n}D_{\tau_n-}\bigg]\\
&\qquad=\underbrace{\mathbb{E}\bigg[\int_0^{\tau_n}\mathcal{H}_sD_s(\sigma\pi_s-\frac{\beta_s}{\sigma}X_s)dW_s\bigg]}_{=0}+\underbrace{\mathbb{E}\bigg[\int_0^{\tau_n}\mathcal{H}_sX_sdD_{s-}\bigg]}_{\leq 0}-\underbrace{\mathbb{E}\bigg[\mathcal{H}_{\tau_n}X_{\tau_n}D_{\tau_n-}\bigg]}_{\geq 0}\notag\\
&\qquad\leq 0\notag.
\end{align} Here we used the facts that \( \mathcal{H}_t \geq 0 \), \( X_t \geq 0 \) since \( (\pi, c) \in \mathcal{A}(x) \), and that \( D \) is nonincreasing and nonnegative. Applying Fatou's lemma for lower bounded functions, together with (\ref{tau_n}), yields
\begin{align*}
\mathbb{E}\bigg[\int_0^{\infty}\mathcal{H}_sD_s(c_s-\ell )ds\bigg]-x&=\mathbb{E}\bigg[\liminf_{n\rightarrow\infty}\int_0^{\tau_n}\mathcal{H}_sD_s(c_s-\ell )ds\bigg]-x\\
&\leq \liminf_{n\rightarrow\infty} \mathbb{E}\bigg[\int_0^{\tau_n}\mathcal{H}_sD_s(c_s-\ell )ds\bigg]-x\leq 0;
\end{align*} hence, \[\mathbb{E}\bigg[\int_0^\infty\mathcal{H}_tD_t(c_t-\ell )dt\bigg]\leq x,\] and by arbitrariness of $D\in\mathcal{D}$, we obtain (\ref{budgetconstraint}); that is, \[\sup_{D\in \mathcal{D}}\mathbb{E}\bigg[\int_0^\infty\mathcal{H}_tD_t(c_t-\ell )dt\bigg]\leq x.\] In particular, since $1\in\mathcal{D}$, we have 
\[
\mathbb{E}\!\left[ \int_0^\infty \mathcal{H}_t (c_t - \ell)\, dt \right]
    \le x.\]
Using the fact that $\mathcal{H}_t = e^{-rt} M_t$, where $(M_t)_t$ is the martingale given by (\ref{martingale}), we also find \[\mathbb{E} \left[ \int_0^\infty \mathcal{H}_t c_t  \, dt \right]\leq x+\mathbb{E} \left[ \int_0^\infty e^{-rt}\ell M_t \, dt \right]=x+\frac{\ell}{r},\] and therefore, \begin{equation}\label{H|c-ell|}\mathbb{E} \left[ \int_0^\infty \mathcal{H}_t |c_t - \ell| \, dt \right]\leq x+\frac{2\ell}{r}<\infty.\end{equation}
\indent\emph{Proof of (2).} We set \begin{equation}\label{A_t}A_t:=\int_0^t\mathcal{H}_s(c_s-\ell)ds,\quad t\geq 0,\end{equation} and have for any $\mathbb{F}^W$-stopping time $\eta$ \[|A_\eta|\leq \int_0^\infty\mathcal{H}_s|c_s-\ell|ds.\] From (\ref{H|c-ell|}), which holds thanks to (\ref{sup=x}) and can be shown as in the proof of (1), it then follows that the process $(A_t)_t$ is continuous and belongs to Class D\footnote{A process $(\zeta_t)_t$ is of Class (D) if the family of random variables $\left\{ \zeta_\tau : \tau<\infty\; \mathbb{P}\text{-a.s.\ is a  stopping time} \right\}$
    is uniformly integrable.}.\\
\indent Now fix an almost surely finite $\mathbb{F}^W$-stopping time $\eta$ and define
\[
D_t^\eta := \mathbf{1}_{\{t<\eta\}},
\qquad t\ge 0,
\]
so that $D^\eta \in \mathcal{D}$, which implies by \eqref{sup=x}
\[
\mathbb{E}\left[\int_0^\eta \mathcal{H}_t (c_t-\ell)\,dt\right]\le x,
\]
or equivalently
\begin{equation}\label{eq:EA-tau}
\mathbb{E}[A_\eta]\le x,
\end{equation}
upon using (\ref{A_t}). We then define the Snell envelope (cf., e.g., page 8 in \cite{peskir2006optimal}) of $(A_t)_t$ by
\[
Y_t := \operatorname*{ess\,sup}_{\eta\geq t}
\mathbb{E}[A_\eta\mid \mathcal{F}^W_t],
\qquad t\ge 0.
\]
Since $(A_t)_t$ is continuous and belongs to Class D, it thus follows that $(Y_t)_{t\ge 0}$ is a continuous supermartingale belonging to Class D, and it is the smallest supermartingale dominating $(A_t)_t$; that is,
\begin{equation}\label{YgeqA}
Y_t \ge A_t,
\qquad\text{for all } t\ge 0.
\end{equation}
Hence, \eqref{eq:EA-tau} implies
\begin{equation}\label{Y_0leqx}
Y_0 = \sup_{\eta\geq0}\mathbb{E}[A_\eta]\le x.
\end{equation}

Next, we define the continuous supermartingale
\[
\bar Y_t := Y_t + (x-Y_0),
\qquad t\ge 0 ,
\]
which, thanks to (\ref{YgeqA}) and (\ref{Y_0leqx}), satisfies
\[
\bar Y_0 = x
\qquad\text{and}\qquad
\bar Y_t \ge A_t
\quad\text{for all } t\ge 0.
\]
By the Doob-Meyer decomposition (cf., e.g., Theorem 3.1 on page 56 in \cite{peskir2006optimal}), there exists a uniformly integrable continuous martingale $(\bar{N}_t)_{t\ge 0}$ and a continuous predictable increasing process $(K_t)_{t\ge 0}$ with $K_0=0$ such that
\begin{equation}\label{barY}
\bar Y_t = x + \bar{N}_t - K_t,
\qquad t\ge 0.
\end{equation}
Moreover, the martingale representation theorem (cf., e.g., Theorem 4.3.4 on page 53 in \cite{oksendal}) yields a progressively measurable process $(\tilde{Z}_t)_{t\ge 0}$ such that
\begin{equation}\label{N}
\bar{N}_t = \int_0^t \tilde{Z}_s\,dW_s,
\qquad t\ge 0,
\end{equation}
and
\[
\int_0^T \tilde{Z}_s^2\,ds<\infty
\qquad \mathbb{P}\text{-a.s. for all } T>0.
\]

We now define
\begin{equation}\label{hatX}
\widehat X_t := \bar Y_t - A_t + K_t,
\qquad t\ge 0.
\end{equation}
Since $\bar Y_t\ge A_t$ and $K_t\ge 0$, we have
\[
\widehat X_t \ge 0,
\qquad t\ge 0,
\] and, upon using $K_0=A_0=0$ and $\bar{Y}_0=x$, also $\widehat X_0=x$.
Plugging (\ref{A_t}) and (\ref{barY}) into (\ref{hatX}) yields
\[
\widehat X_t
= x + \bar{N}_t - K_t - \int_0^t \mathcal{H}_s(c_s-\ell)\,ds + K_t
= x + \int_0^t \mathcal{H}_s(\ell-c_s)\,ds + \int_0^t \tilde{Z}_s\,dW_s,
\] where we have used (\ref{N}) in the final display equation above, as well.
Therefore,
\begin{equation}\label{eq:dhatX}
d\widehat X_t
= \mathcal{H}_t(\ell-c_t)\,dt + \tilde{Z}_t\,dW_t,
\qquad \widehat X_0=x.
\end{equation}

Let us now construct the candidate investment policy $(\pi_t)_t$ as follows
\begin{equation}\label{eq:def-pi-final}
\pi_t
:= \frac{1}{\sigma\mathcal{H}_t}
\left(
\tilde{Z}_t+\frac{\beta_t}{\sigma}\widehat X_t
\right),
\qquad t\ge 0,
\end{equation}
which is clearly $\mathbb{F}^W$-progressively measurable.\\
\indent  In order to show \begin{equation}\label{piadmissible}
\int_0^T \pi_t^2\,dt<\infty
\qquad \mathbb{P}\text{-a.s. for all } T>0,
\end{equation}
we note that the processes $(\mathcal{H}_t)_t$, $(\mathcal{H}^{-1}_t)_t$, $(\beta_t)_t$, and $(\widehat X_t)_t$ are continuous, and hence pathwise bounded. Since we also have $\int_0^T \tilde{Z}_s^2\,ds<\infty$ $\mathbb{P}$-a.s.\ for all $T>0$, we obtain (\ref{piadmissible}) from \eqref{eq:def-pi-final}.

It remains to verify that $(X^{\pi,c}_t)_{t\ge 0}$ is such that $X^{\pi,c}_t\geq 0$ for all $t\geq 0$, where $(X^{\pi,c}_t)_{t\ge 0}$ denotes the wealth controlled through $(\pi,c)$. As the process $(\mathcal{H}_t)_t$ (cf.\ (\ref{stochdiscount})) is such that \[d(\mathcal{H}_t^{-1})=\mathcal{H}_t^{-1}\left(r+\frac{\beta_t^2}{\sigma^2}\right)dt
+\frac{\beta_t}{\sigma}\mathcal{H}_t^{-1}\,dW_t,\] applying It\^o's formula to the process $(\mathcal{H}_t^{-1}\widehat X_t)_t$ yields 
\begin{align*}
d(\mathcal{H}_t^{-1}\widehat X_t)
&= \mathcal{H}_t^{-1}d\widehat X_t
+\widehat X_t\,d(\mathcal{H}_t^{-1})
+d[\widehat X,\mathcal{H}^{-1}]_t \\
&= (\ell-c_t)\,dt + \mathcal{H}_t^{-1}\tilde{Z}_t\,dW_t
+\mathcal{H}_t^{-1}\widehat X_t\left(r+\frac{\beta_t^2}{\sigma^2}\right)dt
+\frac{\beta_t}{\sigma}\mathcal{H}_t^{-1}\widehat X_t\,dW_t
+\frac{\beta_t}{\sigma}\mathcal{H}_t^{-1}\tilde{Z}_t\,dt,
\end{align*}
where \eqref{eq:dhatX} has also been used.
Rearranging terms gives
\[
d(\mathcal{H}_t^{-1}\widehat X_t)
=
\left(
r\mathcal{H}_t^{-1}\widehat X_t+\ell-c_t+\frac{\beta_t}{\sigma}
\left(
\mathcal{H}_t^{-1}\tilde{Z}_t+\frac{\beta_t}{\sigma}\mathcal{H}_t^{-1}\widehat X_t
\right)
\right)dt
+
\left(
\mathcal{H}_t^{-1}\tilde{Z}_t+\frac{\beta_t}{\sigma}\mathcal{H}_t^{-1}\widehat X_t
\right)dW_t,
\]
or equivalently, upon using \eqref{eq:def-pi-final},
\[
d(\mathcal{H}_t^{-1}\widehat X_t)
=
(r\mathcal{H}_t^{-1}\widehat X_t+\beta_t\pi_t-c_t+\ell)\,dt+\sigma\pi_t\,dW_t,
\qquad \mathcal{H}_0^{-1}\widehat X_0=x.
\]
Thus, as $\mathcal{H}_0^{-1}\widehat X_0=x=X^{\pi,c}_0$, comparing the dynamics above with $(\ref{wealthdynamics})$ allows us to conclude
\begin{equation}\label{Xpic}
X_t^{\pi,c}=\mathcal{H}_t^{-1}\widehat X_t,
\qquad t\ge 0,
\end{equation} by the uniqueness of the solution to the SDE. Finally, since $\widehat X_t\ge 0$ and $\mathcal{H}_t>0$ for all $t\ge 0$, we have from (\ref{Xpic})
\[
X_t^{\pi,c}\ge 0
\qquad\text{for all } t\ge 0,
\]
and therefore, $(\pi,c)\in\mathcal{A}(x)$.
\end{proof}
\begin{remark}\label{remark}
Note that since $1 \in \mathcal{D}$, we in particular have from Proposition \ref{prop_staticbudget}-$(1)$
\[
\mathbb{E} \left[ \int_0^\infty \mathcal{H}_t (c_t - \ell)\, dt \right] \leq x.
\]
This means that the present value of discounted consumption, net of labor income, cannot exceed the agent's initial wealth. However, this condition alone does not guarantee that the wealth process remains nonnegative with probability one at all times. Therefore, we need a stronger requirement, which is achieved by using the processes $D$ as in Proposition \ref{prop_staticbudget} above.
\end{remark}
\section{The Dual Problem as a Singular Control Problem}\label{section3}
\subsection{Derivation of the Dual Problem}
In this section, we derive the dual problem expected to be associated to (\ref{V}). Because (cf.\;(\ref{budgetconstraint})) \[
\sup_{D\in \mathcal{D}}\mathbb{E}\bigg[\int_0^\infty\mathcal{H}_tD_t(c_t-\ell )dt\bigg]\leq x,
\] with $(\mathcal{H}_t)_t$ as in (\ref{stochdiscount}), for a given $z>0$ we can write
\begin{align}\label{derivationdualproblem}
\mathbb{E}\bigg[\int_0^\infty e^{-\delta t}u(c_t)dt\bigg]
&\leq \mathbb{E}\bigg[\int_0^\infty e^{-\delta t}u(c_t)dt\bigg] + z\bigg(x - \sup_{D\in \mathcal{D}}\mathbb{E}\bigg[\int_0^\infty \mathcal{H}_tD_t(c_t - \ell)dt\bigg]\bigg) \notag\\
&= \mathbb{E}\bigg[\int_0^\infty e^{-\delta t}u(c_t)dt\bigg] + zx + \inf_{D\in \mathcal{D}}\mathbb{E}\bigg[-\int_0^\infty z\mathcal{H}_tD_t(c_t - \ell)dt\bigg] \\
&= \inf_{D\in \mathcal{D}}\mathbb{E}\bigg[\int_0^\infty e^{-\delta t}\big(u(c_t) - e^{\delta t}z\mathcal{H}_tD_t(c_t - \ell)\big)dt\bigg] + zx.\notag
\end{align}
To simplify notation, in the following we set \begin{equation}\label{Y}Z^{D}_t := e^{\delta t}z\mathcal{H}_tD_t,\quad Z_{0^-} = z>0.\end{equation} Given that the Legendre–Fenchel transform of the utility function \(u\) is such that
\begin{equation}\label{legendre}
\tilde{u}(z) := \sup_{c\geq 0}\left(u(c) - zc\right) = \sup_{c\geq 0}\bigg(\frac{c^{1-\gamma}}{1-\gamma} - zc\bigg)=\frac{\gamma}{1-\gamma}z^{-\frac{1-\gamma}{\gamma}},
\end{equation}
using \(u(c_t) - Z^{D}_t c_t \leq \tilde{u}(Z^{D}_t)\) in (\ref{derivationdualproblem}), we obtain
\begin{equation}\label{Lleq}
\mathbb{E}\bigg[\int_0^\infty e^{-\delta t}u(c_t)dt\bigg] \leq \inf_{D\in \mathcal{D}}\mathbb{E}\bigg[\int_0^\infty e^{-\delta t}\big(\tilde{u}(Z^{D}_t) + \ell Z^{D}_t\big)dt\bigg] + zx,
\end{equation}
which, by arbitrariness of $z>0$, in turn yields
\[
\mathbb{E}\bigg[\int_0^\infty e^{-\delta t}u(c_t)dt\bigg] \leq \inf_{\substack{z>0 \\ D\in\mathcal{D}}}\bigg(\mathbb{E}\bigg[\int_0^\infty e^{-\delta t}\big(\tilde{u}(Z^{D}_t) + \ell Z^{D}_t\big)dt\bigg] + zx\bigg).
\]
Hence, we have the weak duality
\begin{align}\label{weakduality}
\sup_{(\pi,c)\in\mathcal{A}(x)}\mathbb{E}\bigg[\int_0^\infty e^{-\delta t}u(c_t)dt\bigg]\leq \inf_{z>0}\bigg(\inf_{D\in\mathcal{D}}\mathbb{E}\bigg[\int_0^\infty e^{-\delta t}\big(\tilde{u}(Z^{D}_t) + \ell Z^{D}_t\big)dt\bigg] + zx\bigg).
\end{align}
\\
\indent In the subsequent analysis, we shall focus on the problem 
\begin{equation}\label{SCP}
\inf_{D\in\mathcal{D}}\mathbb{E}\bigg[\int_0^\infty e^{-\delta t}\big(\tilde{u}(Z^{D}_t)+\ell Z^{D}_t\big)dt\bigg] ,   
\end{equation} with the aim of proving that actually the strong duality
\begin{equation}\label{strong duality}\sup_{(\pi,c)\in\mathcal{A}(x)}\mathbb{E}\bigg[\int_0^\infty e^{-\delta t}u(c_t)dt\bigg]= \inf_{z>0}\bigg(\inf_{D\in\mathcal{D}}\mathbb{E}\bigg[\int_0^\infty e^{-\delta t}\big(\tilde{u}(Z^{D}_t) + \ell Z^{D}_t\big)dt\bigg] + zx\bigg)\end{equation} holds true. Problem (\ref{SCP}) is a singular stochastic control problem for the two-dimensional state process
\begin{equation}\label{Z^D}
dZ^{D}_t = (\delta - r)Z^{D}_t\,dt - \frac{\beta_t}{\sigma}Z^{D}_t\,dW_t + Z^{D}_t\frac{dD_t}{D_t}, \quad t>0,\quad Z_{0^-}^{D}=z>0,
\end{equation}
\begin{equation}\label{beta}
d\beta_t = \kappa(\overline{\beta} - \beta_t)\,dt - \sigma_\beta\,dW_t,\quad t>0,\quad \beta_0=\beta\in\mathbb{R},
\end{equation} with $D\in\mathcal{D}$. Notice that the dynamics of $(Z_t^D)_t$ is easily obtained from (\ref{Y}) via It\^{o}'s formula.\\
\indent Given the Markovian structure, from now on we stress the dependency of the value of (\ref{SCP}) with respect to the problem's initial data and write \begin{equation}\label{tildeV}\tilde{V}(z,\beta) := \inf_{D \in \mathcal{D}} \mathbb{E}_{z,\beta}\bigg[\int_0^\infty e^{-\delta t}\big(\tilde{u}(Z^{D}_t) + \ell Z^{D}_t\big)\,dt\bigg],\end{equation} where $\mathbb{E}_{z,\beta}[\;\cdot\;]=\mathbb{E}[\;\cdot\;|Z_{0^-}^D=z,\beta_0=\beta]$ denotes the expectation under $\mathbb{P}_{z,\beta}(\cdot):=\mathbb{P}(\;\cdot\;|\;Z^{D}_{0-}=z,\;\beta_0=\beta)$. For later use, we notice that $\lim_{z\downarrow0}\tilde{V}(z,\beta)=0$. Analagously, in the following, we shall write (cf.\ (\ref{firstV}))
\begin{equation}\label{V}V(x, \beta) := \max_{(\pi, c) \in \mathcal{A}(x)} \mathbb{E}_{x,\beta} \left[ \int_{0}^{\infty} e^{-\delta t} u(c_t) \, dt \right],\end{equation}
where $\mathbb{E}_{x,\beta}[\cdot]$ denotes the expectation under $\mathbb{P}_{x,\beta}(\cdot):=\mathbb{P}(\;\cdot\;|\;X_0=x,\;\beta_0=\beta)$.\\
\\
\indent In Proposition \ref{optimalstrategies} below we will show that \[V(x,\beta)=\inf_{z>0}\bigg(\tilde{V}(z,\beta)+zx\bigg),\] so that (\ref{strong duality}) indeed holds true. This will be achieved through a series of intermediate results aimed at characterizing the optimal policy of problem (\ref{SCP}). A major ingredient towards this characterization is the identification of an optimal stopping problem whose value coincides with $\tilde{V}_z$. Such an optimal stopping problem is introduced and studied in the following sections.
\begin{remark}
We emphasize that while the complete-market assumption is needed in order to establish strong duality and recover the solution to the primal problem from the dual formulation, the singular stochastic control problem \eqref{SCP} and the associated auxiliary optimal stopping problem introduced and studied below can be formulated and analyzed as in the subsequent sections more generally when the factor process $(\beta_t)_t$ is driven by a Brownian motion $(W_t^\beta)_t$ satisfying
\[
W_t^\beta=\rho W_t+\sqrt{1-\rho^2}\,W_t^\perp,\qquad t\ge 0,
\]
where $(W^\perp_t)_t$ is a standard Brownian motion independent of $(W_t)_t$ and $\rho\in[-1,1]$.
\end{remark}
\subsection{Derivation of the Auxiliary Optimal Stopping Problem}\label{section4}
Denote by \((Z^{1}_t)_t\) the uncontrolled state process (i.e., $Z^D$ as in (\ref{Z^D}) with $D\equiv 1$) satisfying
\begin{equation}\label{Z^1}
dZ^{1}_t = (\delta - r)Z^{1}_t\,dt - \frac{\beta_t}{\sigma}Z^{1}_t\,dW_t,\quad t>0,\quad Z_0^{1}=z>0.
\end{equation} Furthermore, to simplify notation, set $\mathcal{O}:=(0,\infty)\times\mathbb{R}$.\\
\\
\indent Inspired by \cite{baldursson1996irreversible}, \cite{callegaro2020optimal} and \cite{ferrari2018optimal} we introduce the optimal stopping problem 
\begin{equation}\label{OSTunderP}
v(z,\beta)
:= \inf_{\tau} \mathbb{E}_{z,\beta}\bigg[\int_0^\tau e^{-rt} M_t\,(\tilde{u}'(Z_t^{1}) + \ell)\,dt\bigg],
\end{equation} where $\tilde{u}'(z)=-z^{1/\gamma}$, and where we take the infimum over $\mathbb{F}^W$-stopping times $\tau\geq 0$, $Z^1$ evolves as in (\ref{Z^1}), and $\beta$ as in (\ref{beta}). We expect $v$ to be such that $\tilde{V}_z=v$ on $\mathcal{O}$. Theorem \ref{establish} below will indeed prove that such a relation holds true and that an optimizer for $v$ is in one-to-one correspondence to an optimizer for $\tilde{V}$. In the following, we shall study (\ref{OSTunderP}) and characterize its optimal policy. In order to achieve this, it is convenient to perform a change of measure to remove the martingale $(M_t)_t$ (cf.\ (\ref{martingale})) from the stopping functional. This leads to the next proposition whose proof is postponed to the appendix.
\begin{prop}\label{measurechange}
For $v$ as in (\ref{OSTunderP}) we have 
\begin{equation}\label{OSTunderQ}v(z,\beta)=\inf_{\tau} \mathbb{E}_{z,\beta}^{\mathbb{Q}}\bigg[\int_0^\tau e^{-rt}\,(\tilde{u}'(\hat{Z}_t) + \ell)\,dt\bigg],\end{equation} for a suitable probability measure $\mathbb{Q}$ on a suitable measurable space $(\hat{\Omega},\hat{\mathcal{F}})$. The dynamics of the state processes \((\hat{Z}_t)_t\) and \((\hat{\beta}_t)_t\) are given under $\mathbb{Q}$ by
\begin{align*}
d\hat{Z}_t &= -\frac{\hat{\beta}_t}{\sigma}\hat{Z}_t\,dW^\mathbb{Q}_t + \hat{Z}_t\bigg(\delta - r + \frac{\hat{\beta}_t^2}{\sigma^2}\bigg)\,dt, \quad t>0,\quad\hat{Z}_0=z>0, \\
d\hat{\beta}_t &= -\sigma_\beta\,dW^{\mathbb{Q}}_t + \bigg(\kappa(\overline{\beta} - \hat{\beta}_t) +\,\frac{\hat{\beta}_t}{\sigma}\sigma_\beta\bigg)\,dt,\quad t>0,\quad\hat{\beta}_0=\beta\in\mathbb{R}.
\end{align*} Here, $(W^\mathbb{Q}_t)_t$ is a standard Brownian motion on $(\hat{\Omega},\hat{\mathcal{F}},\mathbb{Q})$, generating the filtration (completed by $\mathbb{Q}\text{-null sets of}\; \hat{\mathcal{F}}$) $\mathbb{F}^{W,\mathbb{Q}}:=(\mathcal{F}_t^{W,\mathbb{Q}})_t$, and $\mathbb{E}_{z,\beta}^\mathbb{Q}[\;\cdot\;]$ is the expectation under $\mathbb{Q}_{z,\beta}=\mathbb{Q}(\;\cdot\mid \hat{Z}_0=z,\;\hat{\beta}_0=\beta)$. Finally, the optimization in (\ref{OSTunderQ}) is performed over $\mathbb{F}^{W,\mathbb{Q}}$-stopping times. 
\end{prop}
\begin{proof}
See Appendix \ref{proofmeasurechange}.
\end{proof}
\noindent
With reference to Proposition \ref{measurechange}, we therefore now turn our attention to characterizing the solution to the optimal stopping problem \begin{equation}\label{OST_new}v(z,\beta):=\inf_\tau\mathbb{E}^\mathbb{Q}_{z,\beta}\bigg[\int_0^\tau e^{-rt}(\tilde{u}'(\hat{Z}_t)+\ell )dt\bigg],\end{equation} subject to
\begin{align}
d\hat{Z}_t&=-\frac{\hat{\beta}_t}{\sigma}\hat{Z}_tdW^{\mathbb{Q}}_t+\hat{Z}_t(\delta-r+\frac{\hat{\beta}_t^2}{\sigma^2})dt,\quad t>0,\quad \hat{Z}_0=z>0,\label{hatZ}\\
d\hat{\beta}_t&=-\sigma_\beta dW_t^{\mathbb{Q}}+(\kappa(\overline{\beta}-\hat{\beta}_t)+\frac{\hat{\beta}_t}{\sigma}\sigma_\beta)dt,\quad t>0,\quad\hat{\beta}_0=\beta\in\mathbb{R}.\label{hatbeta}
\end{align} 
In the following, when needed, we stress the dependence of the unique strong solution to (\ref{hatZ})-(\ref{hatbeta}) on the initial data $(z,\beta)\in\mathcal{O}$ by writing $(\hat{Z}_t^{z,\beta})_t$ and $(\hat{\beta}_t^\beta)_t$.
\subsection{Preliminary Properties of the Optimal Stopping Value Function}\label{preliminary}
In this subsection, we establish preliminary properties of the value function (\ref{OST_new}). For the proof of those, we make the following assumption on the model's parameters. Such requirement in particular ensures well-posedness of $v$ as in (\ref{OST_new}) and, together with Assumption \ref{assumption_novikov}, it will be a \textbf{standing assumption} throughout the rest of the paper.

\begin{assumption}\label{Novikov_OST}
We assume that 
\[
\gamma>\max\bigg\lbrace 1,\frac{\sigma_\beta}{\sigma\bigg(\kappa-\frac{ \sigma_\beta}{\sigma}\bigg)}\bigg\rbrace.
\] 
\end{assumption}
Notice that $\kappa-\frac{\sigma_\beta}{\sigma}>0$ due to Assumption \ref{assumption_novikov}. We then have the following first preliminary finding.
\begin{prop}\label{wellposedness}
It holds
\[
\mathbb{E}^\mathbb{Q}_{z,\beta}\bigg[\int_0^\infty e^{-rt}|\tilde{u}'(\hat{Z}_t)+\ell|\,dt\bigg]<\infty.
\] Moreover, one has \begin{equation}\label{E(Z)ineq}\mathbb{E}^\mathbb{Q}_{z,\beta}[\hat{Z}_t^{-\frac{1}{\gamma}}]
\leq z^{-\frac{1}{\gamma}}\exp\bigg(-\tfrac{1}{\gamma}(\delta-r)t\bigg).\end{equation}
\end{prop}

\begin{proof}
See Appendix \ref{proofwellposedness}.
\end{proof}

The next result directly follows from the expression of $v$ as in (\ref{OST_new}) and the fact that $z\mapsto  \hat{Z}^{z,\beta}_t$ is $\mathbb{Q}$-a.s.\ increasing for all $t\geq 0$.

\begin{prop}\label{monotonicity}
One has that $z\mapsto v(z,\beta)$ is nondecreasing for all $\beta\in\mathbb{R}$.   
\end{prop}

\begin{remark}
Note that the monotonicity of $\beta\mapsto v(z,\beta)$ is not clear, since the process $(\hat{\beta}_t)_t$ also affects the volatility of $(\hat{Z}_t)_t$, and therefore comparison theorems for solutions to SDEs do not apply.
\end{remark}
The next results provide useful bounds and limit behavior of the value function $v$. Their proofs are given in the Appendix.

\begin{prop}\label{bounds}
We have
\[
-\mathbb{E}^\mathbb{Q}_{z,\beta}\bigg[\int_0^\infty e^{-rt} \hat{Z}_t^{-\frac{1}{\gamma}} dt\bigg] \le v(z,\beta) \le 0.
\]
\end{prop}

\begin{proof}
See Appendix \ref{proofbounds}.
\end{proof}
\begin{prop}\label{limits}
It holds that
\begin{equation*}
\lim_{z \to 0} v(z,\beta) = -\infty \quad \text{and} \quad \lim_{z \to \infty} v(z,\beta) = 0.
\end{equation*}
\end{prop}
\begin{proof}
See Appendix \ref{prooflimits}.
\end{proof}

\begin{prop}\label{uppersemi}
The value function $v(z,\beta)$ is continuous for all $(z,\beta) \in \mathcal{O}$.    
\end{prop}

\begin{proof}
\textbf{Step 1.} We first show that $(z,\beta)\mapsto v(z,\beta)$ is upper semicontinuous for all $(z,\beta)\in\mathcal{O}$. Given (\ref{OST_new}), it suffices to show that 
\[
\mathcal{J}(\tau;z,\beta):= \mathbb{E}^\mathbb{Q}_{z,\beta}\bigg[\int_0^\tau e^{-rt} \big(-\hat{Z}_t^{-\frac{1}{\gamma}} + \ell \big) dt \bigg]
\] 
is continuous for all $(z,\beta) \in \mathcal{O}$ and fixed $\tau\geq 0$.\\
\indent Fix $(z_0, \beta_0) \in \mathcal{O}$ and let $(z_n,\beta_n)\subseteq\mathcal{O}$ be a sequence converging to $(z_0,\beta_0)$. Then
\begin{align}{\label{uscinequality}}
\big|\mathcal{J}(\tau;z_n,\beta_n) - \mathcal{J}(\tau;z_0,\beta_0)\big| &\le 
\big|\mathcal{J}(\tau;z_n,\beta_n) - \mathcal{J}(\tau;z_0,\beta_n)\big|\\
&\quad+\big|\mathcal{J}(\tau;z_0,\beta_n) - \mathcal{J}(\tau;z_0,\beta_0)\big|\notag.
\end{align}
\indent For the first term on the right-hand side in (\ref{uscinequality}), we have
\begin{align}\label{uscinequality2}
\big|\mathcal{J}(\tau;z_n,\beta_n) - \mathcal{J}(\tau;z_0,\beta_n)\big| 
&\le \mathbb{E}^\mathbb{Q}\bigg[\int_0^\infty e^{-rt} \big|(z_0 \tilde{Z}_t)^{-\frac{1}{\gamma}} - (z_n \tilde{Z}_t)^{-\frac{1}{\gamma}}\big| dt \bigg] \notag\\
&= |z_0^{-\frac{1}{\gamma}} - z_n^{-\frac{1}{\gamma}}| \, \mathbb{E}^\mathbb{Q}\bigg[\int_0^\infty e^{-rt} \tilde{Z}_t^{-\frac{1}{\gamma}} dt \bigg]\\
&\leq |z_0^{-\frac{1}{\gamma}} - z_n^{-\frac{1}{\gamma}}| \int_0^\infty e^{-rt}e^{-\frac{1}{\gamma}(\delta-r)t}dt,\notag
\end{align} where we have set \begin{equation}\label{Tilde(Z)}
\tilde{Z}_t := \exp\Big(\int_0^t (\delta - r + \tfrac{1}{2} \tfrac{\hat{\beta}_s^2}{\sigma^2}) ds - \int_0^t \frac{\hat{\beta}_s}{\sigma} dW_s^{\mathbb{Q}} \Big),
\end{equation} for $t\geq 0$, and we have used (\ref{E(Z)ineq}) upon noticing that $\tilde{Z}_t=\frac{\hat{Z}_t}{z}$. Given that the integral on the right-hand side of (\ref{uscinequality2}) is finite as $r+\frac{1}{\gamma}(\delta-r)>0$ because $\gamma>1$, we find 
\begin{equation}\label{firstineq}
\lim_{(z_n,\beta_n)\to(z_0,\beta_0)}\big|\mathcal{J}(\tau;z_n,\beta_n) - \mathcal{J}(\tau;z_0,\beta_n)\big|= 0.
\end{equation}\\
\indent For the second term in (\ref{uscinequality}), Fubini-Tonelli Theorem yields
\begin{align}\label{uscdct}
\big|\mathcal{J}(\tau;z_0,\beta_n) - \mathcal{J}(\tau;z_0,\beta_0)\big| 
&\leq \int_0^\infty e^{-rt} \, \mathbb{E}^\mathbb{Q}\Big[\big|(\hat{Z}_t^{z_0,\beta_0})^{-\frac{1}{\gamma}} - (\hat{Z}_t^{z_0,\beta_n})^{-\frac{1}{\gamma}}\big|\Big] dt.
\end{align}
By the same arguments used in the proof of Proposition \ref{wellposedness} (see Appendix \ref{proofwellposedness}), one can show that
\[
e^{-rt} \, \mathbb{E}^\mathbb{Q}\Big[\big|(\hat{Z}_t^{z_0,\beta_0})^{-\frac{1}{\gamma}} - (\hat{Z}_t^{z_0,\beta_n})^{-\frac{1}{\gamma}}\big|\Big] \le 2 z_0^{-\frac{1}{\gamma}} e^{-(r + \frac{1}{\gamma} (\delta-r))t},
\]
with
\[
\int_0^\infty 2 z_0^{-\frac{1}{\gamma}} e^{-(r + \frac{1}{\gamma} (\delta-r))t} dt < \infty,
\]
given that $r+\frac{1}{\gamma}(\delta-r)>0$ by $\gamma>1$. Hence, an application of the Dominated Convergence Theorem in (\ref{uscdct}) yields
\begin{align}\label{cont.star}
&\lim_{(z_n,\beta_n)\to (z_0, \beta_0)} \int_0^\infty e^{-rt} \mathbb{E}^\mathbb{Q}\Big[\big|(\hat{Z}_t^{z_0,\beta_0})^{-\frac{1}{\gamma}} - (\hat{Z}_t^{z_0,\beta_n})^{-\frac{1}{\gamma}}\big|\Big] dt \\
&= \int_0^\infty \lim_{(z_n,\beta_n)\to (z_0, \beta_0)} e^{-rt} \mathbb{E}^\mathbb{Q}\Big[\big|(\hat{Z}_t^{z_0,\beta_0})^{-\frac{1}{\gamma}} - (\hat{Z}_t^{z_0,\beta_n})^{-\frac{1}{\gamma}}\big|\Big] dt.
\end{align}
Finally, by exploiting arguments as in the proof of Proposition \ref{wellposedness} again, we have under Assumption \ref{Novikov_OST} that 
\[
\mathbb{E}^\mathbb{Q}_{z,\beta}\big[(\hat{Z}_t^\beta)^{-p/\gamma}\big] \le z_0^{-p/\gamma} e^{-\frac{p}{\gamma} (\delta-r) t} < \infty,
\]
where $p$ is chosen such that \[1<p<\min\bigg\{\gamma, \frac{\gamma\sigma\bigg(\kappa-\frac{ \sigma_\beta}{\sigma}\bigg)}{\sigma_\beta}\bigg\},\] upon noticing that, by Assumption \ref{Novikov_OST}, we have $\gamma\sigma\bigg(\kappa-\frac{ \sigma_\beta}{\sigma}\bigg)>\sigma_\beta$. Therefore, by Vitali's Convergence Theorem and continuity of $\beta\mapsto \hat{Z}_t^{z,\beta}$, we conclude that
\[
\lim_{(z_n,\beta_n)\to (z_0,\beta_0)} \mathbb{E}^\mathbb{Q}\Big[\big|(\hat{Z}_t^{z_0,\beta_0})^{-\frac{1}{\gamma}} - (\hat{Z}_t^{z_0,\beta_n})^{-\frac{1}{\gamma}}\big|\Big] = 0,
\] which implies due to (\ref{uscdct}) and (\ref{cont.star}) that
\begin{equation}\label{secondineq}
\lim_{(z_n,\beta_n)\to(z_0,\beta_0)}\big|\mathcal{J}(\tau;z_0,\beta_n) - \mathcal{J}(\tau;z_0,\beta_0)\big|= 0.    
\end{equation} 
\indent Finally, by combining (\ref{firstineq}) and (\ref{secondineq}), we obtain
\[
\lim_{(z_n,\beta_n)\to (z_0, \beta_0)} \mathcal{J}(\tau;z_n,\beta_n) = \mathcal{J}(\tau;z_0,\beta_0).
\]
\textbf{Step 2.} Fix again $(z_0,\beta_0) \in \mathcal{O}$ and let $(z_n, \beta_n)\subseteq\mathcal{O}$ be a sequence converging to $(z_0, \beta_0)$ as $n \to \infty$. For each $(z_n,\beta_n)$, let $\tau_n$ be an $\varepsilon$-optimal stopping time for $(z_n,\beta_n)$, with $\varepsilon>0$; that is,
\[
v(z_n,\beta_n) \ge \mathcal{J}(\tau_n;z_n,\beta_n) - \varepsilon.
\]
Since $\tau_n$ is suboptimal for $(z_0,\beta_0)$, we have $v(z_0,\beta_0) \le \mathcal{J}(\tau_n;z_0,\beta_0)$.\\
\indent Defining
\[
\Delta_n := \mathcal{J}(\tau_n;z_n,\beta_n) - \mathcal{J}(\tau_n;z_0,\beta_0),
\]
it then holds
\[
|\Delta_n| \le 
|\mathcal{J}(\tau_n;z_n,\beta_n) - \mathcal{J}(\tau_n;z_0,\beta_n)| + 
|\mathcal{J}(\tau_n;z_0,\beta_n) - \mathcal{J}(\tau_n;z_0,\beta_0)|.
\]
By arguments as in Step 1 above, we have $\Delta_n \to 0$ as $n \to \infty$. Hence,
\begin{align*}
v(z_n,\beta_n) 
\ge \mathcal{J}(\tau_n;z_n,\beta_n) - \varepsilon = \mathcal{J}(\tau_n;z_0,\beta_0) + \Delta_n - \varepsilon \ge v(z_0,\beta_0) + \Delta_n - \varepsilon,
\end{align*}
which, by taking the limit as $n \to \infty$, yields
\[
\liminf_{n \to \infty} v(z_n,\beta_n) \ge v(z_0,\beta_0) - \varepsilon.
\]
Since $\varepsilon > 0$ was arbitrary, we have that $v$ is lower-semicontinuous. \\
\\
\textbf{Step 3.} Combining Step 1 and Step 2, we conclude that $v$ is continuous on $\mathcal{O}$.
\end{proof}
As it is customary in optimal stopping, we now define the continuation (waiting) and stopping regions as
\begin{equation}\label{regions}
\mathcal{W} := \{(z,\beta) \in \mathcal{O} : v(z,\beta) < 0\}, \quad
\mathcal{S} := \{(z,\beta) \in \mathcal{O} : v(z,\beta) = 0\}.
\end{equation}
By the continuity of $v$ (see Proposition \ref{uppersemi}), $\mathcal{W}$ is open and $\mathcal{S}$ is closed.  
Furthermore, the stopping time
\begin{equation}\label{deftau*}
\tau^*(z,\beta) := \inf \{ s \ge 0 : (\hat{Z}_s, \hat{\beta}_s) \in \mathcal{S} \}
\end{equation} 
is optimal (see Corollary 2.9 in Chapter 1 of \cite{peskir2006optimal}). 

\begin{prop}\label{Snonempty}
The stopping region $\mathcal{S}$ is non-empty; that is, $\mathcal{S} \neq \emptyset$.
\end{prop}

\begin{proof}
Suppose $\mathcal{S}=\emptyset$. Then for all $(z,\beta)\in\mathcal{O}$ we have by (\ref{OST_new}), the fact that $\tilde{u}'(\hat{Z}_t)=-\hat{Z}_t^{-\frac{1}{\gamma}}$ (cf.\ (\ref{legendre})), and (\ref{E(Z)ineq}) that
\begin{align*}
0>v(z,\beta)=\mathbb{E}^\mathbb{Q}_{z,\beta}\bigg[\int_0^\infty e^{-rt} \big(-\hat{Z}_t^{-\frac{1}{\gamma}} + \ell \big) dt \bigg] 
&= \frac{\ell}{r} - \mathbb{E}^\mathbb{Q}_{z,\beta}\bigg[\int_0^\infty e^{-rt} \hat{Z}_t^{-\frac{1}{\gamma}} dt \bigg] \\
&\ge \frac{\ell}{r} - z^{-\frac{1}{\gamma}} \int_0^\infty e^{-(\frac{1}{\gamma}(\delta-r) + r) t} dt \\
&= \frac{\ell}{r} - \frac{z^{-\frac{1}{\gamma}}}{\frac{1}{\gamma}(\delta-r) + r}.
\end{align*}
However, the last expression is strictly positive if 
\begin{equation}\label{nonemptyz}
z > \Big(\frac{\ell}{r} \big(\frac{1}{\gamma} (\delta-r) + r \big)\Big)^{-\gamma} > 0,
\end{equation}
which gives the desired contradiction.
\end{proof}
\subsection{Optimal Stopping Boundary and Regularity of the Value Function}\label{freeboundarysection}
In this section, we establish the existence of a lower-semicontinuous optimal stopping boundary (free boundary) that separates continuation and stopping regions and prove further regularity of the value function of the optimal stopping problem.\\
\indent We first show that the boundary $\partial\mathcal{W}$ can be represented by a function 
$z^*:\mathbb{R}\rightarrow [\ell^{-\gamma},\infty]$ and establish connectedness of $\mathcal{W}$ and $\mathcal{S}$
with respect to the $z$-variable.

\begin{lemma}\label{freeboundaryexistence}
There exists a free boundary $z^*:\mathbb{R}\rightarrow [0,\infty]$ such that
\begin{equation}\label{Sinproof}
\mathcal{S}=\{(z,\beta)\in\mathcal{O}: z\geq z^*(\beta)\}.
\end{equation}
Moreover, we have
\begin{equation}\label{z*upperbound}
0<\ell^{-\gamma}\leq z^*(\beta)
\quad\text{for all }\beta\in\mathbb{R},
\end{equation} and $\beta\mapsto z^*(\beta)$ is lower-semicontinuous.
\end{lemma}

\begin{proof}
By Proposition~\ref{monotonicity}, we have that $z\mapsto v(z,\beta)$ is nondecreasing for all $\beta\in\mathbb{R}$.  
Hence, by defining
\begin{equation}\label{defz*}
z^*(\beta):=\inf\{z>0: v(z,\beta)\geq 0\}
\end{equation}
(with the convention $\inf\emptyset=+\infty$), it follows from (\ref{regions}) that
\[
\mathcal{S}=\{(z,\beta)\in\mathcal{O}: z\geq z^*(\beta)\},\quad \text{and}\quad \mathcal{W}=\{(z,\beta)\in\mathcal{O}: z<z^*(\beta)\}.
\]
For the lower bound of $z^*$, we have from (\ref{OST_new}) and $\tilde{u}'(z)=-z^{-\frac{1}{\gamma}}$ that, if $-z^{-\frac{1}{\gamma}}+\ell < 0$, it is optimal to continue, as stopping immediately yields $0$ while continuing for a short time yields a negative contribution to the cost functional. Hence, we obtain \begin{align*}&\{(z,\beta)\in\mathcal{O}\mid-z^{-\frac{1}{\gamma}}+\ell<0\}\subseteq \mathcal{W}\iff \{(z,\beta)\in\mathcal{O}\mid-z^{-\frac{1}{\gamma}}+\ell\geq0\}\supseteq \mathcal{S}\\
&\iff \{(z,\beta)\in\mathcal{O}\mid z\geq\ell^{-\gamma}\}\supseteq \mathcal{S}.\end{align*} It then follows from (\ref{defz*}) that $z^*(\beta)\geq \ell^{-\gamma}>0,\;\text{for all}\;\beta\in\mathbb{R}.$\\
\indent Finally, lower-semicontinuity of $z^*$ is due to the fact that (\ref{Sinproof}) is closed thanks to (\ref{regions}) and continuity of $v$ (cf. Proposition \ref{uppersemi}).
\end{proof}
We continue by proving (local) Lipschitz continuity of $v$ and probabilistic representations of its weak derivatives.
\begin{prop}\label{ProbRepr}
The value function $v$ is (locally) Lipschitz continuous on $\mathcal{O}$. Moreover, its weak derivatives, denoted by $v_z$ and $v_\beta$, admit the following probabilistic representations:
\begin{equation}\label{V_z}
v_z(z,\beta) = \mathbb{E}^\mathbb{Q}_{z,\beta}\Big[\int_0^{\tau^*} e^{-rt} \frac{1}{\gamma} z^{-1} \hat{Z}_t^{-\frac{1}{\gamma}} dt \Big],
\end{equation}
and
\begin{equation}\label{V_beta}
v_\beta(z,\beta)=\mathbb{E}^\mathbb{Q}_{z,\beta}\left[\int_0^{\tau^*} e^{-rt}
\left(\frac{1}{\gamma}(\hat Z_t)^{-\frac{1}{\gamma}}
\left(\int_0^t e^{-as}\frac{\hat{\beta}_s^{\beta}}{\sigma^2}ds
     -\frac{1}{\sigma}\int_0^t e^{-as}dW_s^{\mathbb{Q}}\right)\right)dt\right],
\end{equation} where $a := \kappa - \frac{\sigma_\beta}{\sigma}>0$ by Assumption \ref{assumption_novikov}.
\end{prop}

\begin{proof}
See Appendix \ref{proofprobrepr}.
\end{proof}
Standard results in optimal stopping theory (cf.\ Chapter 3 in \cite{peskir2006optimal}) together with the previous findings imply that the couple $(v,z^*)$ satisfies the free boundary problem

\begin{numcases}{} 
  \mathcal{L}v-rv+\tilde{u}'(z)+\ell=0 & on $0 < z < z^*(\beta)$ \label{freeboundary1} \\
  v = 0 & on $z \geq z^*(\beta)$, \notag
\end{numcases}
where $\mathcal{L}$ is the infinitesimal generator of the process $(\hat{Z}_t,\hat{\beta}_t)_t$ such that
\begin{equation}\label{generator}(\mathcal{L}v)(z,\beta)=\frac{1}{2}\frac{\beta^2}{\sigma^2}(2zv_z+z^2v_{zz})+\frac{1}{2}\sigma_\beta^2 v_{\beta\beta}+(\delta-r)zv_z+\kappa(\overline{\beta}-\beta)v_\beta
+\frac{\beta}{\sigma}\sigma_\beta\,(v_{\beta}+zv_{z\beta}),
\end{equation} 
and the PDE above is intended in the sense of Schwartz distributions (see Corollary 5 in \cite{peskir2025weak}). In Proposition \ref{C^2W} below we will show that $v$ actually solves (\ref{freeboundary1}) in the classical sense. In order to achieve this, we need the following result.

\begin{lemma}\label{strongfeller}
The process $(\hat{Z}_t,\hat{\beta}_t)_t$, given by (\ref{hatZ}) and (\ref{hatbeta}), is strong Feller.
\end{lemma}
\begin{proof}
Recall the dynamics of the process $(\hat{Z}_t,\hat{\beta}_t)_t$ (cf.\;(\ref{hatZ}) and (\ref{hatbeta})),
\[
d\hat{Z}_t=-\frac{\hat{\beta}_t}{\sigma}\hat{Z}_t\, dW_t^{\mathbb{Q}}+\hat{Z}_t\bigg(\delta-r+\frac{\hat{\beta}_t^2}{\sigma^2}\bigg)dt,\quad t>0,\quad \hat{Z}_0=z>0,
\]
and
\[
d\hat{\beta}_t=-\sigma_\beta\, dW_t^{\mathbb{Q}}
+\bigg(\kappa(\overline{\beta}-\hat{\beta}_t)+\frac{\hat{\beta}_t}{\sigma}\sigma_\beta\bigg)dt,\quad t>0,\quad \hat{\beta}_0=\beta\in \mathbb{R}.
\]
\indent Notice that (\ref{generator}) is not uniformly elliptic. As a matter of fact, denoting by $\Sigma(z,\beta)$ the diffusion matrix associated to (\ref{hatZ})-(\ref{hatbeta}) one has $\det(\Sigma\Sigma^T(z,\beta))=0$. We therefore now check that the process $(\hat{Z}_t,\hat{\beta}_t)_t$ satisfies the so-called Hörmander’s condition (see, e.g., condition (H) in Section 2.3.2 in \cite{nualart2006malliavin}). This implies that the second-order infinitesimal generator $\mathcal{L}$ of $(\hat{Z}_t,\hat{\beta}_t)_t$ as in (\ref{generator}) is hypoelliptic and therefore that $(\hat{Z}_t,\hat{\beta}_t)_t$ is a strong Feller process (see Proposition 4 in \cite{ernst2024quickest}).\\
\indent For $(z,\beta)\in\mathcal{O}$, arbitrary but fixed, we then define the Stratonovich-corrected drift vector field and the diffusion vector field driven by $W^\mathbb{Q}$ as follows:
\[
V^{(0)}(z,\beta)=
\begin{pmatrix}
\dfrac{z\big(\beta^2-\sigma\sigma_\beta+2\sigma^2(\delta-r)\big)}{2\sigma^2}\\[8pt]
\dfrac{\beta\sigma_\beta+\kappa\sigma(\overline{\beta}-\beta)}{\sigma}
\end{pmatrix},
\qquad
V^{(1)}(z,\beta)=
\begin{pmatrix}
-\dfrac{\beta}{\sigma}z\\[4pt]
-\sigma_\beta
\end{pmatrix}.
\]
\noindent
\indent We denote by $DV^{(i)}$ the Jacobian matrix of the vector $V^{(i)}$, $i=0,1$. 
The Lie bracket between $V^{(0)}$ and $V^{(1)}$, denoted by $[V^{(0)},V^{(1)}]$, is such that
\[
[V^{(0)},V^{(1)}](z,\beta)
:=DV^{(1)}V^{(0)}(z,\beta)-DV^{(0)}V^{(1)}(z,\beta)
=
\begin{pmatrix}
-\dfrac{z\kappa(\overline{\beta}-\beta)}{\sigma}\\[6pt]
-\dfrac{\sigma_\beta}{\sigma}\big(\kappa\sigma-\sigma_\beta\big)
\end{pmatrix}.
\]
Since $\det(V^{(1)},[V^{(0)},V^{(1)}])$ is not necessarily non-zero, we proceed to the next bracket. Therefore, we compute
\[
[V^{(1)},[V^{(0)},V^{(1)}]](z,\beta)
= D[V^{(0)},V^{(1)}]V^{(1)}(z,\beta)-DV^{(1)}[V^{(0)},V^{(1)}]
=
\begin{pmatrix}
\dfrac{-\sigma_\beta z}{\sigma^2}\big(2\kappa\sigma-\sigma_\beta\big)\\[6pt]
0
\end{pmatrix},
\]
which implies
\[
\det(V^{(1)},[V^{(1)},[V^{(0)},V^{(1)}]])
=-\dfrac{\sigma_\beta^2 z}{\sigma^2}\big(2\kappa\sigma-\sigma_\beta\big)< 0,
\]
where the last inequality is due to \[2\kappa\sigma-\sigma_\beta>0,\] in which Assumption \ref{assumption_novikov} has been used. Hence, $V^{(1)}$ and $[V^{(1)},[V^{(0)},V^{(1)}]]$ are linearly independent and thus span $\mathcal{O}$.
Hörmander’s condition is therefore verified and we conclude that the process $(\hat{Z}_t,\hat{\beta}_t)_t$ is indeed strong Feller.   
\end{proof}

\begin{prop}\label{C^2W}
$v\in C^\infty(\mathcal{W})$ and it solves  in the classical sense 
$$\mathcal{L}v-rv+\tilde{u}'(z)+\ell=0 \quad \text{on} \quad \mathcal{W},$$
where the second-order differential operator $\mathcal{L}$ is defined as in \eqref{generator}.
\end{prop}
\begin{proof}
Since the infinitesimal generator of the process $(\hat{Z}_t,\hat{\beta}_t)_t$ is hypoelliptic (satisfying H\"ormander's conditions; cf.\ the proof of Lemma \ref{strongfeller}), the drift and volatilities of $(\hat{Z}_t,\hat{\beta}_t)_t$ belong to $C^\infty(\mathcal{W})$, and $\tilde{u}'(z)+\ell\in C^\infty(\mathcal{W})$, Corollary 7 in \cite{peskir2025weak} implies that $v$ is not just a solution to (\ref{freeboundary1}) in the sense of Schwartz distributions but $ v\in C^\infty(\mathcal{W})$ and thus solves \eqref{freeboundary1} in the classical sense.    
\end{proof}

The next proposition states that the value function $v$ is not only locally Lipschitz continuous, but actually continuously differentiable. Its proof is based on an application of \cite{Jacka-Snell}, upon noticing that the H\"ormander's condition, verified in the proof of Lemma \ref{strongfeller}, gives existence of a smooth transition density for the process $(\hat{Z}_t,\hat{\beta}_t)_t$. 

\begin{prop}
    \label{prop:C1property}
One has $v\in C^1(\mathcal{O})$.
\end{prop}
\begin{proof}
An application of strong Markov property allows to write
$$v(z,\beta) = g(z,\beta) - f(z,\beta), \quad (z,\beta) \in \mathcal{O},$$
where, for any $(z,\beta) \in \mathcal{O}$, we have set 
\begin{equation}
    \label{eq:def-g}
    g(z,\beta):= \mathbb{E}^\mathbb{Q}_{z,\beta}\bigg[\int_0^{\infty} e^{-rt}(\tilde{u}'(\hat{Z}_t)+\ell )dt\bigg],
\end{equation}
and
\begin{equation}
    \label{eq:def-u}
    f(z,\beta):= \sup_{\tau}\mathbb{E}^\mathbb{Q}_{z,\beta}\Big[e^{-r\tau} g(\hat{Z}_{\tau},\hat{\beta}_{\tau})\Big].
\end{equation}
Hence, the $C^1$-property of $v$ reduces to check that for $g$ and $f$.

By the proof of Lemma \ref{strongfeller}, we know that the process  $(\hat{Z}_t,\hat{\beta}_t)_t$ satisfies the H\"ormander's condition. Given that (modulo a change of measure to remove the quadratic term in $\hat{\beta}$ appearing in the drift of the dynamics for $\hat{Z}$) the coefficients of the evolution of $(\hat{Z}_t,\hat{\beta}_t)_t$ are linear and infinitely many times differentiable (see \eqref{hatZ} and \eqref{hatbeta}), it thus follows from Theorem 9-(iii) and Remark 11 in \cite{Bally} or Theorem 2.3.3 in \cite{nualart2006malliavin}, among others, that, for any $t>0$, $(\hat{Z}_t,\hat{\beta}_t)_t$ admits a transition density that is infinitely many times differentiable in its arguments.

An application of the Dominated Convergence Theorem then shows that $g\in C^1(\mathcal{O})$. It thus remains to check the continuous differentiability of $f$. With reference to the notation in \cite{Jacka-Snell}, for any $t \geq 0$, we set $\xi_t:=(\hat{Z}_t,\hat{\beta}_t)$,
$$X_t:=e^{-rt}g(\xi_t)= \mathbb{E}^\mathbb{Q}_{z,\beta}\bigg[\int_t^{\infty} e^{-rs}(\tilde{u}'(\hat{Z}_s)+\ell )ds\,\Big|\,\mathcal{F}_t^{W,\mathbb{Q}}\bigg],$$
and we can write $X_t = M_t + A_t$,
where 
$$M_t:= \mathbb{E}^\mathbb{Q}_{z,\beta}\bigg[\int_0^{\infty} e^{-rs}(\tilde{u}'(\hat{Z}_s)+\ell )ds\,\Big|\,\mathcal{F}_t^{W,\mathbb{Q}}\bigg],\quad \text{and} \quad A_t:= - \int_0^{t} e^{-rs}(\tilde{u}'(\hat{Z}_s)+\ell )ds.$$
Notice that, since $\mathbb{E}^\mathbb{Q}_{z,\beta}\bigg[\int_0^{\infty} e^{-rs}|\tilde{u}'(\hat{Z}_s)+\ell |\;ds\bigg]<\infty$ by Proposition \ref{wellposedness}, $(M_t)_t$ is a uniformly integrable $\mathbb{F}^{W,\mathbb{Q}}$-martingale and $dA_t = dA^+_t + dA^-_t$ with
$$dA^{-}_t := -e^{-rt}(\tilde{u}'(\hat{Z}_t)+\ell)^{+}dt \quad \text{and} \quad dA^{+}_t := e^{-rt}(\tilde{u}'(\hat{Z}_t)+\ell)^{-}dt,$$
which are clearly absolutely continuous with respect to the Lebesgue measure $dm_2:=dt$. Moreover, the set $\partial \mathcal{D}$ in \cite{Jacka-Snell} reads in our case $\{(z,\beta):\, z=z^*(\beta)\}$, which has zero measure with respect to $dm_1=dzd\beta$. Finally, as already noted, the process $(\xi_t)_t=(\hat{Z}_t,\hat{\beta}_t)_t$ admits a density with respect to $dm_1$ having spatial derivatives which are (locally) uniformly continuous in $\mathcal{O} \times [t_0,t_1]$, for any $0 < t_0 < t_1 < \infty$. Hence, Corollary 7 in \cite{Jacka-Snell} holds, $f\in C^1(\mathcal{O})$, and the proof is complete.
\end{proof}
An immediate consequence of Propositions \ref{C^2W}, \ref{prop:C1property}, \eqref{generator}, and the fact that $v=0$ in the interior of $\mathcal{S}$, denoted by $\mathring{\mathcal{S}}$, is the following Corollary.

\begin{corollary}\label{corregularity}
One has $v\in C^1(\mathcal{O})\cap C^\infty(\mathcal{W}\cup\mathring{\mathcal{S}})$. Furthermore, $\frac{1}{2}\frac{\beta^2}{\sigma^2} z^2v_{zz}+\frac{1}{2}\sigma_\beta^2 v_{\beta\beta}
+\frac{ \sigma_\beta}{\sigma}\,\beta z v_{z\beta}$ admits a continuous extension to $\overline{\mathcal{W}}$.
\end{corollary}

\begin{remark}\label{c1remark}
The proof of Proposition \ref{prop:C1property} employs the probabilistic approach developed in \cite{Jacka-Snell} (and recently used in \cite{ferrari2018optimal} and \cite{LambertonTerenzi2}) to establish continuous differentiability of the value function $v$ in the optimal stopping problem in the presence of a smooth transition density for the underlying state process.

An alternative approach to $C^1$-regularity in optimal stopping was developed in \cite{de2020global}, where continuous differentiability of the value function is linked to the probabilistic regularity of points on the stopping boundary. In our setting, however, following the approach of \cite{de2020global} is challenging, as it is not clear how to prove the probabilistic regularity of the free boundary. Indeed, such a property is typically established in the literature when one can show either monotonicity or (local) Lipschitz continuity of the free boundary. However, proving these properties is highly non-trivial in our framework because the stochastic factor $(\hat{\beta}_t)_t$ acts explicitly as stochastic volatility for the state process $(\hat{Z}_t)_t$ (cf.\ (\ref{hatZ})). This intrinsic stochastic volatility makes it difficult to exploit (\ref{V_z}) and (\ref{V_beta}) to establish monotonicity of $z^*$ or to derive the uniform bounds required to prove that the free boundary is Lipschitz continuous (cf.\ \cite{MR3904414}).
\end{remark}

To conclude this section, we summarize the results obtained so far in the following theorem.

\begin{theorem}
Recall the optimal stopping problem  (cf.\ (\ref{OST_new}))
\begin{equation*}v(z,\beta):=\inf_\tau\mathbb{E}^\mathbb{Q}_{z,\beta}\bigg[\int_0^\tau e^{-rt}(\tilde{u}'(\hat{Z}_t)+\ell )dt\bigg],\end{equation*} where the procesess $(\hat{Z}_t,\hat{\beta_t})_t$ are given by (\ref{hatZ}) and (\ref{hatbeta}).\\
\indent There exists a lower-semicontinuous free boundary
\[
z^*:\mathbb{R}\rightarrow [\ell^{-\gamma},\infty]
\]
such that the stopping and continuation regions (cf.\ (\ref{regions})) are given by
\[
\mathcal{S}=\{(z,\beta)\in\mathcal{O}: z\geq z^*(\beta)\},\quad \text{and}\quad \mathcal{W}=\{(z,\beta)\in\mathcal{O}: z< z^*(\beta)\},
\] and the optimal stopping time (cf.\ (\ref{deftau*})) is given by
\begin{equation}\label{deftau*new}\tau^*(z,\beta) = \inf \{ s \ge 0 : \hat{Z}_s\geq z^*(\hat{\beta}_s) \},\quad \mathbb{Q}_{z,\beta}\text{-a.s.}\end{equation}
Additionally, $v\in C^1(\mathcal{O})\cap C^\infty(\mathcal{W}\cup\mathring{\mathcal{S}})$ and the couple $(v,z^*)$ satisfies (in the classical sense) the free boundary problem:

\[
\begin{cases}
  \mathcal{L}v - rv + \tilde{u}'(z) + \ell = 0, & \text{on } \mathcal{W}, \\
  v = 0, & \text{on } \mathcal{S}, \\
  v_z(z^*(\beta),\beta)=0=v_{\beta}(z^*(\beta),\beta), & \beta \in \mathbb{R},
\end{cases}
\]
where the infinitesimal generator $\mathcal{L}$ is given by (\ref{generator}).
\end{theorem}


\section{Back to the Primal Problem}\label{section5}
In the previous section, we characterized the solution to the auxiliary optimal stopping problem given in (\ref{OST_new}), and hence also for the problem (\ref{OSTunderP}) due to Proposition \ref{measurechange}. The following result establishes a connection between the singular control problem (\ref{SCP}) and the auxiliary optimal stopping problem (\ref{OSTunderP}) using probabilistic arguments as in \cite{baldursson1996irreversible}, \cite{callegaro2020optimal} and \cite{ferrari2018optimal}.
\begin{theorem}\label{establish}
It holds that
\begin{equation}\label{tildeVhatV}
\tilde{V}(z,\beta) = \int_0^z v(y,\beta)\,dy,\quad (z,\beta)\in \mathcal{O},
\end{equation}
where $\tilde{V}$ and $v$ are given by (\ref{tildeV}) and (\ref{OSTunderP}) (equivalently, (\ref{OSTunderQ})). Furthermore, the optimal singular control for (\ref{tildeV}) is given by \begin{equation}\label{D^*def} D_t^*=\exp(-\xi_t^*),\quad t>0,\quad D_{0^-}^*=1,
\end{equation}
where, for $(z,\beta)\in\mathcal{O}$,
\[
\xi_t^*:=\sup\left\{y\ge 0 \mid \tau^*(ze^{-y}, \beta) < t\right\},\quad t>0,\quad \xi_{0^-}^*=0,
\] and $\tau^*(z,\beta)$ is the optimal stopping time for (\ref{OSTunderP}) given in (\ref{deftau*new}).
\end{theorem}
\begin{proof}
Let us define the candidate value function $U(z,\beta) := \int_0^z v(y,\beta)\,dy$ with $v$ as in (\ref{OSTunderP}). We need to show $U(z,\beta)=\tilde{V}(z,\beta)$ for all $(z,\beta)\in\mathcal{O}$.\\
\\
\indent \textbf{Step 1.} In this step, we establish $U(z,\beta)\leq\tilde{V}(z,\beta)$ for all $(z,\beta)\in\mathcal{O}$. To that end, let $D \in \mathcal{D}$ be an arbitrary admissible singular control and we define its left-continuous inverse process $\tau^D$ as
\begin{equation}\label{tauD}
    \tau^D(\alpha) := \inf \{ t \ge 0 \mid D_t < \alpha \},\quad\alpha \in (0, 1].
\end{equation}
The process $\tau^D:=\{\tau^D(\alpha) \mid \alpha\in(0,1]\}$ has nonincreasing, left-continuous sample paths and hence it admits right-limits $\tau^D_{+}(\alpha):=\inf\{t\geq 0\mid D_t\leq \alpha\}$. The set of points $\alpha\in(0,1]$ at which $\tau^D(\alpha)(\omega)\neq\tau^D_+(\alpha)(\omega)$ is countable for a.e.\ $\omega\in\Omega$. Since $(D_t)_t$ is right-continuous and $\tau^D(\alpha)$ is the first entry time of an open set, it is an $\mathcal{F}^W_{t+}$-stopping time for any given and fixed $\alpha\in(0,1]$. However, $(\mathcal{F}^W_t)_t$ is right-continuous, hence $\tau^D(\alpha)$ is an $\mathbb{F}^W$-stopping time. For a fixed $z>0$ and $y \in (0, z]$, we consider $\alpha = \frac{y}{z}$ in the following. Upon using sub-optimality of $\tau^D(\frac{y}{z})$ for (\ref{OSTunderP}), we have
\begin{equation}\label{eq:v_inequality}
    v(y,\beta) \le \mathbb{E}_{z,\beta}\bigg[\int_0^{\tau^D(\frac{y}{z})} e^{-rt} M_t \big(\tilde{u}'(\tfrac{y}{z} Z^1_t) + \ell\big)\,dt\bigg].
\end{equation}
Integrating \eqref{eq:v_inequality} with respect to $y$ over $[0,z]$ then yields
\begin{equation*}
    U(z,\beta) = \int_0^z v(y,\beta)\,dy \le \int_0^z \mathbb{E}_{z,\beta}\bigg[\int_0^{\tau^D(\frac{y}{z})} e^{-rt} M_t \big(\tilde{u}'(\tfrac{y}{z} Z^1_t) + \ell\big)\,dt\bigg]\,dy.
\end{equation*}
By Fubini's Theorem, the previous inequality is equivalent to
\begin{equation*}
    U(z,\beta) \le \mathbb{E}_{z,\beta}\bigg[ \int_0^z\bigg(\int_0^\infty \mathbf{1}_{\{t < \tau^D(\frac{y}{z})\}} e^{-rt} M_t \big(\tilde{u}'(\tfrac{y}{z} Z^1_t) + \ell\big)\,dt\bigg) \,dy \bigg].
\end{equation*}
Since
\begin{equation*}
    t < \tau^D(\frac{y}{z}) \iff D_t \ge \frac{y}{z} \iff y \le z D_t,
\end{equation*}
we then have
\begin{equation}\label{eq:U_intermediate}
    U(z,\beta) \le \mathbb{E}_{z,\beta}\bigg[ \int_0^\infty e^{-rt} M_t \bigg( \int_0^{z D_t} \big(\tilde{u}'(\tfrac{y}{z} Z^1_t) + \ell\big)\,dy \bigg) dt \bigg].
\end{equation}
For the inner integral on the right-hand side of (\ref{eq:U_intermediate}), upon setting $w = \frac{y}{z} Z^1_t$ and noticing that, by (\ref{Y}), $Z^{D}_t =  Z^1_tD_t$ with $Z_t^1=e^{\delta t}z\mathcal{H}_t$, we have 
\begin{equation*}
    \int_0^{z D_t} \big(\tilde{u}'(\tfrac{y}{z} Z^1_t) + \ell\big)\,dy = \frac{z}{Z^1_t} \int_0^{Z^{D}_t} (\tilde{u}'(w) + \ell)\,dw = \frac{z}{Z^1_t} \big(\tilde{u}(Z^{D}_t) + \ell Z^{D}_t \big).
\end{equation*}
Substituting this back into \eqref{eq:U_intermediate} and using the relation $e^{-rt} M_t \frac{z}{Z^1_t} = e^{-\delta t}$ (cf.\ (\ref{martingale}) and (\ref{Y})), we obtain
\begin{equation*}
    U(z,\beta) \le \mathbb{E}_{z,\beta}\bigg[ \int_0^\infty e^{-\delta t} \big(\tilde{u}(Z^{D}_t) + \ell Z^{D}_t \big)\,dt \bigg].
\end{equation*}
Since $D \in \mathcal{D}$ was arbitrary, we conclude $U(z,\beta) \le \tilde{V}(z,\beta)$ for all $(z,\beta)\in\mathcal{O}$.\\
\\
\textbf{Step 2.} Let $\tau^*(z,\beta)$ denote the optimal stopping time for the problem \eqref{OSTunderP} with initial data $(z,\beta)$ which is given by (\ref{deftau*new}). Since $z\mapsto v(z,\beta)$ is nondecreasing (cf. Proposition \ref{monotonicity}), $z \mapsto \tau^*(z,\beta)$ is nonincreasing, and therefore $\eta\mapsto \tau^*(ze^{-\eta},\beta)$ is nondecreasing on $(0,\infty)$, for all $(z,\beta)\in\mathcal{O}$. Hence, we can define the process $\xi^*$ as the generalized inverse
\begin{equation}\label{xi^*def}
    \xi^*_t := \sup \left\{ \eta \ge 0 \mid \tau^*(z e^{-\eta}, \beta) < t \right\},\quad t>0,\quad \xi_{0^-}^*=0.
\end{equation} We now take $D^*_t=\exp(-\xi^*_t)$, with $(\xi_t^*)_t$ as in (\ref{xi^*def}) (cf.\ (\ref{D^*def}) as well). Optimality of $\tau^*(y,\beta)$ for the optimal stopping problem (\ref{OSTunderP}) then implies
\begin{equation*}
    v(y,\beta) = \mathbb{E}_{z,\beta}\bigg[\int_0^{\tau^*(y,\beta)} e^{-rt} M_t \big(\tilde{u}'(\tfrac{y}{z} Z^1_t) + \ell\big)\,dt\bigg].
\end{equation*}
Since we have from (\ref{xi^*def}) 
\begin{align*} D_t^* < \frac{y}{z} \iff \xi_t^*> \ln(\frac{z}{y})\iff \tau^*(ze^{-\ln(\frac{z}{y})},\beta)<t \iff \tau^*(y, \beta) < t, \end{align*}
recalling (\ref{tauD}) for $D=D^*$, we obtain
$\tau^{D^*}(\frac{y}{z}) = \tau^*(y,\beta)$ for all $y\in(0,z]$. Hence, integrating over $[0,z]$ yields
\begin{equation*}
    U(z,\beta) = \int_0^z \mathbb{E}_{z,\beta}\bigg[\int_0^{\tau^{D^*}(\frac{y}{z})} e^{-rt} M_t \big(\tilde{u}'(\tfrac{y}{z} Z^1_t) + \ell\big)\,dt\bigg]\,dy.
\end{equation*}
As before, an application of Fubini's Theorem leads to
\begin{equation*}
    U(z,\beta) = \mathbb{E}_{z,\beta}\bigg[ \int_0^z \bigg( \int_0^\infty \mathbf{1}_{\{t < \tau^{D^*}(\frac{y}{z})\}} e^{-rt} M_t \big(\tilde{u}'(\tfrac{y}{z} Z^1_t) + \ell\big)\,dt \bigg) \, dy \bigg].
\end{equation*}
Observing that $t < \tau^{D^*}(\frac{y}{z})$ is equivalent to $y \le z D^*_t$ then yields
\begin{equation}\label{UD*}
    U(z,\beta) = \mathbb{E}_{z,\beta}\bigg[ \int_0^\infty e^{-rt} M_t \bigg( \int_0^{z D^*_t} \big(\tilde{u}'(\tfrac{y}{z} Z^1_t) + \ell\big)\,dy \bigg) \, dt \bigg].
\end{equation}
For the inner integral in (\ref{UD*}) we use again the change of variables $w := \frac{y}{z} Z^1_t$, which implies 
\begin{align}\label{changeofvariables}
    \int_0^{z D^*_t} \big(\tilde{u}'(\tfrac{y}{z} Z^1_t) + \ell\big)\,dy
    &= \frac{z}{Z^1_t} \int_0^{Z^{D^*}_t} \big(\tilde{u}'(w) + \ell\big) \, dw \notag\\
    &= \frac{z}{Z^1_t} \big(\tilde{u}(Z^{D^*}_t) + \ell Z^{D^*}_t \big),
\end{align}
and, upon using $e^{-rt} M_t \frac{z}{Z^1_t} = e^{-\delta t}$ (cf.\ (\ref{martingale}) and (\ref
Y)), we obtain thanks to (\ref{UD*}) and (\ref{changeofvariables}) that
\begin{equation*}
    U(z,\beta) = \mathbb{E}_{z,\beta}\bigg[ \int_0^\infty e^{-\delta t} \big(\tilde{u}(Z^{D^*}_t) + \ell Z^{D^*}_t \big) \, dt \bigg] \geq \tilde{V}(z,\beta),
\end{equation*} where we used that $D^*\in\mathcal{D}$ since $(\xi_t^*)_t$ is nondecreasing, càdlàg, and $\mathbb{F}^W$-adapted.\\
\\
\textbf{Step 3.} Combining Step 1 and Step 2, we finally have $\tilde{V}(z,\beta)=U(z,\beta)=\int_0^z v(y,\beta) dy$ for all $(z,\beta)\in\mathcal{O}$ and that $(D^*_t)_t$ as in (\ref{D^*def}) is the optimal singular control.
\end{proof}
A direct consequence of Theorem \ref{establish} and Corollary \ref{corregularity} is the following Corollary.

\begin{corollary}\label{cor}
For $\tilde{V}$ as in (\ref{tildeV}) and $v$ as in (\ref{OSTunderP}) (equivalently, (\ref{OSTunderQ})), we have \[\tilde{V}_z(z,\beta)=v(z,\beta),\quad(z,\beta)\in\mathcal{O}.\] Consequently, it holds that $\tilde{V} \in C^1(\mathcal{O})$ with $\tilde{V}_{z} \in C^1(\mathcal{O}) \cap C^{\infty}(\mathcal{W} \cup \mathring{\mathcal{S}})$.   
\end{corollary}

The next proposition further characterizes the optimal singular control of problem (\ref{SCP}).
\begin{prop}
Recall (\ref{D^*def}). Then $D^*$ admits the representation
\begin{equation}\label{D^*}
D_t^* = \inf_{0 \le s \le t} \left( \frac{z^*(\beta_s)}{Z_s^{1}} \wedge 1 \right),\quad t\geq 0, \quad D_{0^-}^*=1.
\end{equation}
\end{prop}
\begin{proof}
By (\ref{D^*def}), $D^*_{0^-}=1$ and for any $t\geq 0$ one has \[D_t^*=\exp(-\xi_t^*),\] where $\xi_t^*=\sup\left\{y\ge 0 \mid \tau^*(ze^{-y}, \beta) < t\right\}$ and $\tau^*(z,\beta) = \inf \{ s \ge 0 : \hat{Z}^{z,\beta}_s\geq z^*(\hat{\beta}^\beta_s) \}$ (cf.\ (\ref{deftau*new})). We then obtain the following chain of equivalences for $t\geq 0$: \begin{align*}
\tau^*(z e^{-y}, \beta) < t&\iff \exists s \in [0, t]: e^{-y} Z_s^{ 1} \ge z^*(\beta_s)\iff \exists s \in [0, t]: y \le \ln(\frac{Z_s^{1}}{z^*(\beta_s)})\\
&\iff \xi_t^* = \sup_{0 \le s \le t} \left[ \ln \left( \frac{Z_s^{1}}{z^*(\beta_s)} \right) \right]^+.
\end{align*}
Finally, substituting this into $D_t^* = \exp(-\xi_t^*)$ and observing that $\exp(-(\ln x)^+) = \min(1, \frac{1}{x})$, we obtain (\ref{D^*}). \\
\indent Clearly, $D^*$ as in (\ref{D^*}) is nonincreasing and càdlàg given the lower-semicontinuity of $z^*$ (cf.\ Lemma \ref{freeboundaryexistence}). As a matter of fact, $t \mapsto D_t^*$ admits left-limits at any point since it is nonincreasing. To show that $D^*$ has right-continuous sample paths, we follow the proof of Proposition 5.8 in \cite{de2017optimal} and first notice that
\begin{equation}\label{Dcadlag}
    \liminf_{s \downarrow t} \left( \frac{z^*(\beta_s)}{Z_s^1} \wedge 1 \right) \ge \frac{z^*(\beta_t)}{Z_t^1} \wedge 1,
\end{equation}
by the lower-semicontinuity of $z^*$ (cf.\ Lemma \ref{freeboundaryexistence}) and the continuity of the state processes $(Z_t^1, \beta_t)_t$. Moreover, from (\ref{Dcadlag}) we obtain
\begin{align}\label{dcadlag2}
    \lim_{s \downarrow t} D_s^* &= D_t^* \wedge \lim_{s \downarrow t} \inf_{t < u \le s} \left( \frac{z^*(\beta_u)}{Z_u^1} \wedge 1 \right) \notag \\
    &= D_t^* \wedge \liminf_{s \downarrow t} \left( \frac{z^*(\beta_s)}{Z_s^1} \wedge 1 \right)  \\
    &\ge D_t^* \wedge \left( \frac{z^*(\beta_t)}{Z_t^1} \wedge 1 \right) = D_t^*.\notag
\end{align}
Since $\lim_{s \downarrow t} D_s^* \le D_t^*$ by the monotonicity of $t \mapsto D_t^*$, (\ref{dcadlag2}) implies right continuity.
\end{proof}
\indent  We now have all the necessary ingredients to derive the optimal controls for our primal optimization problem (cf.\ (\ref{V}))
\[
V(x,\beta)=\max_{(\pi,c)\in\mathcal{A}(x)}\mathbb{E}_{x,\beta}\bigg[\int_{0}^{\infty}e^{-\delta t}u(c_t)\,dt\bigg].
\]
In the following, when needed, we stress the dependence of the unique strong solution to (\ref{Z^D}) on the initial data $(z,\beta)\in\mathcal{O}$ and on $D\in\mathcal{D}$ by writing $Z^{z,\beta,D}$.
\begin{prop}\label{optimalstrategies}
Let $\tilde{V}$ as in (\ref{tildeV}) (see also (\ref{tildeVhatV})). We have 
\begin{equation}\label{primalstrongduality}
V(x,\beta)=\inf_{z>0}\big(\tilde{V}(z,\beta)+zx\big),\quad (x,\beta)\in \mathcal{O},
\end{equation}
and for all $(x,\beta)\in\mathcal{O}$, there exists $\hat{z}:=\hat{z}(x,\beta)>0$ such that $\tilde{V}_z(\hat{z},\beta)=-x$. Furthermore, the optimal primal controls for (\ref{V}) are given by 
\[
c_t^*=(Z_t^{\hat{z},\beta,D^*})^{-\frac{1}{\gamma}}
\qquad\text{and}\qquad
\pi_t^*=\frac{\beta_t}{\sigma^2}Z_t^{\hat{z},\beta, D^*}\tilde{V}_{zz}(Z_t^{\hat{z},\beta,D^*},\beta_t) + \frac{ \sigma_\beta}{\sigma} \tilde{V}_{z\beta}(Z_t^{\hat{z},\beta,D^*},\beta_t),\quad t\geq 0,
\]
where $(D^*_t)_t$ is the optimal singular control for problem (\ref{SCP}); that is 
\[
D_t^*=\inf_{0\le s\le t}\left( \frac{z^*(\beta_s)}{Z_s^{1}} \right)\wedge 1,\quad t\geq 0,\quad D_{0^-}^*=1.
\]
Finally, the optimal wealth process is such that 
\[
X_t^*=-\tilde{V}_z(Z_t^{\hat{z},\beta, D^*},\beta_t),\quad t\geq 0,
\]
with $(X_t^*)_{t\geq0}:=(X_t^{\pi^*,c^*})_{t\geq0}$.
\end{prop}

\begin{proof}
\textbf{Step 1.} We first establish the existence of $\hat{z}$ such that $\tilde{V}_z(\hat{z},\beta)=-x$ for all $x>0$. Notice that the minimization problem $\inf_{z>0} (\tilde{V}(z,\beta)+zx)$ is equivalent to solving 
\begin{equation}\label{minimizer}
v(z,\beta)=-x,
\end{equation}
upon using the identity $\tilde{V}_z=v$ given by Corollary \ref{cor}.  
Since $v$ is continuous (see Proposition \ref{uppersemi}), nondecreasing in $z$ (see Proposition \ref{monotonicity}), by Proposition \ref{limits} satisfies 
\[
\lim_{z\to 0}v(z,\beta)=-\infty,
\] as well as $v(z,\beta)=0$ for any $z\geq z^*(\beta)$ and $z\mapsto v(z,\beta)$ is strictly increasing on $\mathcal{W}$, there exists a unique $0<\hat{z}<z^*(\beta)$ such that $\tilde{V}_z(\hat{z},\beta)=-x$.\\
\\
\indent \textbf{Step 2.} Next, we prove that the strong duality relation (\ref{primalstrongduality}) indeed holds. Since we have already shown the weak-duality (cf.\ (\ref{weakduality})), namely 
\begin{equation}\label{wd}V(x,\beta) \leq \inf_{z>0}\bigg(\tilde{V}(z,\beta) + zx\bigg),\end{equation} it suffices to consider the reverse inequality. Recall the optimal singular control to problem $(\ref{SCP})$, given by (\ref{D^*}), and define $c_t^*:=(Z_t^{z,\beta,D^*})^{-\frac{1}{\gamma}}$ as the candidate optimal consumption plan. Then, we set \[
\chi(z):=\mathbb{E}_{z,\beta}\bigg[\int_0^\infty \mathcal{H}_tD_t^*(c_t^*-\ell)\,dt\bigg],
\] with the aim of showing that actually $\chi(z)=-\tilde{V}_z(z,\beta)$ for all $(z,\beta)\in\mathcal{O}$. To that end, we fix $(z,\beta)\in\mathcal{O}$ and $\varepsilon>0$ (small enough), and note that $D^*$ is independent of $\varepsilon$. Since $D^*$ is suboptimal for $\tilde{V}(z+\varepsilon,\beta)$, we find
\begin{align*}
\tilde{V}(z+\varepsilon,\beta) - \tilde{V}(z,\beta) &\le \mathbb{E}\Big[\int_0^\infty e^{-\delta t} \Big( \big( \tilde{u}(Z_t^{z+\varepsilon,\beta, D^*}) - \tilde{u}(Z_t^{z,\beta, D^*}) \big) + \ell \big( Z_t^{z+\varepsilon,\beta, D^*} - Z_t^{z,\beta, D^*} \big) \Big) dt \Big] \\
&= \Big( (z+\varepsilon)^{-\frac{1-\gamma}{\gamma}} - z^{-\frac{1-\gamma}{\gamma}} \Big) \mathbb{E}\Big[\int_0^\infty e^{-\delta t} \tilde{u}\Big(Z_t^{1,\beta, D^*}\Big) dt \Big]\\
&\quad+ \varepsilon \, \ell \, \mathbb{E}\Big[\int_0^\infty e^{-\delta t} Z_t^{1,\beta, D^*} dt \Big].
\end{align*}
Dividing by $\varepsilon$ and letting $\varepsilon \to 0$ gives
\begin{equation}\label{limsup_z_tilde}
\limsup_{\varepsilon \to 0} \frac{\tilde{V}(z+\varepsilon,\beta) - \tilde{V}(z,\beta)}{\varepsilon} \le \mathbb{E}\Big[\int_0^\infty e^{-\delta t} \Big( -(Z_t^{z,\beta,D^*})^{-\frac{1}{\gamma}} + \ell \Big) Z_t^{1,\beta, D^*} dt \Big].
\end{equation}
A symmetric argument applied to $\tilde{V}(z,\beta) - \tilde{V}(z-\varepsilon,\beta)$, by using now that $D^{*}$ is suboptimal for the problem starting at $(z-\varepsilon,\beta)$, gives 
\begin{equation}\label{liminf_z_tilde}
\liminf_{\varepsilon \to 0} \frac{\tilde{V}(z,\beta) - \tilde{V}(z-\varepsilon,\beta)}{\varepsilon} \ge \mathbb{E}\Big[\int_0^\infty e^{-\delta t} \Big( -(Z_t^{z,\beta,D^*})^{-\frac{1}{\gamma}} + \ell \Big) Z_t^{1,\beta, D^*} dt \Big].
\end{equation}
Upon using that $\tilde{V}_z$ exists and is continuous (as $\tilde{V}_z=v$), (\ref{limsup_z_tilde}) and (\ref{liminf_z_tilde}) imply \[\tilde{V}_z(z,\beta)=\mathbb{E}\Big[\int_0^\infty e^{-\delta t} \Big( -(Z_t^{z,\beta,D^*})^{-\frac{1}{\gamma}} + \ell \Big) Z_t^{1,\beta, D^*} dt \Big].\] 
Since now $(Z_t^{z,\beta,D^*})^{-\frac{1}{\gamma}}=c_t^*$ and $e^{-\delta t}Z_t^{1,\beta, D^*}=\mathcal{H}_tD_t^*$ (cf.\ (\ref{Y})), we have \[-\tilde{V}_z(z,\beta)=\mathbb{E}_{z,\beta}\bigg[\int_0^\infty \mathcal{H}_tD_t^*(c_t^*-\ell)\,dt\bigg]=\chi(z).\] In particular, for $\hat{z}$ as in Step 1, we have \begin{equation}\label{chi1}-\tilde{V}_z(\hat{z},\beta)=\mathbb{E}_{\hat{z},\beta}\bigg[\int_0^\infty \mathcal{H}_tD_t^*(c_t^*-\ell)\,dt\bigg]=\chi(\hat{z}),\end{equation} which, combined with $x=-\tilde{V}_z(\hat{z},\beta)$, yields $x=\chi(\hat{z})$.\\
\indent The previous findings and (\ref{Y}) yield the following chain of equations 
\begin{align*}
\hat{z}\mathbb{E}_{\hat{z},\beta}\bigg[\int_0^\infty \mathcal{H}_tD_t^*(c_t^*-\ell)\,dt\bigg]
&=\mathbb{E}_{\hat{z},\beta}\bigg[\int_0^\infty e^{-\delta t}Z_t^{D^*}(c_t^*-\ell)\,dt\bigg]\\
&=\mathbb{E}_{\hat{z},\beta}\bigg[\int_0^\infty e^{-\delta t}\big(u(c_t^*)-\tilde{u}(Z_t^{D^*})\big)\,dt\bigg]
     -\mathbb{E}_{\hat{z},\beta}\bigg[\int_0^\infty e^{-\delta t}\ell Z_t^{D^*}\,dt\bigg].
\end{align*} By Proposition \ref{prop_staticbudget}-(2), there exists an investment strategy $\pi^*$ such that $(\pi^*,c^*)\in\mathcal{A}(x)$ thanks to $\chi(\hat{z})=x$. Hence, by using (\ref{wd}), we obtain
\begin{align}\label{chainofineq}
\tilde{V}(\hat{z},\beta)+\hat{z}x
&=\mathbb{E}_{\hat{z},\beta}\bigg[\int_0^\infty e^{-\delta t}u(c_t^*)\,dt\bigg]\le \sup_{(\pi,c)\in\mathcal{A}(x)}\mathbb{E}_{x,\beta}\bigg[\int_0^\infty e^{-\delta t}u(c_t)\,dt\bigg]\notag\\
&\le \inf_{z>0} \big(\tilde{V}(z,\beta)+zx\big)\le \tilde{V}(\hat{z},\beta)+\hat{z}x.
\end{align}
This in turn yields the strong duality
\[
V(x,\beta)=\inf_{z>0}(\tilde{V}(z,\beta)+zx)
=\tilde{V}(\hat{z},\beta)+\hat{z}x.
\]
\indent \textbf{Step 3.} In this step, we derive the optimal primal controls associated to the stochastic control problem (\ref{V}). It follows from (\ref{chainofineq}) that
\[
V(x,\beta)=\mathbb{E}_{\hat{z},\beta}\bigg[\int_0^\infty e^{-\delta t}u(c_t^*)\,dt\bigg],
\]
so that $c_t^*=(Z_t^{\hat{z},\beta,D^*})^{-\frac{1}{\gamma}}$ is optimal. Moreover, by the strong Markov property and $x=-\tilde{V}_z(\hat{z},\beta)$, we have
\[
X_t^*=-\tilde{V}_z(Z_t^{\hat{z},\beta,D^*},\beta_t).
\]
Recalling the regularity of $\tilde{V}_z = v$ as in Corollary \ref{corregularity}, noticing that $\mathbb{P}((Z_t^{\hat{z},\beta,\D^*},\beta_t)\in \mathcal{W})=1$ for all $t\geq 0-$, and that the support of the random measure $dD^*_{\cdot}$ is $\{t\geq 0:\, \tilde{V}_{zz}(z^*(\beta_t),\beta_t)=0\}$ (due to $\tilde{V}_{zz}=v_z$ and Proposition \ref{prop:C1property}), an application of Itô-Meyer’s formula leads to
\begin{align}\label{comparingdiffusion}
dX_t^*
=&-\tilde{V}_{zz}(Z_t^{\hat{z},\beta,D^*},\beta_t)\left[(\delta-r)Z_t^{\hat{z},\beta,D^*}\,dt-\frac{\beta_t}{\sigma}Z_t^{\hat{z},\beta,D^*}\,dW_t\right]\\
 &-\tilde{V}_{z\beta}(Z_t^{\hat{z},\beta,D^*},\beta_t)\left[\kappa(\overline{\beta}-\beta_t)\,dt-\sigma_\beta\,dW_t\right]\notag\\
&-\frac12\frac{\beta_t^2}{\sigma^2}(Z_t^{\hat{z},\beta,D^*})^2\tilde{V}_{zzz}(Z_t^{\hat{z},\beta,D^*},\beta_t)\;dt\notag\\
 &-\frac12\sigma_\beta^2\tilde{V}_{z\beta\beta}(Z_t^{\hat{z},\beta,D^*},\beta_t)\,dt
 -\frac{\beta_t}{\sigma}Z_t^{\hat{z},\beta,D^*}\sigma_\beta\tilde{V}_{zz\beta}(Z_t^{\hat{z},\beta,D^*},\beta_t)\,dt\notag,\quad t\geq 0.
\end{align}
Comparing the $dW_t$ terms in (\ref{wealthdynamics}) and (\ref{comparingdiffusion}) finally gives
\[
\pi_t^*=\frac{\beta_t}{\sigma^2}Z_t^{\hat{z},\beta,D^*}\tilde{V}_{zz}(Z_t^{\hat{z},\beta,D^*},\beta_t) + \frac{ \sigma_\beta}{\sigma} \tilde{V}_{z\beta}(Z_t^{\hat{z},\beta,D^*},\beta_t),\quad t\geq 0,
\]
which completes the proof.
\end{proof}
\section{Numerical Illustrations}\label{numerics}
In this section, we provide numerical illustrations of our results. In order to illustrate the singular control mechanism and how $(D^*_t)_t$ ensures the nonnegativity of the wealth process, in Figure \ref{fig:simulation}, we simulate the optimal primal state process $(X_t^{\pi^*,c^*})_t$, the optimal dual state process $(Z^{\hat{z},D^*}_t)_t$, and the optimal singular control $(D^*_t)_t$. For this simulation, we fix the parameters: $r = 0.03$, $\delta = 0.04$, $\ell=0.6$, $\gamma = 1.5$, $\kappa = 0.25$, $\bar{\beta} = 0.05$, $\sigma_{\beta} = 0.03$, and $\sigma = 0.18$.\\
\indent Figure \ref{fig:sim_A} depicts a path of the controlled dual state process $(Z^{\hat{z},D^*}_t)_t$ (black line) and the (moving stochastic) free boundary $(z^*(\beta_t))_t$ (red line). Observe that each time the dual state hits the free boundary, it is reflected and pushed downward. This action is driven by the singular control $(D^*_t)_t$. Indeed, from Figure \ref{fig:sim_B}, we see that each time the free boundary is reached, the decreasing singular control jumps to keep the process below the boundary. Crucially, in Figure \ref{fig:sim_C}, we see that the points where the dual state touches the free boundary correspond exactly to the wealth process hitting zero. At these points, the wealth is reflected upward, ensuring that the no-borrowing constraint is satisfied.
\begin{figure}[h]
    \centering
    
    \begin{subfigure}[t]{0.48\textwidth}
        \centering
        \includegraphics[width=\textwidth]{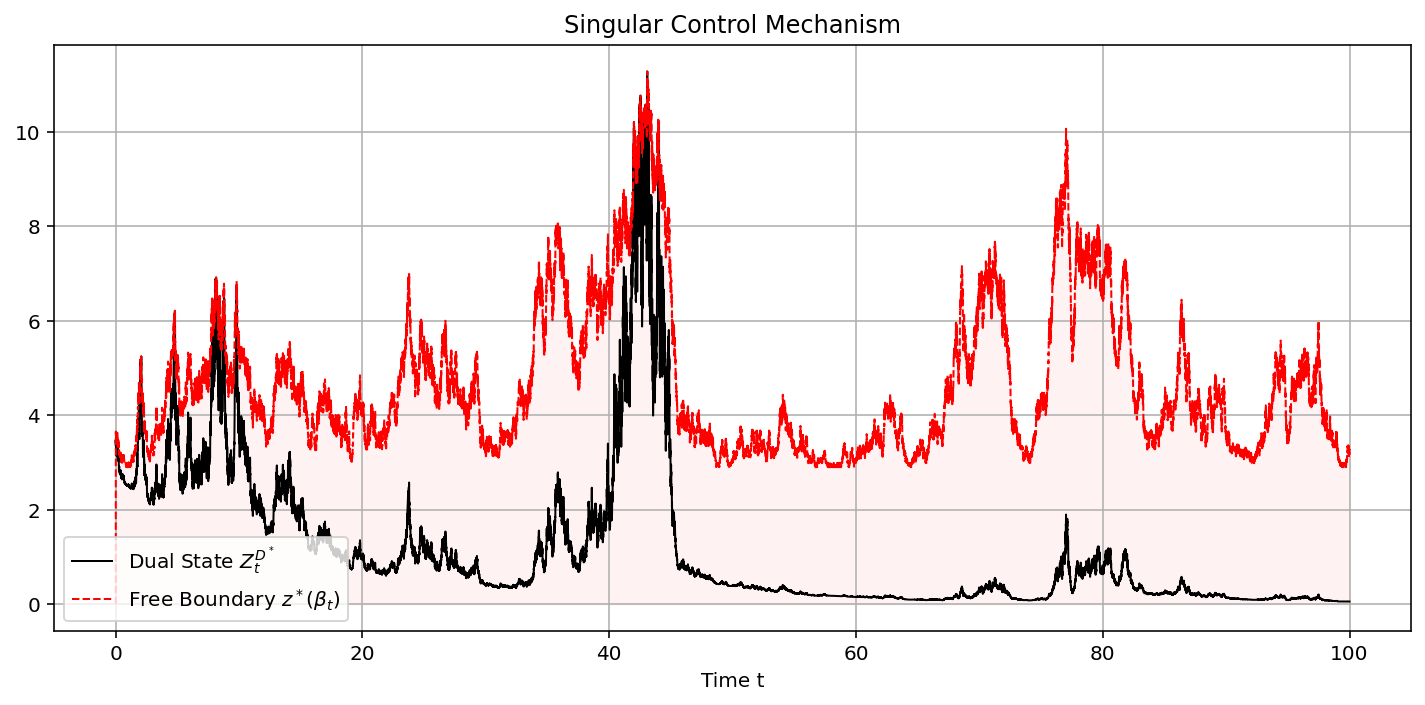}
        \caption{Simulation of dual state $(Z^{*}_t)_t$.}
        \label{fig:sim_A}
    \end{subfigure}
    \hfill 
    \begin{subfigure}[t]{0.48\textwidth}
        \centering
        \includegraphics[width=\textwidth]{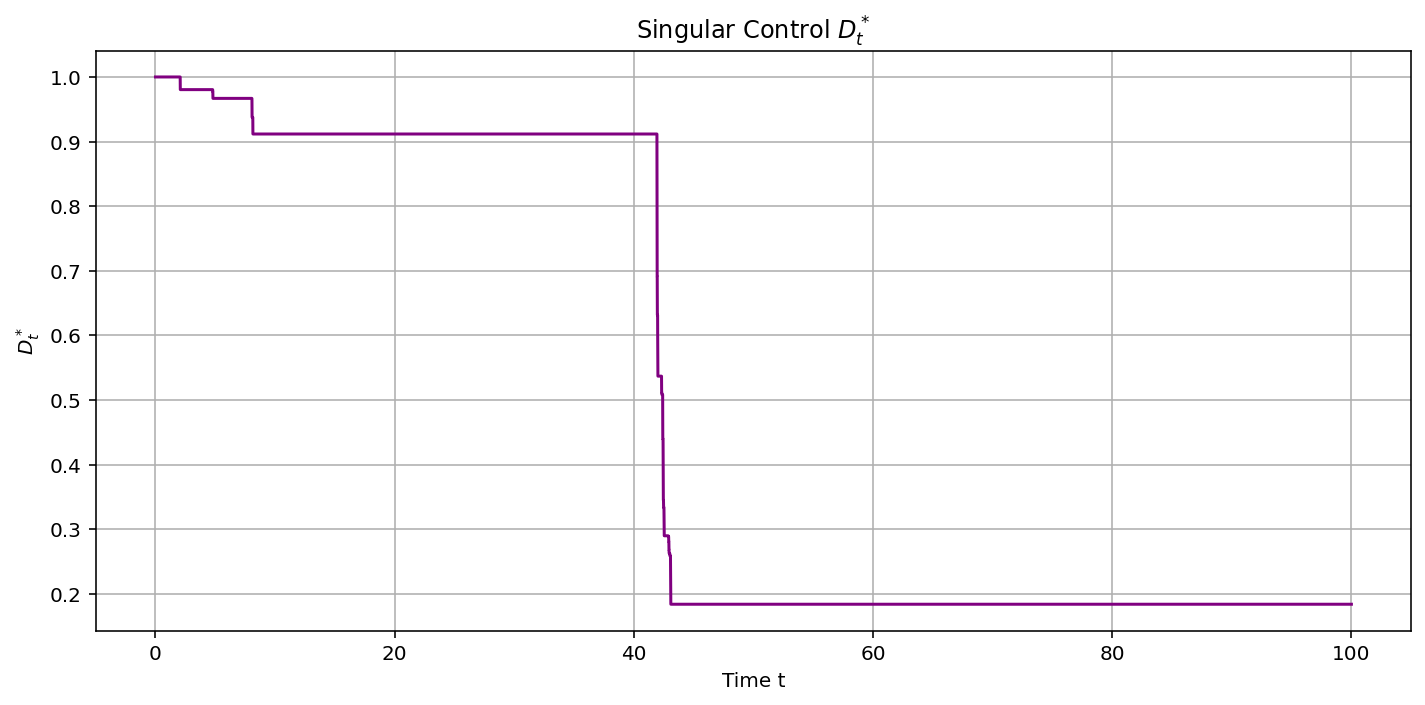}
        \caption{Simulation of optimal singular control $(D^*_t)_t$.}
        \label{fig:sim_B}
    \end{subfigure}

    \vspace{0.5cm}

    \begin{subfigure}[t]{0.48\textwidth}
        \centering
        \includegraphics[width=\textwidth]{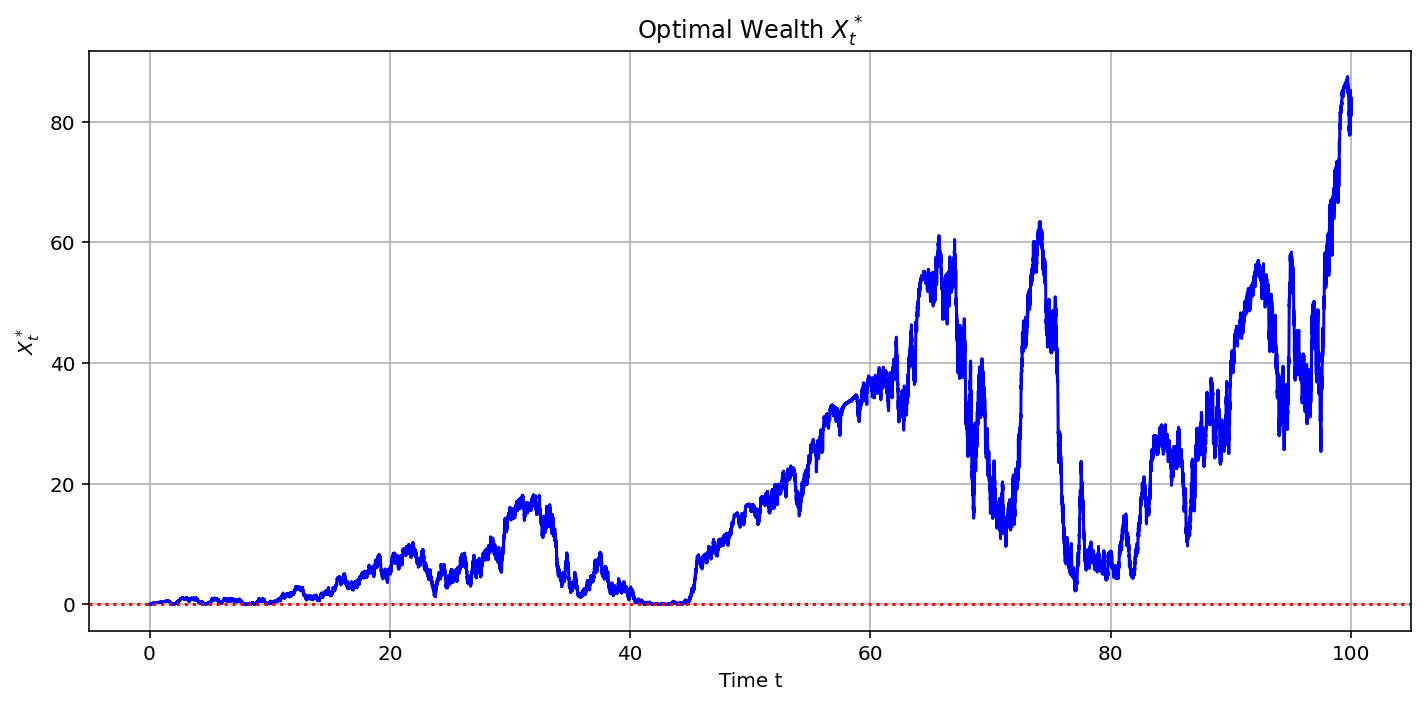}
        \caption{Simulation of optimal wealth $(X^*_t)_t$.}
        \label{fig:sim_C}
    \end{subfigure}
    
    \caption{Simulation of optimal state processes.}
    \label{fig:simulation}
\end{figure}\\
\\
One of the main contributions of this paper is the inclusion of the stochastic factor $(\beta_t)_t$, representing the expected excess return of the risky asset. In contrast to the standard Merton model, where $(\beta_t)_t$ is constant, we model it as a stochastic process. Figure \ref{fig:comparison} compares the optimal wealth trajectories under these two frameworks.\\
\indent For this comparison, we fix the parameters: $r = 0.03$, $\delta = 0.04$, $\ell = 0.6$, $\gamma = 1.5$, $\bar{\beta} = 0.05$, and $\sigma = 0.18$. In the stochastic case, we additionally set $\kappa=0.25$ and $\sigma_\beta=0.03$, while in the constant case we set $\kappa=0=\sigma_\beta$ and $\beta_0=\bar{\beta}$ to ensure that $(\beta_t)_t \equiv \bar{\beta}$. In both cases, we simulate $10,000$ paths of the Brownian motion $(W_t)_t$, and plot the optimal wealth as the average over these paths. The blue line depicts the wealth of the agent who assumes a stochastic $(\beta_t)_t$, while the red line represents the wealth of the agent who assumes a constant $(\beta_t)_t$. In the following, we denote by \emph{Stochastic Agent} the agent with stochastic $(\beta_t)_t$, and by \emph{Constant Agent} the agent with constant $(\beta_t)_t$.\\
\indent We observe that the Stochastic Agent systematically accumulates greater wealth. This outcome is driven by a robust economic mechanism: The Stochastic Agent’s strategic asset allocation, namely increasing exposure to the risky asset when expected returns are high (that is, when $\beta_t$ is high), generates substantial excess returns. As a result, the Stochastic Agent’s wealth exceeds that of the Constant Agent. In contrast, the Constant Agent lacks this flexibility and is confined to a rigid strategy that systematically fails to exploit time-varying risk premia, leading to inferior wealth accumulation.\\
    \begin{figure}[h]
    \centering
    \includegraphics[width=0.75\textwidth]{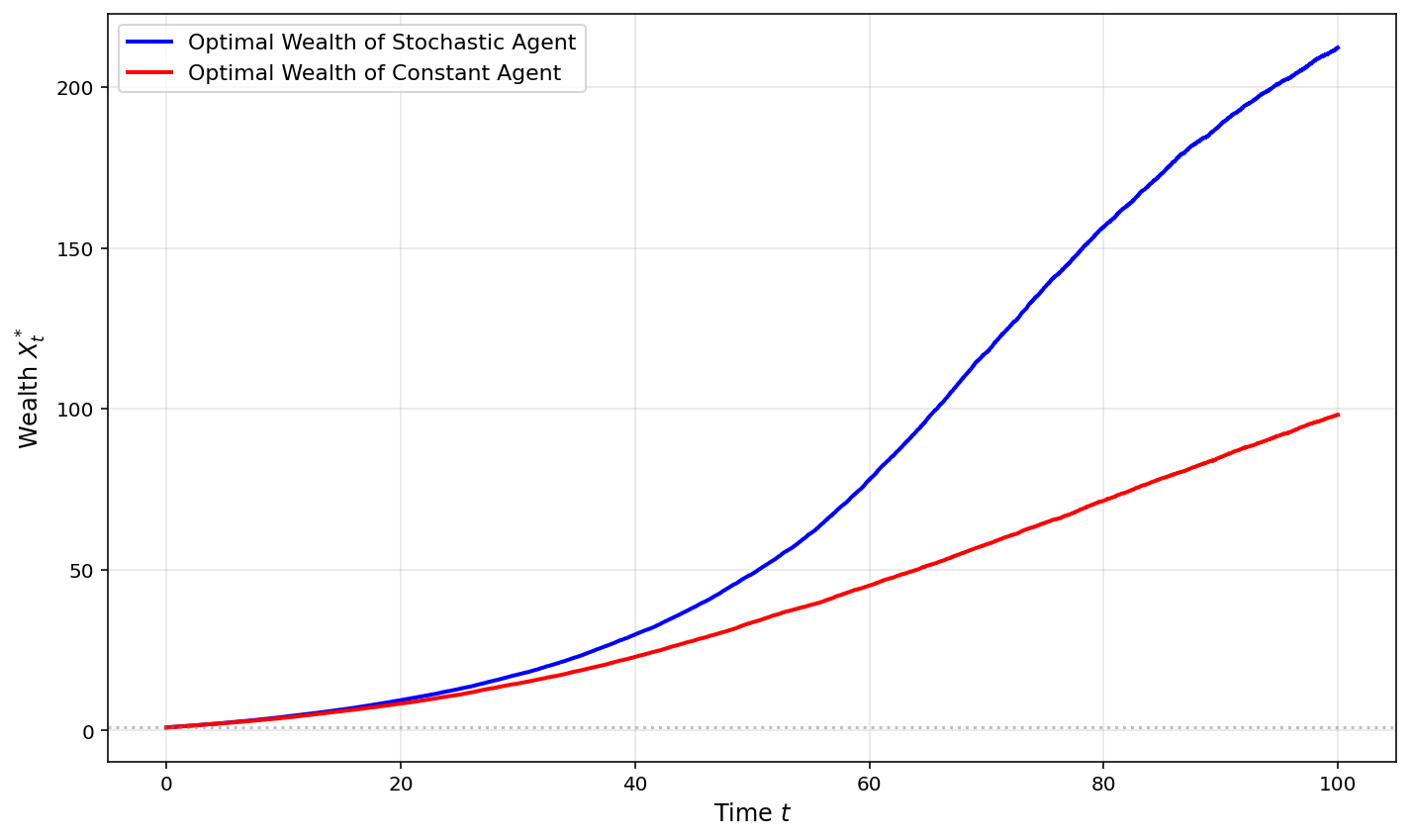}
    \caption{Comparison of the average optimal wealth trajectories for the Stochastic Agent and the Constant Agent ($(\beta_t)_t \equiv \bar{\beta}$).}
    \label{fig:comparison}
\end{figure}
\indent At the same time, the initial advantage builds up only gradually, so that the two wealth paths remain relatively close at the beginning of the horizon. Economically, under perfect negative correlation, adverse contemporaneous shocks to the risky asset are accompanied by improvements in future investment opportunities. This hedging effect dampens the short-run urgency of aggressive portfolio adjustment and smooths the initial wealth dynamics. Hence, although the Stochastic Agent is still able to exploit the time-varying investment opportunity set and eventually achieves substantially higher wealth, the momentum of this outperformance is weaker at the start. Over time, however, the cumulative benefits of adapting consumption and portfolio decisions to the stochastic factor $(\beta_t)_t$ become increasingly important, leading to a persistent and widening wealth advantage over the Constant Agent.\\
\\
\indent Next, we illustrate the optimal control strategies $(\pi^*_t)_t$ and $(c^*_t)_t$ given by Proposition \ref{optimalstrategies} for different parameter choices. To visualize the controls effectively, we plot the optimal strategies as functions of wealth $X_t$, while fixing $\beta_t$ at its equilibrium level $\bar{\beta}$. It is important to note that, while the stochastic factor is held constant in these figures, the policy functions and underlying value functions are derived from the full dynamic model. Thus, the agent explicitly accounts for the stochastic evolution of $(\beta_t)_t$ throughout the optimization, and the plots represent a cross-section of this dynamic strategy.\\
\indent In Figure \ref{laborplots}, we compare the optimal control strategies $(\pi^*_t)_t$ and $(c^*_t)_t$ for different values of labor income $\ell\in\lbrace0.2,0.6,1\rbrace$, while we fix the other parameters as: $
r=0.03,\;  \delta=0.04, \; \gamma=1.5,  \;\kappa=0.25, \; \bar{\beta}=0.05, \; \sigma_\beta=0.03$, and $\sigma=0.18$.\\
\indent We observe that both consumption and risky investment are increasing in wealth $X_t$ and labor income $\ell$. This is consistent with standard economic intuition: higher levels of wealth and labor income increase the agent's total effective wealth, thereby relaxing the budget constraint. Consequently, the agent increases both consumption and their allocation to the risky asset.
 \begin{figure}[h]
    \centering
    \begin{subfigure}[h]{0.45\textwidth}
        \centering
        \includegraphics[width=\textwidth]{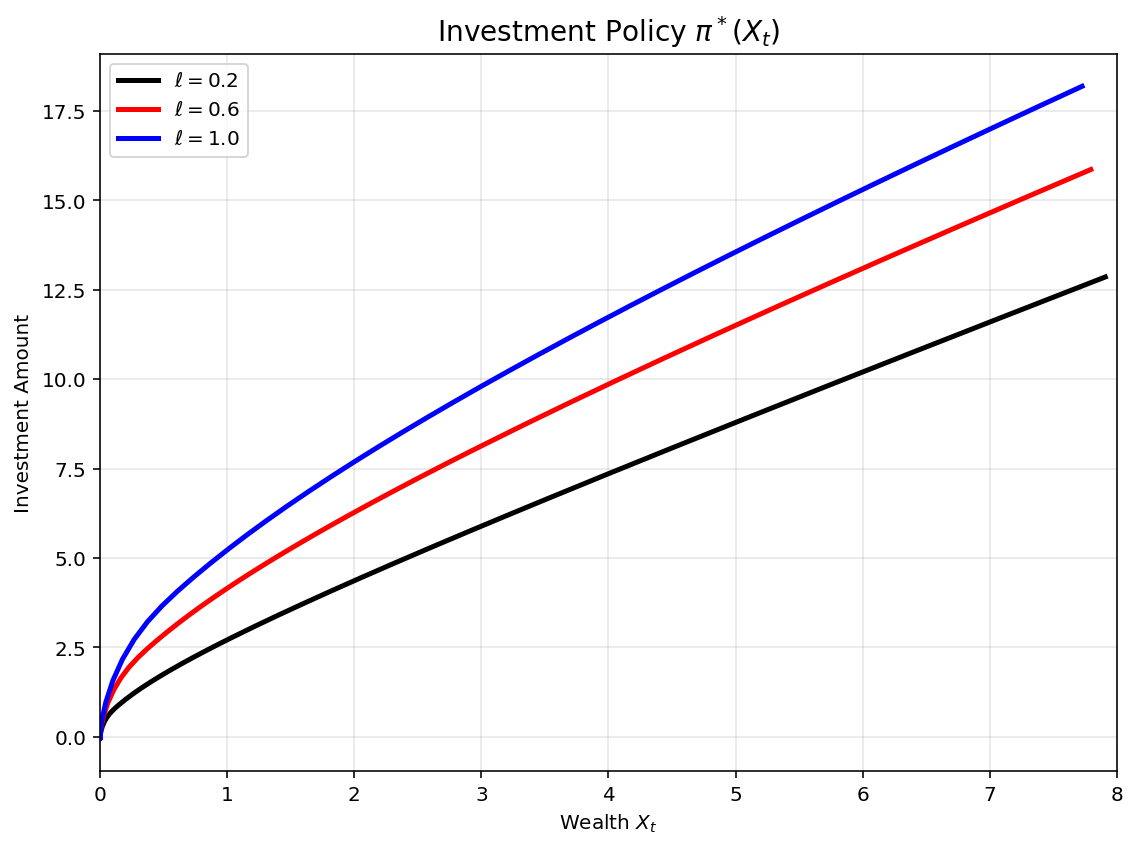}
        \caption{Optimal investment policy $\pi^*_t$ for different values of labor income $\ell$.}
    \end{subfigure}
    \hfill 
    \begin{subfigure}[h]{0.45\textwidth}
        \centering
        \includegraphics[width=\textwidth]{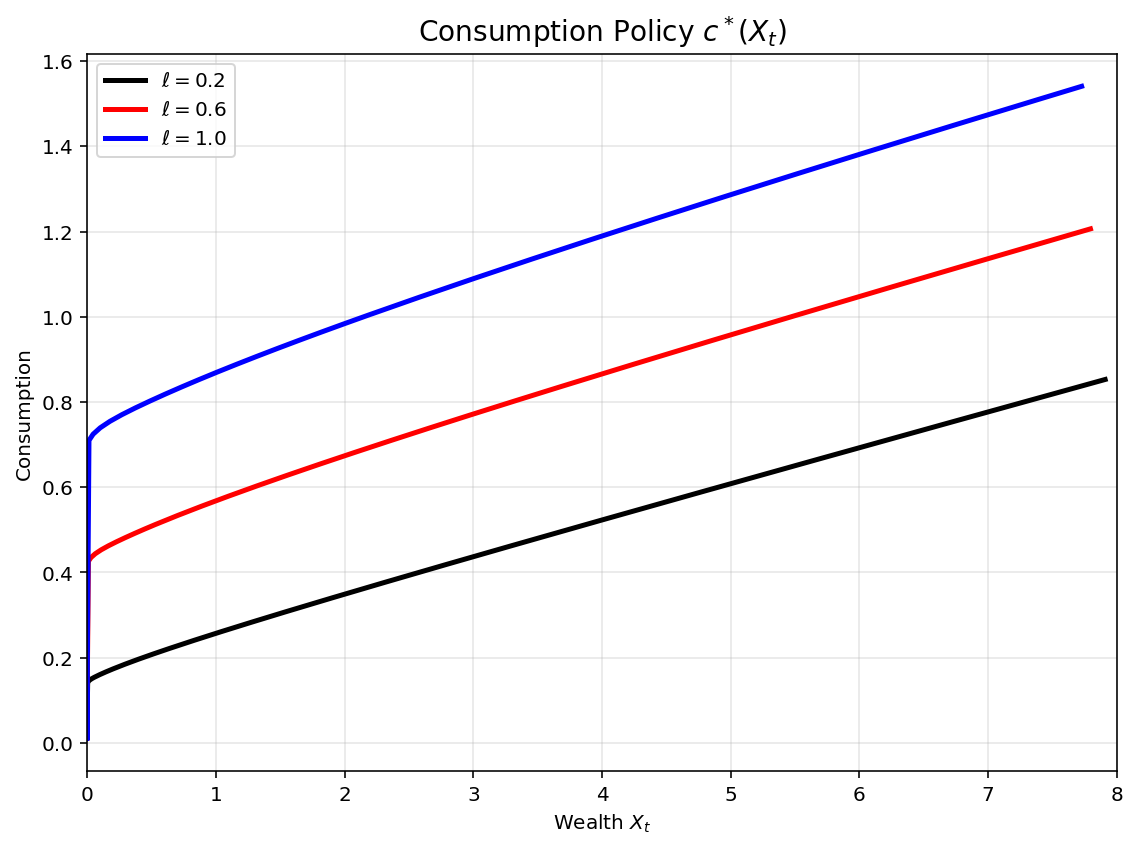}
        \caption{Optimal consumption policy $c^*_t$ for different values of labor income $\ell$.}
    \end{subfigure}
    
    \caption{Optimal policies for different values of labor income $\ell$.}
    \label{laborplots}
\end{figure}\\
\\
In Figure \ref{riskaversionplots}, we compare the optimal control strategies $(\pi^*_t)_t$ and $(c^*_t)_t$ for different values of risk aversion $\gamma\in\lbrace1.2,1.5,2\rbrace$, while we fix the other parameters as: $
r=0.03,\;  \delta=0.04,\;  \ell=0.6,\;  \kappa=0.25,\;  \bar{\beta}=0.05,  \;\sigma_\beta=0.03$, and $\sigma=0.18$.\\
\indent As before, we observe that both optimal consumption $(c^*_t)_t$ and risky investment $(\pi^*_t)_t$ are strictly increasing in wealth $X_t$ since the agent has more money to allocate. Regarding the effect of risk preferences, the risky investment $(\pi^*_t)_t$ is decreasing in the risk aversion parameter $\gamma$. This inverse relationship is consistent with Merton's classic results, reflecting that more risk-averse agents fear market volatility more and consequently reduce their exposure to the risky asset.\\
\indent Conversely, the consumption policy $(c^*_t)_t$ shifts upward as $\gamma$ increases. This behavior is driven by the Intertemporal Elasticity of Substitution (IES). Agents with lower risk aversion (higher IES) are more willing to postpone current consumption to capitalize on investment opportunities, whereas highly risk-averse agents (lower IES) invest less and choose to consume larger amounts of their current wealth.\\
\begin{figure}[h]
    \centering
    \begin{subfigure}[h]{0.45\textwidth}
        \centering
        \includegraphics[width=\textwidth]{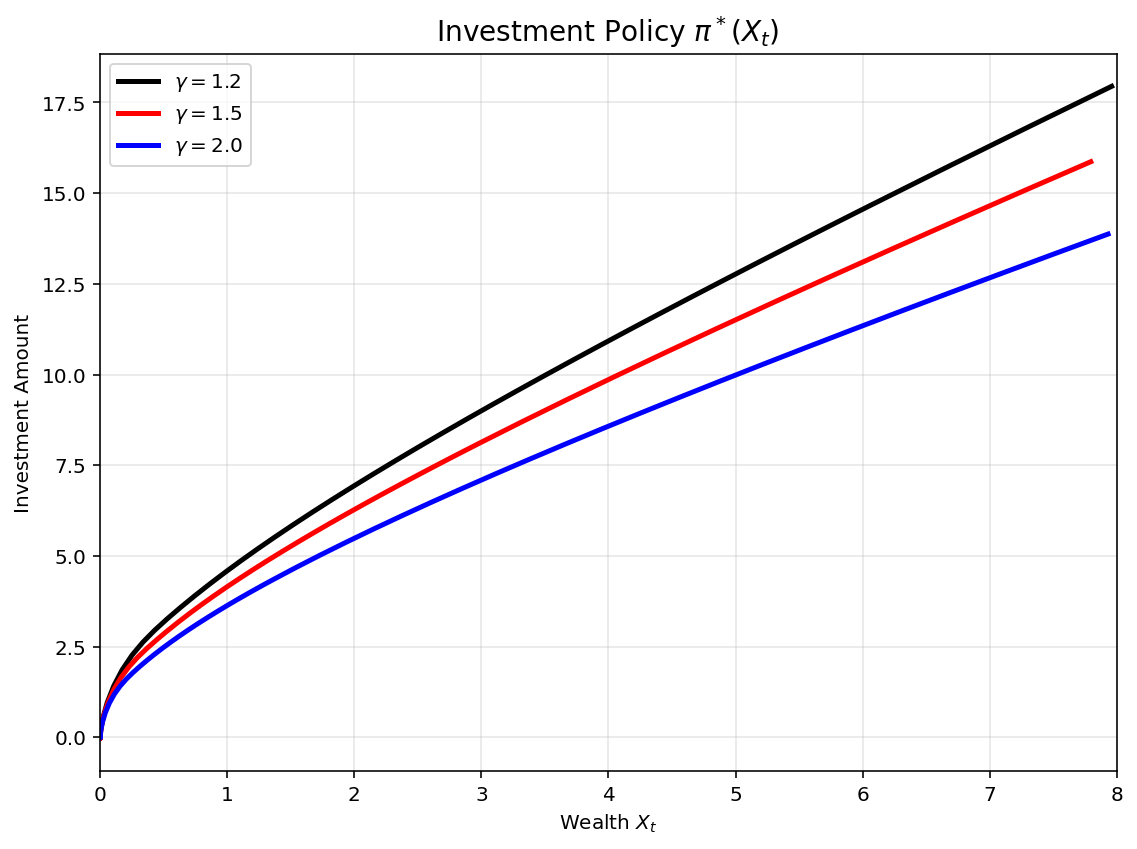}
        \caption{Optimal investment policy $\pi^*_t$ for different values of risk aversion $\gamma$.}
    \end{subfigure}
    \hfill 
    \begin{subfigure}[h]{0.45\textwidth}
        \centering
        \includegraphics[width=\textwidth]{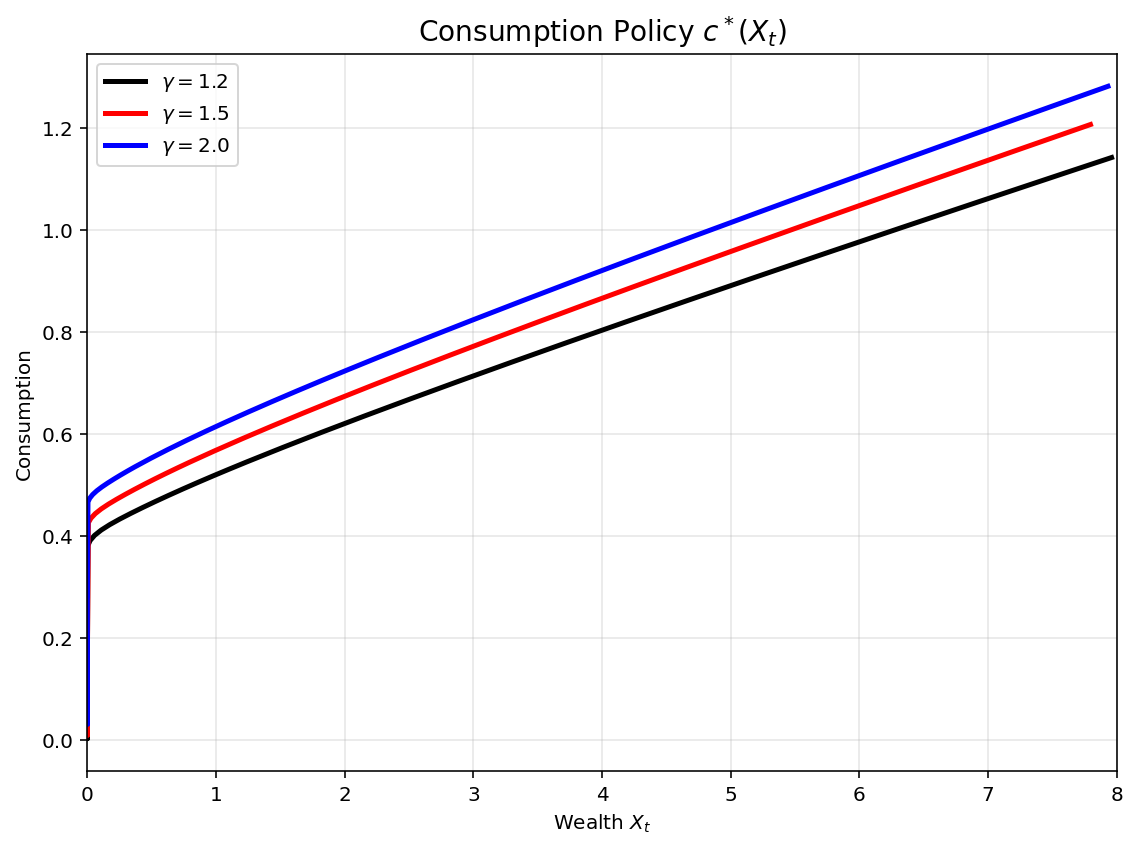}
        \caption{Optimal consumption policy $c^*_t$ for different values of risk aversion $\gamma$.}
    \end{subfigure}
    
    \caption{Optimal policies for different values of risk aversion $\gamma$.}
    \label{riskaversionplots}
\end{figure}\\
Continuing with our cross-sectional approach, Figure \ref{betaplots} compares the optimal control policies $(\pi^*_t)_t$ and $(c^*_t)_t$ for different values of expected excess returns $\beta_t\in\lbrace0.02,0.05,0.12\rbrace$, while we fix the other parameters as: $
r=0.03,\;  \delta=0.04,\;  \ell=0.6,\; \gamma=1.5,\;  \kappa=0.25,\;  \bar{\beta}=0.05, \;\sigma_\beta=0.03$, and $\sigma=0.18$.\\
\indent We observe that the investment policy $(\pi_t^*)_t$ is strictly increasing in the expected excess return $\beta_t$. When the risk premium is high, the agent aggressively leverages the portfolio to capitalize on favorable investment opportunities. Conversely, when the premium is low, the agent substantially reduces exposure to the risky asset.\\
\indent The consumption policy $(c_t^*)_t$ displays a non-monotonic relationship with $\beta_t$. At low wealth levels ($X_t < 3.5$), consumption decreases as $\beta_t$ increases. In this region, the agent seeks to accumulate wealth more rapidly. Since a higher $\beta_t$ implies more profitable investment opportunities, the agent optimally cuts current consumption in order to finance larger risky positions. At higher wealth levels, however, the agent has sufficient capital to fully exploit the high expected returns. This relaxes the need for aggressive saving and allows consumption to rise, eventually exceeding the level observed in the low-$\beta_t$ state.

\begin{figure}[h]
    \centering
    \begin{subfigure}[h]{0.45\textwidth}
        \centering
        \includegraphics[width=\textwidth]{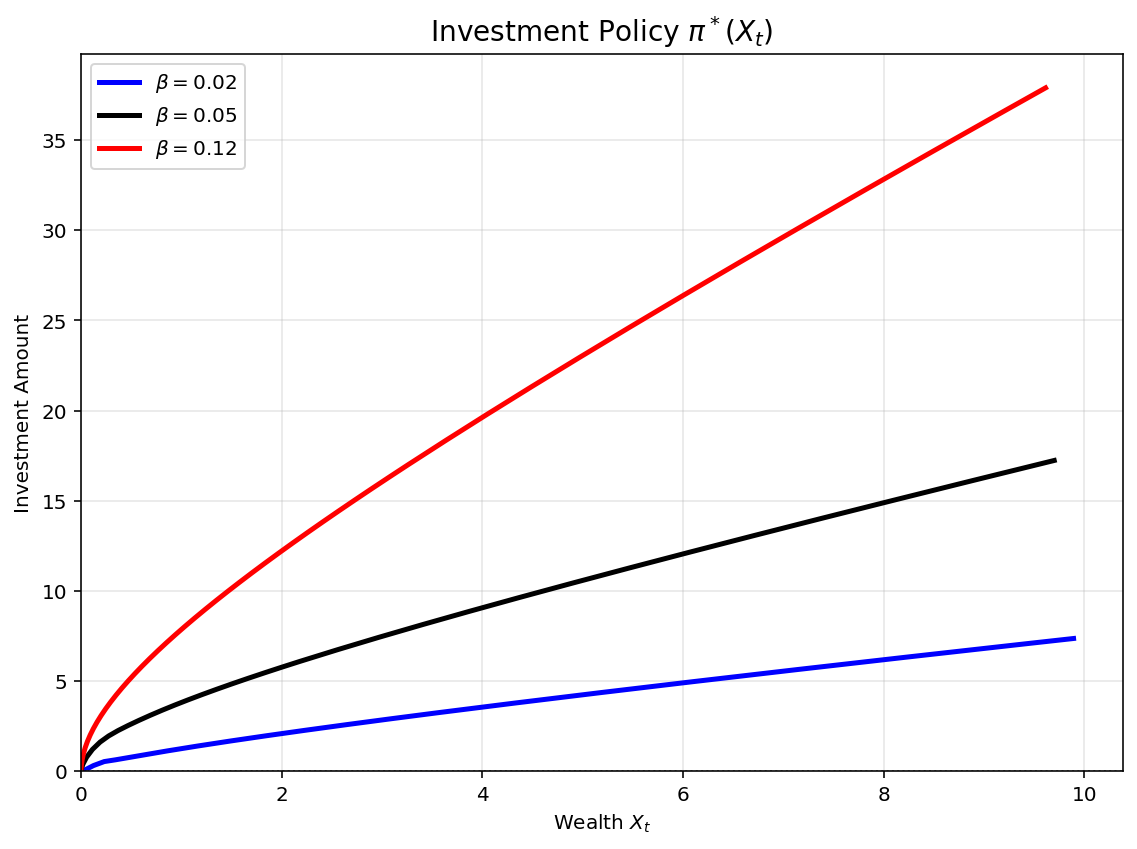}
        \caption{Cross-sections of the optimal investment policy $\pi^*_t$ for different fixed realizations of the expected excess return $\beta_t$.}
    \end{subfigure}
    \hfill 
    \begin{subfigure}[h]{0.45\textwidth}
        \centering
        \includegraphics[width=\textwidth]{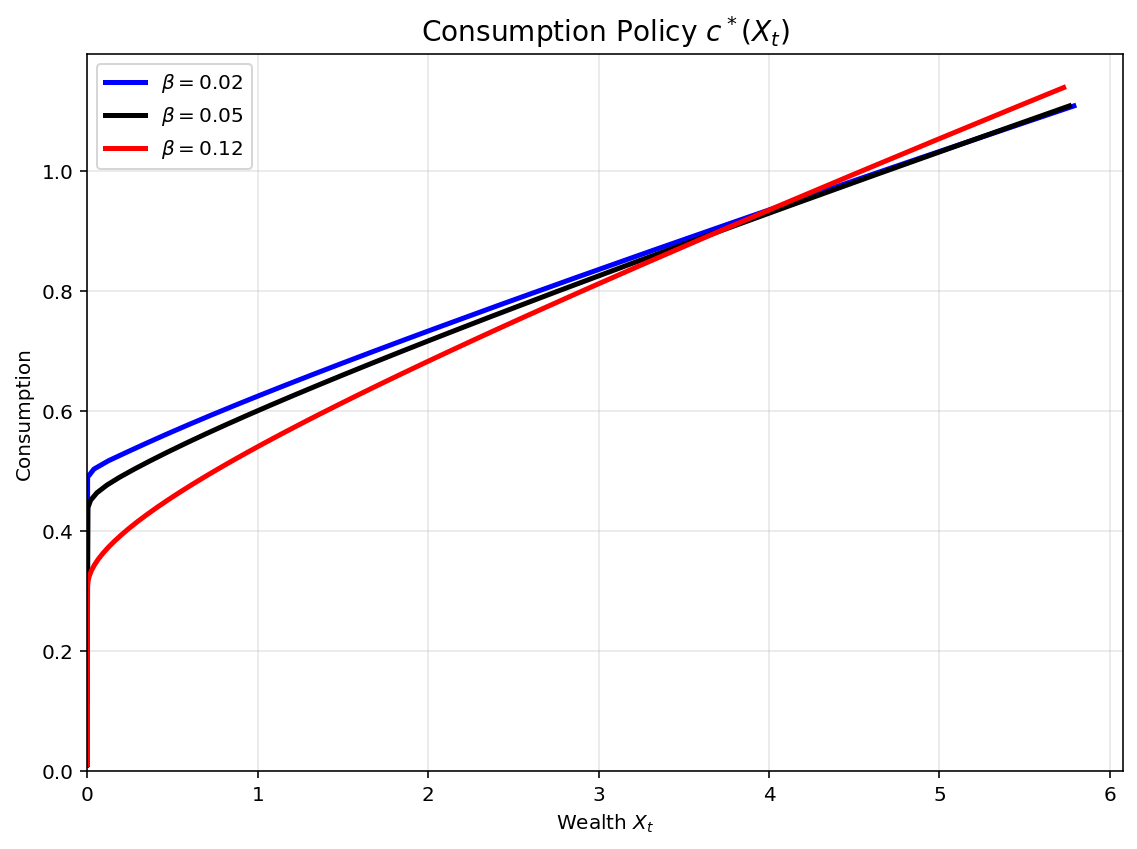}
        \caption{Cross-sections of the optimal consumption policy $c^*_t$ for different fixed realizations of the expected excess return $\beta_t$.}
    \end{subfigure}
    
    \caption{Cross-sections of the optimal policies for different fixed realizations of the expected excess return $\beta_t$.}
    \label{betaplots}
\end{figure}
\section{Conclusions}
\noindent
In this paper, we have studied the infinite-horizon consumption and portfolio problem of an investor subject to a no-borrowing constraint and labor income in the complete-market case of the Kim--Omberg model, where expected excess returns follow a mean-reverting Ornstein--Uhlenbeck process and are perfectly negatively correlated with the risky asset return shocks. Using a Lagrange duality approach, we have formulated the dual problem as a two-dimensional singular control problem involving the marginal value of wealth and the stochastic factor. The solution is governed by an auxiliary optimal stopping problem, which features a free boundary separating the continuation and stopping regions. We have provided a detailed probabilistic analysis for this optimal stopping problem and established properties of the free boundary and of the value function. Finally, we have retrieved the solutions to the primal optimization problem via duality and provided numerical illustrations.\\
\indent An interesting direction is the extension to the incomplete-market case. In that setting, the duality approach developed here is no longer directly applicable, as our proof of strong duality would break down, and a different mathematical approach, likely involving the associated HJB equation and verification arguments, would be needed. Such a challenging extension is left for future research.\\
\\ 
\noindent
\textbf{Acknowledgements.} Financial support by the German Research Foundation (DFG) [RTG\\
2865/1 - 492988838] is gratefully acknowledged.

{\normalsize
\appendix
\section{Technical Proofs}
\subsection{Proof of Proposition \ref{measurechange} }\label{proofmeasurechange}
\begin{proof}
Since $(M_t)_t$ is a martingale, we can fix a finite time horizon $T\geq 0$ and define the new measure \(\mathbb{\bar{P}}\) on $(\Omega,\mathcal{F})$ via
\[
\frac{d\mathbb{\bar{P}}}{d\mathbb{P}}\bigg|_{\mathcal{F}^W_t} = M_t,\quad t\in[0,T].
\] Using Girsanov's Theorem, we introduce a new standard Brownian motion \((W^\mathbb{\bar{P}}_t)_t\) under \(\mathbb{\bar{P}}\) such that \(dW_t^\mathbb{\bar{P}} = dW_t + \frac{\beta_t}{\sigma}\,dt\). Then, the dynamics of the state processes \((Z^{1}_t)_t\) and \((\beta_t)_t\) are given under $\bar{\mathbb{P}}$ by
\begin{align*}
dZ^1_t &= -\frac{\beta_t}{\sigma}Z^1_t\,dW^\mathbb{\bar{P}}_t + Z^1_t\bigg(\delta - r + \frac{\beta_t^2}{\sigma^2}\bigg)\,dt, \quad t>0,\quad Z^1_0=z>0, \\
d\beta_t &= -\sigma_\beta\,dW_t^{\mathbb{\bar{P}}} + \bigg(\kappa(\overline{\beta} - \beta_t) +\,\frac{\beta_t}{\sigma}\sigma_\beta\bigg)\,dt,\quad t>0, \quad\beta_0=\beta\in\mathbb{R}.
\end{align*} Under $\mathbb{\bar{P}}$, we then have for any $\mathbb{F}^W$-stopping time $\tau$: \begin{equation}\label{bar(P)eq}\mathbb{E}_{z,\beta}\bigg[\int_0^{\tau\wedge T} e^{-rt} M_t\,(\tilde{u}'(Z_t^{1}) + \ell)\,dt\bigg]= \mathbb{E}_{z,\beta}^{\mathbb{\bar{P}}}\bigg[\int_0^{\tau\wedge T} e^{-rt}\,(\tilde{u}'(Z^1_t) + \ell)\,dt\bigg],
\end{equation} where $\mathbb{E}^\mathbb{\bar{P}}_{z,\beta}[\cdot]$ denotes the expectation under $\mathbb{\bar{P}}_{z,\beta}(\cdot)=\mathbb{\bar{P}}(\;\cdot\;|\;Z^1_0=z,\;\beta_0=\beta)$. As the measure $\mathbb{\bar{P}}$ depends on $T$, we cannot directly let $T\to\infty$. To bypass this problem, we proceed as in \cite{callegaro2020optimal} (see also \cite{de2020optimal}). We observe that the coefficients of the SDEs for $Z^1$ and $\beta$ do not depend on the horizon $T$. Therefore, the law of the process $(Z^1, \beta)$ on $[0, T]$ under $\mathbb{\bar{P}}$ is consistent for any $T$. This allows us to introduce a new auxiliary probability space $(\hat{\Omega}, \hat{\mathcal{F}}, \mathbb{Q})$ equipped with a standard Brownian motion $(W^\mathbb{Q}_t)_t$, generating the filtration (completed by $\mathbb{Q}\text{-null sets of}\; \hat{\mathcal{F}}$) $\mathbb{F}^{W,\mathbb{Q}}:=(\mathcal{F}_t^{W,\mathbb{Q}})_t$. On this space, we let $(\hat{Z}_t)_t$ and $(\hat{\beta}_t)_t$ be the unique strong solutions to
\begin{align*}
d\hat{Z}_t &= -\frac{\hat{\beta}_t}{\sigma}\hat{Z}_t\,dW^\mathbb{Q}_t + \hat{Z}_t\bigg(\delta - r + \frac{\hat{\beta}_t^2}{\sigma^2}\bigg)\,dt, \quad t>0,\quad\hat{Z}_0=z>0, \\
d\hat{\beta}_t &= -\sigma_\beta\,dW^{\mathbb{Q}}_t + \bigg(\kappa(\overline{\beta} - \hat{\beta}_t) +\,\frac{\hat{\beta}_t}{\sigma}\sigma_\beta\bigg)\,dt,\quad t>0,\quad\hat{\beta}_0=\beta\in\mathbb{R}.
\end{align*} We then define the following value functions (where the infimum is taken over all $\mathbb{F}^{W,\mathbb{Q}}$-stopping times)
\[
v^\mathbb{Q}(z, \beta; T) := \inf_{\tau} \mathbb{E}^\mathbb{Q}_{z,\beta}\left[\int_0^{\tau \wedge T} e^{-rt} (\tilde{u}'(\hat{Z}_t) + \ell) \, dt \right],
\] and \[
v^\mathbb{Q}(z, \beta) := \inf_{\tau} \mathbb{E}^\mathbb{Q}_{z,\beta}\left[\int_0^{\tau} e^{-rt} (\tilde{u}'(\hat{Z}_t) + \ell) \, dt \right].
\]
By using the Dominated Convergence Theorem (whose application is justified by using arguments as those in the proof of the upcoming Proposition \ref{wellposedness}), we obtain
\begin{equation}
\label{eq:aux_convergence}
\lim_{T \to \infty} v^\mathbb{Q}(z, \beta; T) = v^\mathbb{Q}(z, \beta), \quad \lim_{T \to \infty} v(z,\beta;T) = v(z,\beta),
\end{equation} where we have set \[v(z,\beta;T):=\inf_{\tau} \mathbb{E}_{z,\beta}\bigg[\int_0^{\tau\wedge T} e^{-rt} M_t\,(\tilde{u}'(Z_t^{z,1}) + \ell)\,dt\bigg].\] Since now
\[
\inf_\tau \mathbb{E}_{z,\beta}^{\bar{\mathbb{P}}} \bigg[\int_0^{\tau\wedge T} e^{-rt}\,(\tilde{u}'(Z^1_t) + \ell)\,dt\bigg]
= \inf_\tau \mathbb{E}^{\mathbb{Q}}_{z,\beta} \bigg[\int_0^{\tau\wedge T} e^{-rt}\,(\tilde{u}'(\hat{Z}_t) + \ell)\,dt\bigg],
\]
the equivalence in law of the processes $(Z^1_t, \beta_t, W_t^{\bar{\mathbb{P}}}, W_t^{\beta,\bar{\mathbb{P}}})_t$ under $\bar{\mathbb{P}}$ and $(\hat{Z}_t, \hat{\beta}_t, W_t^\mathbb{Q}, W_t^{\beta,\mathbb{Q}})_t$ under $\mathbb{Q}$ on $[0,T]$, and (\ref{bar(P)eq}) allow us to write 
\[
v^\mathbb{Q}(z, \beta) = \lim_{T\to\infty} v^\mathbb{Q}(z, \beta;T) = \lim_{T\to\infty} v(z,\beta;T) = v(z,\beta),
\] which implies
\begin{equation*}v(z,\beta)=\inf_{\tau} \mathbb{E}_{z,\beta}^{\mathbb{Q}}\bigg[\int_0^\tau e^{-rt}\,(\tilde{u}'(\hat{Z}_t) + \ell)\,dt\bigg].\end{equation*}
\end{proof}
\subsection{Proof of Proposition \ref{wellposedness}}
\begin{proof}\label{proofwellposedness}
It follows from (\ref{legendre}) that $\tilde{u}'(\hat{Z}_t)=-\hat{Z}^{-\frac{1}{\gamma}}_t$. Therefore,
\[
\mathbb{E}^\mathbb{Q}_{z,\beta}\bigg[\int_0^\infty e^{-rt}|-\hat{Z}^{-\frac{1}{\gamma}}_t+\ell|dt\bigg]
\leq
\mathbb{E}^\mathbb{Q}_{z,\beta}\bigg[\int_0^\infty e^{-rt}\hat{Z}^{-\frac{1}{\gamma}}_tdt\bigg]
+\ell\int_0^\infty e^{-rt}dt.
\]
Hence, in order to prove that $v(z,\beta)\in\mathbb{R}$ it suffices to show that
\[
\mathbb{E}^\mathbb{Q}_{z,\beta}\bigg[\int_0^\infty e^{-rt}\hat{Z}^{-\frac{1}{\gamma}}_tdt\bigg]<\infty.
\]
By Fubini-Tonelli Theorem, we may write
\[
\mathbb{E}^\mathbb{Q}_{z,\beta}\bigg[\int_0^\infty e^{-rt}\hat{Z}^{-\frac{1}{\gamma}}_tdt\bigg]
=\int_0^\infty e^{-rt}\mathbb{E}^\mathbb{Q}_{z,\beta}[\hat{Z}_t^{-\frac{1}{\gamma}}]\,dt.
\]
By using the explicit representation of the strong solution to (\ref{hatZ}), we find
\begin{align*}
&\mathbb{E}^\mathbb{Q}_{z,\beta}[\hat{Z}_t^{-\frac{1}{\gamma}}]
=z^{-\frac{1}{\gamma}}\mathbb{E}^\mathbb{Q}_{1,\beta}\bigg[\exp\bigg(-\frac{1}{\gamma}\int_0^t (\delta-r+\tfrac{1}{2}\tfrac{\hat{\beta}_s^2}{\sigma^2})ds
+\frac{1}{\gamma}\int_0^t \tfrac{\hat{\beta}_s}{\sigma}dW_s^{\mathbb{Q}}\bigg)\bigg]\\
&=z^{-\frac{1}{\gamma}}\mathbb{E}^\mathbb{Q}_{1,\beta}\bigg[\exp\bigg(-\frac{1}{\gamma}(\delta-r)t-\tfrac{1}{2\sigma^2\gamma}\int_0^t\hat{\beta}_s^2ds+\frac{1}{\gamma}\int_0^t \tfrac{\hat{\beta}_s}{\sigma}dW_s^{\mathbb{Q}}\bigg)\bigg]\\
&=z^{-\frac{1}{\gamma}}\exp\big(-\frac{1}{\gamma}(\delta-r)t\big)\cdot\\
&\quad \cdot\mathbb{E}^\mathbb{Q}_{1,\beta}\bigg[\exp\bigg(-\frac{1}{2\sigma^2\gamma}\int_0^t\hat{\beta}_s^2ds
+\tfrac{1}{\sigma\gamma}\int_0^t \hat{\beta}_sdW^{\mathbb{Q}}_s
-\tfrac{1}{2}\tfrac{ 1}{\sigma^2\gamma^2}\int_0^t\hat{\beta}_s^2ds
+\tfrac{1}{2}\tfrac{ 1}{\sigma^2\gamma^2}\int_0^t\hat{\beta}_s^2ds\bigg)\bigg]\\
&=z^{-\frac{1}{\gamma}}\exp\big(-\frac{1}{\gamma}(\delta-r)t\big)\\
&\quad\cdot\mathbb{E}^\mathbb{Q}_{1,\beta}\bigg[\exp\bigg(\tfrac{1-\gamma}{2\sigma^2\gamma^2}\int_0^t\hat{\beta}_s^2ds\bigg)
\exp\bigg(\tfrac{1}{\sigma\gamma}\int_0^t \hat{\beta}_sdW^{\mathbb{Q}}_s
-\tfrac{1}{2}\tfrac{ 1}{\sigma^2\gamma^2}\int_0^t\hat{\beta}_s^2ds\bigg)\bigg].
\end{align*}
The process
\[
N_t:=\exp\bigg(\tfrac{1}{\sigma\gamma}\int_0^t \hat{\beta}_sdW^{\mathbb{Q}}_s
-\tfrac{1}{2}\tfrac{ 1}{\sigma^2\gamma^2}\int_0^t\hat{\beta}_s^2ds\bigg)
\]
defines a martingale under $\mathbb{Q}$ by Assumption \ref{Novikov_OST}. Hence, Girsanov’s Theorem allows us to define a new probability measure by
\[
\frac{d\mathbb{Q'}}{d\mathbb{Q}}\bigg|_{\mathcal{F}^{W,\mathbb{Q}}_t}=N_t,
\] and to obtain, as $\gamma>1$,
\begin{align*}\label{boundforZ}
\mathbb{E}^\mathbb{Q}_{z,\beta}[\hat{Z}_t^{-\frac{1}{\gamma}}]
&=z^{-\frac{1}{\gamma}}\exp\big({-\frac{1}{\gamma}}(\delta-r)t\big)
\mathbb{E}^\mathbb{Q'}_{1,\beta}\bigg[\exp\bigg(\bigg(\tfrac{1-\gamma}{2\sigma^2 \gamma^2}\bigg)\int_0^t\hat{\beta}_s^2ds\bigg)\bigg]\notag\\
&\leq z^{-\frac{1}{\gamma}}\exp\bigg(-\tfrac{1}{\gamma}(\delta-r)t\bigg),
\end{align*}
which proves (\ref{E(Z)ineq}). Overall,
\begin{equation}\label{finalexpression}
\mathbb{E}^\mathbb{Q}_{z,\beta}\bigg[\int_0^\infty e^{-rt}\hat{Z}^{{-\frac{1}{\gamma}}}_tdt\bigg]
=\int_0^\infty e^{-rt}\mathbb{E}^\mathbb{Q}_{z,\beta}[\hat{Z}_t^{-\frac{1}{\gamma}}]dt
\leq z^{-\frac{1}{\gamma}}\int_0^\infty e^{-(r+\frac{1}{\gamma}(\delta-r))t}\, dt<\infty,
\end{equation}
since $r+\frac{1}{\gamma}(\delta-r)>0$ as $\gamma>1$.
\end{proof}
\subsection{Proof of Proposition \ref{bounds}}\label{proofbounds}
\begin{proof}
Taking the suboptimal stopping time $\tau=0$ in (\ref{OST_new}) clearly yields $v(z,\beta)\leq 0$ for all $(z,\beta)$ on $\mathcal{O}$. As for the lower bound, using that $\tilde{u}'(\hat{Z}_t)=-\hat{Z}_t^{-\frac{1}{\gamma}}$ by (\ref{legendre}), we have for any $(z,\beta)\in\mathcal{O}$ and any stopping time $\tau$
\begin{align*}
\mathbb{E}^\mathbb{Q}_{z,\beta}\bigg[\int_0^\tau e^{-rt} (\tilde{u}'(\hat{Z}_t) + \ell) dt \bigg]
&\ge -\mathbb{E}^\mathbb{Q}_{z,\beta}\bigg[\int_0^\tau e^{-rt} \hat{Z}_t^{-\frac{1}{\gamma}} dt \bigg] \\
&\ge -\mathbb{E}^\mathbb{Q}_{z,\beta}\bigg[\int_0^\infty e^{-rt} \hat{Z}_t^{-\frac{1}{\gamma}} dt \bigg]>-\infty,
\end{align*}
where the last inequality is due (\ref{E(Z)ineq}) and the fact that $r+\frac{1}{\gamma}(\delta-r)>0$ as $\gamma>1$. Arbitrariness of $\tau\geq 0$ then implies the result.
\end{proof}
\subsection{Proof of Proposition \ref{limits}}\label{prooflimits}
\begin{proof}
We have
\begin{equation}\label{limitineq}
0 \geq \lim_{z\to\infty}v(z,\beta) \geq \lim_{z\to \infty}\left( - z^{-\frac{1}{\gamma}} \mathbb{E}^\mathbb{Q}_{1,\beta}\bigg[\int_0^\infty e^{-rt} \hat{Z}_t^{-\frac{1}{\gamma}} dt\bigg]\right)=0,
\end{equation}
where we have used that the expectation in the right-hand side of (\ref{limitineq}) does not depend on $z$. Similarly, one also obtains
\begin{align*}
\lim_{z\to 0}v(z,\beta) &\leq \lim_{z\to 0}\mathbb{E}^\mathbb{Q}_{z,\beta}\bigg[\int_0^\infty e^{-rt}(-\hat{Z}_t^{-\frac{1}{\gamma}} + \ell) dt \bigg]\\
&= \lim_{z\to 0}\left(- z^{-\frac{1}{\gamma}} \mathbb{E}^\mathbb{Q}_{1,\beta}\bigg[\int_0^\infty e^{-rt} \hat{Z}_t^{-\frac{1}{\gamma}} dt \bigg] + \frac{\ell}{r}\right)= -\infty.
\end{align*}
\end{proof}
\subsection{Proof of Proposition \ref{ProbRepr}}\label{proofprobrepr}
\begin{proof}
The proof is divided into two steps and borrows arguments from \cite{MR3904414}.\\
\\
\indent\textbf{Step 1.} We first show that the value function $v$ is (locally) Lipschitz continuous in the $z$-variable and derive the probabilistic representation for the weak derivative $v_z$. To that end,  we fix $(z,\beta)\in\mathcal{O}$ and $\varepsilon>0$, and let $\tau^*$ be the optimal stopping time for the problem with initial data $(z,\beta)$ (independent of $\varepsilon$).\\
\indent For the purpose of showing the Lipschitz property, we may restrict to $\varepsilon \le \varepsilon_0$ with $\varepsilon_0 \in (0,1)$ such that $z-\varepsilon>0$. Using $\tilde{u}'(\hat{Z}_t)=-\hat{Z}_t^{-\frac{1}{\gamma}}$ and the Mean-Value Theorem, we have
\begin{align}\label{zlipbound}
|v(z+\varepsilon,\beta) - v(z,\beta)| &\le \mathbb{E}^\mathbb{Q}\Big[\int_0^\infty e^{-rt} (\hat{Z}^{1,\beta}_t)^{-\frac{1}{\gamma}} | z^{-\frac{1}{\gamma}} - (z+\varepsilon)^{-\frac{1}{\gamma}} | dt \Big] \notag\\
&= | z^{-\frac{1}{\gamma}} - (z+\varepsilon)^{-\frac{1}{\gamma}} | \, \mathbb{E}^\mathbb{Q}\Big[\int_0^\infty e^{-rt} (\hat{Z}^{1,\beta}_t)^{-\frac{1}{\gamma}} dt \Big] \\
&= \frac{1}{\gamma} \zeta^{-\frac{1}{\gamma}-1} \, \varepsilon \, \mathbb{E}^\mathbb{Q}\Big[\int_0^\infty e^{-rt} (\hat{Z}^{1,\beta}_t)^{-\frac{1}{\gamma}} dt \Big],\notag
\end{align}
for some $\zeta \in [z, z+\varepsilon]$. Upon setting
\[
c(z,\beta) := \frac{1}{\gamma} z^{-\frac{1}{\gamma}-1} \, \mathbb{E}^\mathbb{Q}\Big[\int_0^\infty e^{-rt} (\hat{Z}^{1,\beta}_t)^{-\frac{1}{\gamma}} dt \Big],
\]
we thus obtain from (\ref{zlipbound})
\[
|v(z+\varepsilon,\beta) - v(z,\beta)| \le c(z,\beta) \, \varepsilon,
\]
and by symmetry,
\[
|v(z,\beta) - v(z-\varepsilon,\beta)| \le c(z,\beta) \, \varepsilon.
\]
Since $c(z,\beta) > 0$ is finite by Proposition \ref{wellposedness} and can be taken uniformly on compact sets, we conclude that $v(\cdot,\beta)$ is (locally) Lipschitz continuous in $z$.\\
\indent In order to derive the probabilistic representation of the weak derivative $v_z$, we note that $\tau^*$ is suboptimal for $v(z+\varepsilon,\beta)$, and thus we have
\begin{align*}
v(z+\varepsilon,\beta) - v(z,\beta) &\le \mathbb{E}^\mathbb{Q}\Big[\int_0^{\tau^*} e^{-rt} (\hat{Z}^{1,\beta}_t)^{-\frac{1}{\gamma}} \big( z^{-\frac{1}{\gamma}} - (z+\varepsilon)^{-\frac{1}{\gamma}} \big) dt \Big] \\
&= \big( z^{-\frac{1}{\gamma}} - (z+\varepsilon)^{-\frac{1}{\gamma}} \big) \mathbb{E}^\mathbb{Q}\Big[\int_0^{\tau^*} e^{-rt} (\hat{Z}^{1,\beta}_t)^{-\frac{1}{\gamma}} dt \Big].
\end{align*}
Dividing by $\varepsilon$ and then letting $\varepsilon \to 0$ yields
\begin{equation}\label{limsup_z}
\limsup_{\varepsilon \to 0} \frac{v(z+\varepsilon,\beta) - v(z,\beta)}{\varepsilon} \le \mathbb{E}^\mathbb{Q}\Big[\int_0^{\tau^*} e^{-rt} \frac{1}{\gamma} z^{-1} (\hat{Z}^{z,\beta}_t)^{-\frac{1}{\gamma}} dt \Big].
\end{equation}
A symmetric argument applied to $v(z,\beta) - v(z-\varepsilon,\beta)$ gives the reverse inequality
\begin{equation}\label{liminf_z}
\liminf_{\varepsilon \to 0} \frac{v(z,\beta) - v(z-\varepsilon,\beta)}{\varepsilon} \ge \mathbb{E}^\mathbb{Q}\Big[\int_0^{\tau^*} e^{-rt} \frac{1}{\gamma} z^{-1} (\hat{Z}^{z,\beta}_t)^{-\frac{1}{\gamma}} dt \Big].
\end{equation}
Hence, since $v$ is (locally) Lipschitz continuous in $z$, it follows from (\ref{limsup_z}) and (\ref{liminf_z}) that for any point $z$ of differentiability (belonging to a set of full measure) the weak derivative $v_z$ is given by
\[
v_z(z,\beta) = \mathbb{E}^\mathbb{Q}\Big[\int_0^{\tau^*} e^{-rt} \frac{1}{\gamma} z^{-1} (\hat{Z}^{z,\beta}_t)^{-\frac{1}{\gamma}} dt \Big].
\]
\\  
\textbf{Step 2.} In this step, we show that $v$ is (locally) Lipschitz continuous in the $\beta$-variable and derive the probabilistic representation for the weak derivative $v_\beta$. Again, let $(z,\beta)\in\mathcal{O}$, and to simplify notation, we set 
\[
a := \kappa - \frac{\sigma_\beta}{\sigma} > 0 \quad \text{(by Assumption \ref{assumption_novikov})}, \quad b := \frac{\kappa \overline{\beta}}{a},
\] so that the unique strong solution to (\ref{hatbeta}) can be written as
\begin{equation}\label{expressionbeta}
\hat{\beta}^\beta_t = \beta e^{-at} + b(1 - e^{-at}) - \sigma_\beta \int_0^t e^{-a(t-s)} dW_s^{\mathbb{Q}},
\end{equation}
and it readily follows that $\frac{\partial \hat{\beta}^\beta_t}{\partial \beta} = e^{-at}$. Since
\[
\hat{Z}^{z,\beta}_t = z \exp\Bigg(\int_0^t \big(\delta - r + \frac{1}{2} \frac{(\hat{\beta}^\beta_s)^2}{\sigma^2}\big) ds - \int_0^t \frac{\hat{\beta}^\beta_s}{\sigma} dW_s^{\mathbb{Q}}\Bigg),
\] 
we have by Theorem $V.7.39$ in \cite{protter2012stochastic}
\begin{equation}\label{protter}
\frac{\partial \hat{Z}_t}{\partial \beta} = \hat{Z}_t \Bigg(\int_0^t e^{-as} \frac{\hat{\beta}_s}{\sigma^2} ds - \frac{1}{\sigma} \int_0^t e^{-as} dW_s^{\mathbb{Q}} \Bigg),
\end{equation}
upon using that, due to (\ref{expressionbeta}), one has almost surely
\[
\partial_\beta \Big(\int_0^t \frac{\hat{\beta}^\beta_s}{\sigma} dW_s^{\mathbb{Q}}\Big) 
:= \lim_{\varepsilon \to 0} \frac{\int_0^t \frac{\hat{\beta}_s^{\beta+\varepsilon}}{\sigma} dW_s^{\mathbb{Q}} - \int_0^t \frac{\hat{\beta}_s^\beta}{\sigma} dW_s^{\mathbb{Q}}}{\varepsilon} 
= \int_0^t \frac{e^{-as}}{\sigma} dW_s^{\mathbb{Q}}.
\]
\indent As in Step 1, we may restrict to $\varepsilon \le \varepsilon_0$ for some $\varepsilon_0\in(0,1)$, and by the Mean-Value Theorem and (\ref{protter}), we have
\begin{align*}
\big|v(z,\beta+\varepsilon) - v(z,\beta)\big| 
&\le \int_0^\infty e^{-rt} \mathbb{E}^\mathbb{Q}\Big[ \big| (\hat{Z}_t^{z,\beta})^{-\frac{1}{\gamma}} - (\hat{Z}_t^{z,\beta+\varepsilon})^{-\frac{1}{\gamma}} \big| \Big] dt \\
&= \frac{\varepsilon}{\gamma} \int_0^\infty e^{-rt} \mathbb{E}^\mathbb{Q}\Big[ (\hat{Z}_t^{z,\beta_\varepsilon})^{-\frac{1}{\gamma}} \Big| \int_0^t e^{-as} \frac{\hat{\beta}_s^{\beta_\varepsilon}}{\sigma^2} ds - \frac{1}{\sigma} \int_0^t e^{-as} dW_s^{\mathbb{Q}} \Big| \Big] dt,
\end{align*}
where $\beta_\varepsilon \in (\beta, \beta+\varepsilon)$. \\
Let \[1<p<\min\bigg\{\gamma, \frac{\gamma\sigma\bigg(\kappa-\frac{ \sigma_\beta}{\sigma}\bigg)}{\sigma_\beta}\bigg\}\] and $q>1$ with $\frac{1}{p}+\frac{1}{q}=1$. H\"older's inequality and (\ref{E(Z)ineq}) then yield
\begin{align}\label{lipbetaeq}
\big|v(z,\beta+\varepsilon) - v(z,\beta)\big| 
&\le \frac{\varepsilon}{\gamma} \int_0^\infty e^{-rt} \mathbb{E}^\mathbb{Q}[(\hat{Z}_t^{z,\beta_\varepsilon})^{-\frac{p}{\gamma}}]^{1/p} 
\mathbb{E}^\mathbb{Q}\Big[\Big| \int_0^t e^{-as} \frac{\hat{\beta}_s^{\beta_\varepsilon}}{\sigma^2} ds - \frac{1}{\sigma} \int_0^t e^{-as} dW_s^{\mathbb{Q}} \Big|^q \Big]^{1/q} dt \notag\\
&\le \frac{\varepsilon}{\gamma} z^{-\frac{1}{\gamma}} \int_0^\infty e^{(-\frac{1}{\gamma}(\delta-r) - r)t} \mathbb{E}^\mathbb{Q}\Big[\Big| \int_0^t e^{-as} \frac{\hat{\beta}_s^{\beta_\varepsilon}}{\sigma^2} ds - \frac{1}{\sigma} \int_0^t e^{-as} dW_s^{\mathbb{Q}} \Big|^q \Big]^{1/q} dt.
\end{align}
\indent We define $\phi := \frac{1}{\gamma} (\delta-r) + r > 0$ and $A^\varepsilon_t := \int_0^t e^{-as} \frac{\hat{\beta}_s^{\beta_\varepsilon}}{\sigma^2} ds - \frac{1}{\sigma} \int_0^t e^{-as} dW_s^{\mathbb{Q}}$, so that the inequality (\ref{lipbetaeq}) then reads as
\begin{equation}\label{A_tineq}\big|v(z,\beta+\varepsilon) - v(z,\beta)\big|\leq\frac{\varepsilon}{\gamma} z^{-\frac{1}{\gamma}} \int_0^\infty e^{-\phi t} \norm{A^\varepsilon_t}_{L^q} dt,\end{equation} where we have set $\norm{\cdot}_{L^q}=\mathbb{E}[|\cdot|^q]^{\frac{1}{q}}$.
We now decompose $A^\varepsilon_t$ into a deterministic part $G^\varepsilon_t$ and a stochastic part $B_t$. That is, \begin{equation}\label{defA}A^\varepsilon_t=G^\varepsilon_t+B_t,\end{equation} with
\begin{align}
G^\varepsilon_t &:= \frac{1}{\sigma^2} \int_0^t e^{-as} (\beta_\varepsilon e^{-as} + b(1 - e^{-as})) ds, \label{deterministic}\\ 
B_t &:= -\frac{\sigma_\beta}{\sigma^2} \int_0^t e^{-as} \Big( \int_0^s e^{-a(s-u)} dW_u^{\mathbb{Q}} \Big) ds - \frac{1}{\sigma} \int_0^t e^{-as} dW_s^{\mathbb{Q}}.\label{stochastic}
\end{align}
A direct calculation yields \[G^\varepsilon_t=\frac{1}{\sigma^2}\bigg(\frac{\beta_{\varepsilon}-b}{2a}(1-e^{-2at})+\frac{b}{a}(1-e^{-at})\bigg),\] which implies
\begin{equation}\label{D_tbound}
|G^\varepsilon_t| \leq \frac{1}{\sigma^2a}\bigg(\frac{|\beta_{\varepsilon}-b|}{2}+|b|\bigg)\leq \frac{1}{\sigma^2a}\left(\frac{|\beta|+1+|b|}{2}+|b|\right)=:C_G(\beta),
\end{equation} with $C_G(\beta)$ independent of $\varepsilon>0$. For $B_t$, an application of Stochastic Fubini Theorem (cf., e.g., Theorem $65$ in section IV of \cite{protter2012stochastic}) leads to
\begin{align}
B_t &= -\frac{\sigma_\beta}{\sigma^2} \int_0^t \left( \int_u^t e^{-as} e^{-a(s-u)} ds \right) dW_u^{\mathbb{Q}} - \frac{1}{\sigma} \int_0^t e^{-au} dW_u^{\mathbb{Q}} \notag\\
&= -\frac{\sigma_\beta}{2a\sigma^2} \int_0^t \left( e^{-au} - e^{-2at + au} \right) dW_u^{\mathbb{Q}} - \frac{1}{\sigma} \int_0^t e^{-au} dW_u^{\mathbb{Q}} \label{B}\\
&= \int_0^t \underbrace{\left( -\frac{\sigma_\beta }{2a\sigma^2} (e^{-au} - e^{-2at + au}) - \frac{1}{\sigma} e^{-au} \right)}_{=: f(u,t)} dW_u^{\mathbb{Q}}\notag.
\end{align}
The deterministic function $f(u,t)$ in (\ref{B}) can be bounded as follows
\[
|f(u,t)| \le e^{-au} \left( \frac{\sigma_\beta }{2a\sigma^2} + \frac{1}{\sigma} \right) =: C_f e^{-au},\]
and hence, we obtain
\begin{align}\label{Var(S_t)}
\mathrm{Var}(B_t) &= \int_0^t f(u,t)^2 du \le \frac{C_f^2 }{2a} =: K.
\end{align}
Upon noticing that $B_t$ is Gaussian with mean $0$ and variance $\mathrm{Var}(B_t) \le K$, we may write $B_t = \sqrt{\mathrm{Var}(B_t)} Y$, where $Y \sim \mathcal{N}(0,1)$. Therefore, using (\ref{Var(S_t)}), we obtain
\begin{equation}\label{S_tbound}
\|B_t\|_{L^q} = \sqrt{\mathrm{Var}(B_t)} \|Y\|_{L^q} \le \sqrt{K} \|Y\|_{L^q} =: C_q.
\end{equation}
The fact that $Y\sim \mathcal{N}(0,1)$ implies $\|Y\|_{L^q} < \infty$ for all $q \ge 1$. Hence, (\ref{D_tbound}) and (\ref{S_tbound}) imply
\begin{equation}\label{A_tbound}
\|A^\varepsilon_t\|_{L^q} \le \|G^\varepsilon_t\|_{L^q} + \|B_t\|_{L^q} = |G^\varepsilon_t| + \|B_t\|_{L^q} \le C_G(\beta) + C_q =: M_q(\beta),
\end{equation}
where $M_q(\beta)$ is independent of $\varepsilon>0$. Inequalities (\ref{A_tineq}) and (\ref{A_tbound}) in turn lead to
\begin{align*}
\big|v(z,\beta+\varepsilon)-v(z,\beta)\big|
\le \frac{\varepsilon}{\gamma} z^{-\frac{1}{\gamma}} \int_0^\infty e^{-\phi t} \norm{A^\varepsilon_t}_{L^q} dt &\le \frac{\varepsilon}{\gamma} z^{-\frac{1}{\gamma}} M_q(\beta) \int_0^\infty e^{-\phi t} dt \\
&=\frac{M_q(\beta)}{\phi \gamma} z^{-\frac{1}{\gamma}} \varepsilon =: c(z,\beta) \varepsilon,
\end{align*}
where we have $0 < c(z,\beta) < \infty$. A symmetric argument also yields
\[
\big|v(z,\beta) - v(z,\beta-\varepsilon)\big| \le c(z,\beta) \varepsilon.
\]
Therefore, the value function $v(z,\cdot)$ is (locally) Lipschitz continuous in $\beta$, with a constant that can be taken uniform over compact sets.\\
\indent Next, we derive the probabilistic representation for the weak derivative $v_\beta$, which exists for almost every $\beta\in\mathbb{R}$. Suppose that $\beta$ is a point of differentiability and denote by $\tau^*$ the optimal 
stopping time for the problem with initial data $(z,\beta)$ (independent of $\varepsilon$). 
Since $\tau^*$ is suboptimal for $v(z,\beta+\varepsilon)$, we obtain
\begin{align*}
v(z,\beta+\varepsilon)-v(z,\beta)
\le \mathbb{E}^\mathbb{Q}\!\left[\int_{0}^{\tau^*} 
e^{-rt}\Big( (\hat Z_t^{z,\beta})^{-\frac{1}\gamma} - (\hat Z_t^{z,\beta+\varepsilon})^{-\frac{1}\gamma} \Big)\,dt\right].
\end{align*}
Dividing by $\varepsilon$ and using the Mean Value Theorem yields
\begin{equation}\label{ineqbetaprobrepr}
\frac{v(z,\beta+\varepsilon)-v(z,\beta)}{\varepsilon}
\le 
\mathbb{E}^\mathbb{Q}\!\left[\int_{0}^{\tau^*} e^{-rt}
\frac{1}{\gamma}(\hat Z_t^{z,\beta_\varepsilon})^{-\frac{1}\gamma}
\left(\int_0^t e^{-as}\frac{\hat\beta_s^{\beta_\varepsilon}}{\sigma^2}\,ds
      -\frac{1}{\sigma}\int_0^t e^{-as}\,dW_s^{\mathbb{Q}}\right)dt\right],
\end{equation}
where $\beta_\varepsilon\in(\beta,\beta+\varepsilon)$.  
Since \[\mathbb{E}^\mathbb{Q}\!\left[\int_{0}^{\infty} e^{-rt}
\frac{1}{\gamma}(\hat Z_t^{z,\beta_\varepsilon})^{-\frac{1}\gamma}
\left|\int_0^t e^{-as}\frac{\hat\beta_s^{\beta_\varepsilon}}{\sigma^2}\,ds
      -\frac{1}{\sigma}\int_0^t e^{-as}\,dW_s^{\mathbb{Q}}\right|dt\right]<\infty,\]which follows from the same arguments used in the proof of (local) Lipschitz continuity of $v$ in the $\beta$-variable, using Fubini's Theorem, the right-hand side of (\ref{ineqbetaprobrepr}) becomes 
\begin{equation}\label{integralbetaprobrepr}
 \int_0^\infty \mathbb{E}^\mathbb{Q}\!\left[
e^{-rt}\left(\frac{1}{\gamma}(\hat Z_t^{z,\beta_\varepsilon})^{-\frac{1}\gamma}
\left(\int_0^t e^{-as}\frac{\hat{\beta}_s^{\beta_\varepsilon}}{\sigma^2}ds
     -\frac{1}{\sigma}\int_0^t e^{-as}dW_s^{\mathbb{Q}}\right)\right)
\mathbf{1}_{\{t\le \tau^*\}}\right]dt . 
\end{equation}
\indent As before, we may restrict to $\varepsilon\le \varepsilon_0$ for some $\varepsilon_0\in (0,1)$. Again, exploiting the arguments from the proof of the Lipschitz continuity of $v$ in the $\beta$-variable, we can bound
\[
\left|\mathbb{E}^\mathbb{Q}\!\left[
e^{-rt}\left(\frac{1}{\gamma}(\hat Z_t^{z,\beta_\varepsilon})^{-\frac{1}\gamma}
\left(\int_0^t e^{-as}\frac{\hat{\beta}_s^{\beta_\varepsilon}}{\sigma^2}ds
     -\frac{1}{\sigma}\int_0^t e^{-as}dW_s^{\mathbb{Q}}\right)\right)
\mathbf{1}_{\{t\le \tau^*\}}\right]\right|
\]
by a function that is independent of $\varepsilon$ and Lebesgue-integrable over $(0,\infty)$. Therefore, Dominated 
Convergence Theorem implies \begin{align*}&\lim_{\varepsilon\to 0}\int_0^\infty \mathbb{E}^\mathbb{Q}\!\left[
e^{-rt}\left(\frac{1}{\gamma}(\hat Z_t^{z,\beta_\varepsilon})^{-\frac{1}\gamma}
\left(\int_0^t e^{-as}\frac{\hat{\beta}_s^{\beta_\varepsilon}}{\sigma^2}ds
     -\frac{1}{\sigma}\int_0^t e^{-as}dW_s^{\mathbb{Q}}\right)\right)
\mathbf{1}_{\{t\le \tau^*\}}\right]dt\\
&=\int_0^\infty \lim_{\varepsilon\to 0}\mathbb{E}^\mathbb{Q}\!\left[
e^{-rt}\left(\frac{1}{\gamma}(\hat Z_t^{z,\beta_\varepsilon})^{-\frac{1}\gamma}
\left(\int_0^t e^{-as}\frac{\hat{\beta}_s^{\beta_\varepsilon}}{\sigma^2}ds
     -\frac{1}{\sigma}\int_0^t e^{-as}dW_s^{\mathbb{Q}}\right)\right)
\mathbf{1}_{\{t\le \tau^*\}}\right]dt.\end{align*}
\indent Next, we show that the family \begin{equation}\label{UIfamily}\bigg\lbrace\left|(\hat Z_t^{z,\beta_\varepsilon})^{-1/\gamma}
\left(\int_0^t e^{-as}\frac{\hat{\beta}_s^{\beta_\varepsilon}}{\sigma^2}ds
      -\frac{1}{\sigma}\int_0^t e^{-as}dW_s^{\mathbb{Q}}\right)
\right| \mathbf{1}_{\{t\le \tau^*\}}\bigg\rbrace_{\varepsilon\in(0,\varepsilon_0)}\end{equation} is uniformly integrable. We let  \[1<\tilde{p}<\min\bigg\{\gamma, \frac{\gamma\sigma\bigg(\kappa-\frac{ \sigma_\beta}{\sigma}\bigg)}{\sigma_\beta}\bigg\},\] $m\in(1,\tilde p)$, and we define 
$p:=\frac{\tilde p}{m}>1$ and $q>1$ with $\frac{1}{p}+\frac{1}{q}=1$. 
Hölder's inequality then yields
\begin{align}\label{beta_lipschitzbound}
&\mathbb{E}^\mathbb{Q}\!\left[(\hat Z_t^{z,\beta_\varepsilon})^{-\frac{m}{\gamma}}
\left|\left(\int_0^t e^{-as}\frac{\hat{\beta}_s^{\beta_\varepsilon}}{\sigma^2}ds
           -\frac{1}{\sigma}\int_0^t e^{-as}dW_s^{\mathbb{Q}}\right)\right|^m\right]\notag \\
&\qquad\le 
\mathbb{E}^\mathbb{Q}\!\left[(\hat Z_t^{z,\beta_\varepsilon})^{-\frac{\tilde p}{\gamma}}\right]^\frac{1}{p}
\,
\mathbb{E}^\mathbb{Q}\!\left[\left|
\int_0^t e^{-as}\frac{\hat{\beta}_s^{\beta_\varepsilon}}{\sigma^2}ds
-\frac{1}{\sigma}\int_0^t e^{-as}dW_s^{\mathbb{Q}}\right|^{mq}\right]^\frac{1}{q}.
\end{align}
Now, recalling (\ref{defA}) and using the arguments used for the derivation of (\ref{A_tbound}), we find 
\begin{equation}\label{A_t^sbound}
\|A^\varepsilon_t\|_{L^{mq}}\le M_{mq}(\beta),
\end{equation} since $mq>1$, which is uniform in $\varepsilon$. Hence, combining (\ref{beta_lipschitzbound}), (\ref{A_t^sbound}) and (\ref{E(Z)ineq}) gives
\[
\sup_{\varepsilon\le \varepsilon_0}
\mathbb{E}^\mathbb{Q}\!\left[
\left|(\hat Z_t^{z,\beta_\varepsilon})^{-\frac{1}{\gamma}}
\left(\int_0^t e^{-as}\frac{\hat{\beta}_s^{\beta_\varepsilon}}{\sigma^2}ds
      -\frac{1}{\sigma}\int_0^t e^{-as}dW_s^{\mathbb{Q}}\right)
\right|^m \mathbf{1}_{\{t\le \tau^*\}}\right]
\leq 
z^{-\frac{m}{\gamma}}\,e^{-\frac{m}{\gamma}(\delta-r)t} M_{mq}(\beta)^m,
\]
and thus the uniform integrability of the family (\ref{UIfamily}) follows. Vitali’s Convergence Theorem, together with the continuities of $\beta\mapsto \hat{Z}^{z,\beta}$ and $\beta\mapsto \hat{\beta}^{\beta}$, then imply
\begin{align}\label{final}
&\lim_{\varepsilon\to0}\mathbb{E}^\mathbb{Q}\!\left[\int_0^{\tau^*} e^{-rt}
\left(\frac{1}{\gamma}(\hat Z_t^{z,\beta_\varepsilon})^{-\frac{1}{\gamma}}
\left(\int_0^t e^{-as}\frac{\hat{\beta}_s^{\beta_\varepsilon}}{\sigma^2}ds
     -\frac{1}{\sigma}\int_0^t e^{-as}dW_s^{\mathbb{Q}}\right)\right)dt\right] \notag\\
&\quad=
\mathbb{E}^\mathbb{Q}\!\left[\int_0^{\tau^*} e^{-rt}
\left(\frac{1}{\gamma}(\hat Z_t^{z,\beta})^{-\frac{1}{\gamma}}
\left(\int_0^t e^{-as}\frac{\hat{\beta}_s^{\beta}}{\sigma^2}ds
     -\frac{1}{\sigma}\int_0^t e^{-as}dW_s^{\mathbb{Q}}\right)\right)dt\right].
\end{align}
Using (\ref{final}) into (\ref{ineqbetaprobrepr}) gives
\begin{align}\label{limsupbeta}
&\limsup_{\varepsilon\to0}
\frac{v(z,\beta+\varepsilon)-v(z,\beta)}{\varepsilon}\notag\\ 
&\leq \mathbb{E}_{z,\beta}^\mathbb{Q}\left[\int_0^{\tau^*} e^{-rt}
\left(\frac{1}{\gamma}(\hat Z_t^{z,\beta})^{-\frac{1}{\gamma}}
\left(\int_0^t e^{-as}\frac{\hat{\beta}_s^{\beta}}{\sigma^2}ds
     -\frac{1}{\sigma}\int_0^t e^{-as}dW_s^{\mathbb{Q}}\right)\right)dt\right],
\end{align}
and, arguing symmetrically, also
\begin{align}\label{liminfbeta}
&\liminf_{\varepsilon\to0}
\frac{v(z,\beta)-v(z,\beta-\varepsilon)}{\varepsilon}
\notag\\
&\geq 
\mathbb{E}^\mathbb{Q}_{z,\beta}\left[\int_0^{\tau^*} e^{-rt}
\left(\frac{1}{\gamma}(\hat Z_t^{z,\beta})^{-\frac{1}{\gamma}}
\left(\int_0^t e^{-as}\frac{\hat{\beta}_s^{\beta}}{\sigma^2}ds
     -\frac{1}{\sigma}\int_0^t e^{-as}dW_s^{\mathbb{Q}}\right)\right)dt\right].
\end{align}
Hence, since $v$ is (locally) Lipschitz continuous in $\beta$, it follows from (\ref{limsupbeta}) and (\ref{liminfbeta}) that for almost every $\beta\in\mathbb{R}$ the weak derivative $v_\beta$ is given by
\begin{equation}\label{vbetaproof}
v_\beta(z,\beta)=\mathbb{E}^\mathbb{Q}_{z,\beta}\left[\int_0^{\tau^*} e^{-rt}
\left(\frac{1}{\gamma}(\hat Z_t^{z,\beta})^{-\frac{1}{\gamma}}
\left(\int_0^t e^{-as}\frac{\hat{\beta}_s^{\beta}}{\sigma^2}ds
     -\frac{1}{\sigma}\int_0^t e^{-as}dW_s^{\mathbb{Q}}\right)\right)dt\right].
\end{equation} 
\end{proof}
\bibliographystyle{siam}
\bibliography{references}

@article{fama1988dividend,
  title={Dividend yields and expected stock returns},
  author={Fama, Eugene F and French, Kenneth R},
  journal={Journal of Financial Economics},
  volume={22},
  number={1},
  pages={3--25},
  year={1988},
  publisher={Elsevier}
}

@article{LambertonTerenzi2,
  title={Properties of the American price function in the Heston-type models},
  author={Lamberton, Damien and Terenzi, Giulia},
  journal={arXiv preprint arXiv:1904.01653},
  year={2019}
}

@article{Jacka-Snell,
  title={Local times, optimal stopping and semimartingales},
  author={Jacka, SD},
  journal={The Annals of Probability},
  pages={329--339},
  year={1993},
  publisher={JSTOR}
}

@article{Bally,
  title={An elementary introduction to Malliavin calculus},
  author={Bally, Vlad},
  journal={Lecture Notes},
  year={2003},
  school={INRIA}
}

@article{poterba1988mean,
  title={Mean reversion in stock prices: Evidence and implications},
  author={Poterba, James M and Summers, Lawrence H},
  journal={Journal of Financial Economics},
  volume={22},
  number={1},
  pages={27--59},
  year={1988},
  publisher={Elsevier}
}

@article{jeon2025finite,
  title={The Finite-Horizon Retirement Problem with Borrowing Constraint: A Zero-Sum Stopper vs. Singular-Controller Game},
  author={Jeon, Junkee and Kim, Takwon and Yang, Zhou},
  journal={SSRN Electronic Journal},
  year={2025},
  note={Available at SSRN},
  doi={10.2139/ssrn.4364441}
}

@article{el1998optimization,
  title={Optimization of consumption with labor income},
  author={El Karoui, Nicole and Jeanblanc-Picqu{\'e}, Monique},
  journal={Finance and Stochastics},
  volume={2},
  number={4},
  pages={409--440},
  year={1998},
  publisher={Springer}
}

@article {de2020global,
    AUTHOR = {De Angelis, T. and Peskir, G.},
     TITLE = {Global {$C^1$} regularity of the value function in optimal
              stopping problems},
   JOURNAL = {The Annals of Applied Probability},
  FJOURNAL = {The Annals of Applied Probability},
    VOLUME = {30},
      YEAR = {2020},
    NUMBER = {3},
     PAGES = {1007--1031},
      ISSN = {1050-5164,2168-8737},
   MRCLASS = {60G40 (35K10 35R35 35R60 60H30 60J25)},
  MRNUMBER = {4133366},
       DOI = {10.1214/19-AAP1517},
       URL = {https://doi.org/10.1214/19-AAP1517},
}

@article{baldursson1996irreversible,
  title={Irreversible investment and industry equilibrium},
  author={Baldursson, Fridrik M and Karatzas, Ioannis},
  journal={Finance and Stochastics},
  volume={1},
  number={1},
  pages={69--89},
  year={1996},
  publisher={Springer}
}

@article{kim1996dynamic,
  title={Dynamic nonmyopic portfolio behavior},
  author={Kim, Tong Suk and Omberg, Edward},
  journal={The Review of Financial Studies},
  volume={9},
  number={1},
  pages={141--161},
  year={1996},
  publisher={Oxford University Press}
}

@book{peskir2006optimal,
  title={Optimal stopping and free-boundary problems},
  author={Peskir, Goran and Shiryaev, Albert},
  year={2006},
  publisher={Springer}
}

@article {peskir2025weak,
    AUTHOR = {Peskir, Goran},
     TITLE = {Weak solutions in the sense of {S}chwartz to {D}ynkin's
              characteristic operator equation},
   JOURNAL = {Potential Analysis},
  FJOURNAL = {Potential Analysis. An International Journal Devoted to the
              Interactions between Potential Theory, Probability Theory,
              Geometry and Functional Analysis},
    VOLUME = {63},
      YEAR = {2025},
    NUMBER = {4},
     PAGES = {1887--1905},
      ISSN = {0926-2601,1572-929X},
   MRCLASS = {60J25 (35K20 35K65 47D07 60J35 60J60)},
  MRNUMBER = {4990506},
       DOI = {10.1007/s11118-025-10225-0},
       URL = {https://doi.org/10.1007/s11118-025-10225-0},
}

@article{he1993labor,
  title={Labor income, borrowing constraints, and equilibrium asset prices},
  author={He, Hua and Pages, Henri F},
  journal={Economic Theory},
  volume={3},
  number={4},
  pages={663--696},
  year={1993},
  publisher={Springer}
}

@article{gutekunst2025optimal,
  title={Optimal Investment and Consumption in a Stochastic Factor Model},
  author={Gutekunst, Florian and Herdegen, Martin and Hobson, David},
  journal={arXiv preprint arXiv:2509.09452},
  year={2025}
}

@article{hata2018optimal,
  title={An optimal consumption problem for general factor models},
  author={Hata, Hiroaki and Nagai, Hideo and Sheu, Shuenn-Jyi},
  journal={SIAM Journal on Control and Optimization},
  volume={56},
  number={5},
  pages={3149--3183},
  year={2018},
  publisher={SIAM}
}

@article{hata2025optimal,
  title={Optimal Consumption and Investment Problem Using a Power Utility Function under a General Nonlinear Stochastic Factor Model},
  author={Hata, Hiroaki},
  journal={SIAM Journal on Control and Optimization},
  volume={63},
  number={5},
  pages={3588--3617},
  year={2025},
  publisher={SIAM}
}

@article{munk2004optimal,
  title={Optimal consumption and investment strategies with stochastic interest rates},
  author={Munk, Claus and S{\o}rensen, Carsten},
  journal={Journal of Banking \& Finance},
  volume={28},
  number={8},
  pages={1987--2013},
  year={2004},
  publisher={Elsevier}
}

@article{mehra1985equity,
  title={The equity premium: A puzzle},
  author={Mehra, Rajnish and Prescott, Edward C},
  journal={Journal of Monetary Economics},
  volume={15},
  number={2},
  pages={145--161},
  year={1985},
  publisher={Elsevier}
}

@article{campbell1999force,
  title={By force of habit: A consumption-based explanation of aggregate stock market behavior},
  author={Campbell, John Y and Cochrane, John H},
  journal={Journal of Political Economy},
  volume={107},
  number={2},
  pages={205--251},
  year={1999},
  publisher={The University of Chicago Press}
}

@book{karatzas2014brownian,
  title={Brownian motion and stochastic calculus},
  author={Karatzas, Ioannis and Shreve, Steven},
  year={2014},
  publisher={Springer}
}

@book {protter2012stochastic,
    AUTHOR = {Protter, Philip E.},
     TITLE = {Stochastic integration and differential equations},
    SERIES = {Stochastic Modelling and Applied Probability},
    VOLUME = {21},
   EDITION = {Second},
      NOTE = {Corrected third printing},
 PUBLISHER = {Springer-Verlag, Berlin},
      YEAR = {2005},
     PAGES = {xiv+419},
      ISBN = {3-540-00313-4},
   MRCLASS = {60-02 (60G44 60H05 60H10 60H20)},
  MRNUMBER = {2273672},
MRREVIEWER = {Evelyn\ Buckwar},
       DOI = {10.1007/978-3-662-10061-5},
       URL = {https://doi.org/10.1007/978-3-662-10061-5},
}

@article{callegaro2020optimal,
  title={Optimal reduction of public debt under partial observation of the economic growth},
  author={Callegaro, Giorgia and Ceci, Claudia and Ferrari, Giorgio},
  journal={Finance and Stochastics},
  volume={24},
  number={4},
  pages={1083--1132},
  year={2020},
  publisher={Springer}
}

@book{nualart2006malliavin,
  title={The Malliavin calculus and related topics},
  author={Nualart, David},
  year={2006},
  publisher={Springer}
}

@article{ernst2024quickest,
  title={Quickest real-time detection of multiple Brownian drifts},
  author={Ernst, Philip A and Mei, Hongwei and Peskir, Goran},
  journal={SIAM Journal on Control and Optimization},
  volume={62},
  number={3},
  pages={1832--1856},
  year={2024},
  publisher={SIAM}
}

@article{de2020optimal,
  title={Optimal dividends with partial information and stopping of a degenerate reflecting diffusion},
  author={De Angelis, Tiziano},
  journal={Finance and Stochastics},
  volume={24},
  number={1},
  pages={71--123},
  year={2020},
  publisher={Springer}
}

@article{Touzi-American,
  title={American options exercise boundary when the volatility changes randomly},
  author={Touzi, Nizar},
  journal={Applied Mathematics and Optimization},
  volume={39},
  number={3},
  pages={411--422},
  year={1999},
  publisher={Springer}
}

@article {LambertonTerenzi,
    AUTHOR = {Lamberton, Damien and Terenzi, Giulia},
     TITLE = {Variational formulation of American option prices in the Heston model},
   JOURNAL = {SIAM Journal on Financial Mathematics},
  FJOURNAL = {SIAM Journal on Financial Mathematics},
    VOLUME = {10},
      YEAR = {2019},
    NUMBER = {1},
     PAGES = {261--308},
      ISSN = {1945-497X},
   MRCLASS = {91G20 (60H30 60J60 91G80)},
  MRNUMBER = {3928342},
MRREVIEWER = {Wasim\ Ul-Haq},
       DOI = {10.1137/17M1158872},
       URL = {https://doi.org/10.1137/17M1158872},
}

@article{ferrari2018optimal,
  title={On the optimal management of public debt: A singular stochastic control problem},
  author={Ferrari, Giorgio},
  journal={SIAM Journal on Control and Optimization},
  volume={56},
  number={3},
  pages={2036--2073},
  year={2018},
  publisher={SIAM}
}

@article{de2017optimal,
  title={Optimal boundary surface for irreversible investment with stochastic costs},
  author={De Angelis, Tiziano and Federico, Salvatore and Ferrari, Giorgio},
  journal={Mathematics of Operations Research},
  volume={42},
  number={4},
  pages={1135--1161},
  year={2017},
  publisher={INFORMS}
}

@article {MR3904414,
    AUTHOR = {De Angelis, Tiziano and Stabile, Gabriele},
     TITLE = {On {L}ipschitz continuous optimal stopping boundaries},
   JOURNAL = {SIAM Journal on Control and Optimization},
  FJOURNAL = {SIAM Journal on Control and Optimization},
    VOLUME = {57},
      YEAR = {2019},
    NUMBER = {1},
     PAGES = {402--436},
      ISSN = {0363-0129,1095-7138},
   MRCLASS = {60G40 (35R35)},
  MRNUMBER = {3904414},
MRREVIEWER = {Zuoquan\ Xu},
       DOI = {10.1137/17M1113709},
       URL = {https://doi.org/10.1137/17M1113709},
}

@article {bandini,
    AUTHOR = {Bandini, Elena and De Angelis, Tiziano and Ferrari, Giorgio
              and Gozzi, Fausto},
     TITLE = {Optimal dividend payout under stochastic discounting},
   JOURNAL = {Mathematical Finance. An International Journal of Mathematics,
              Statistics and Financial Economics},
  FJOURNAL = {Mathematical Finance. An International Journal of Mathematics,
              Statistics and Financial Economics},
    VOLUME = {32},
      YEAR = {2022},
    NUMBER = {2},
     PAGES = {627--677},
      ISSN = {0960-1627,1467-9965},
   MRCLASS = {91G50 (35R35 60G40 93E20)},
  MRNUMBER = {4398652},
MRREVIEWER = {John\ P.\ Lehoczky},
       DOI = {10.1111/mafi.12339},
       URL = {https://doi.org/10.1111/mafi.12339},
}

@article {Merton,
    AUTHOR = {Merton, Robert C.},
     TITLE = {Optimum consumption and portfolio rules in a continuous-time
              model},
   JOURNAL = {Journal of Economic Theory},
  FJOURNAL = {Journal of Economic Theory},
    VOLUME = {3},
      YEAR = {1971},
    NUMBER = {4},
     PAGES = {373--413},
      ISSN = {0022-0531,1095-7235},
   MRCLASS = {90A15 (90A10)},
  MRNUMBER = {456373},
       DOI = {10.1016/0022-0531(71)90038-X},
       URL = {https://doi.org/10.1016/0022-0531(71)90038-X},
}

@article{merton1969lifetime,
  title={Lifetime portfolio selection under uncertainty: The continuous-time case},
  author={Merton, Robert C},
  journal={The Review of Economics and Statistics},
  pages={247--257},
  year={1969},
  publisher={JSTOR}
}

@article{wachter2002portfolio,
  title={Portfolio and consumption decisions under mean-reverting returns: An exact solution for complete markets},
  author={Wachter, Jessica A},
  journal={Journal of Financial and Quantitative Analysis},
  volume={37},
  number={1},
  pages={63--91},
  year={2002},
  publisher={Cambridge University Press}
}

@article {quickest,
    AUTHOR = {Johnson, Peter and Peskir, Goran},
     TITLE = {Quickest detection problems for {B}essel processes},
   JOURNAL = {The Annals of Applied Probability},
  FJOURNAL = {The Annals of Applied Probability},
    VOLUME = {27},
      YEAR = {2017},
    NUMBER = {2},
     PAGES = {1003--1056},
      ISSN = {1050-5164,2168-8737},
   MRCLASS = {60G40 (35K67 45G10 60H30 60J60 62C10)},
  MRNUMBER = {3655860},
MRREVIEWER = {Robert\ C.\ Dalang},
       DOI = {10.1214/16-AAP1223},
       URL = {https://doi.org/10.1214/16-AAP1223},
}

@article {christensen,
    AUTHOR = {Christensen, S\"oren and Crocce, Fabi\'an and Mordecki,
              Ernesto and Salminen, Paavo},
     TITLE = {On optimal stopping of multidimensional diffusions},
   JOURNAL = {Stochastic Processes and their Applications},
  FJOURNAL = {Stochastic Processes and their Applications},
    VOLUME = {129},
      YEAR = {2019},
    NUMBER = {7},
     PAGES = {2561--2581},
      ISSN = {0304-4149,1879-209X},
   MRCLASS = {60G40 (60J60 62L15)},
  MRNUMBER = {3958442},
       DOI = {10.1016/j.spa.2018.07.014},
       URL = {https://doi.org/10.1016/j.spa.2018.07.014},
}

@article {frey,
    AUTHOR = {Frey, R\"udiger},
     TITLE = {Superreplication in stochastic volatility models and optimal
              stopping},
   JOURNAL = {Finance and Stochastics},
  FJOURNAL = {Finance and Stochastics},
    VOLUME = {4},
      YEAR = {2000},
    NUMBER = {2},
     PAGES = {161--187},
      ISSN = {0949-2984,1432-1122},
   MRCLASS = {91B28 (60G40 60H99)},
  MRNUMBER = {1780325},
MRREVIEWER = {Jak\v sa\ Cvitani\'c},
       DOI = {10.1007/s007800050010},
       URL = {https://doi.org/10.1007/s007800050010},
}

@article {jacka,
    AUTHOR = {Assing, Sigurd and Jacka, Saul and Ocejo, Adriana},
     TITLE = {Monotonicity of the value function for a two-dimensional
              optimal stopping problem},
   JOURNAL = {The Annals of Applied Probability},
  FJOURNAL = {The Annals of Applied Probability},
    VOLUME = {24},
      YEAR = {2014},
    NUMBER = {4},
     PAGES = {1554--1584},
      ISSN = {1050-5164,2168-8737},
   MRCLASS = {60G40 (91G20 93E20)},
  MRNUMBER = {3211004},
MRREVIEWER = {\L.\ Stettner},
       DOI = {10.1214/13-AAP956},
       URL = {https://doi.org/10.1214/13-AAP956},
}

@article{guasoni2025variational,
  title={A Variational Approach to Portfolio Choice},
  author={Guasoni, Paolo and Lawless, Emmet and Tai, Ho Man},
  journal={Available at SSRN 5669613},
  year={2025}
}

@article{barberis2000investing,
  title={Investing for the long run when returns are predictable},
  author={Barberis, Nicholas},
  journal={The Journal of Finance},
  volume={55},
  number={1},
  pages={225--264},
  year={2000},
  publisher={Wiley Online Library}
}

@book {oksendal,
    AUTHOR = {\O ksendal, Bernt},
     TITLE = {Stochastic differential equations},
    SERIES = {Universitext},
   EDITION = {Sixth},
      NOTE = {An introduction with applications},
 PUBLISHER = {Springer-Verlag, Berlin},
      YEAR = {2003},
     PAGES = {xxiv+360},
      ISBN = {3-540-04758-1},
   MRCLASS = {60H10 (60G44 60J60)},
  MRNUMBER = {2001996},
       DOI = {10.1007/978-3-642-14394-6},
       URL = {https://doi.org/10.1007/978-3-642-14394-6},
}

\end{document}